\RequirePackage{snapshot}
\newif\ifproofs\proofsfalse\ifproofs\RequirePackage[displaymath,mathlines]{lineno}\fi

\newif\ifsubsections
\subsectionstrue 

\documentclass[]{pcmi}

\usepackage[sc]{mathpazo}          
\usepackage{eulervm}               
\usepackage[mathscr]{eucal}		   
\usepackage[scaled=0.86]{berasans} 
\usepackage[scaled=1]{inconsolata} 
\usepackage[T1]{fontenc}

\usepackage[%
	protrusion=true,
	expansion=false,
	auto=false
	]{microtype}

\usepackage{xcolor}
\definecolor{cmykred}{cmyk}{0,1,1,0}
\usepackage{graphicx}
\graphicspath{{./figures/}}

\ifdraft
    \definecolor{linkred}{rgb}{0.7,0.2,0.2}
    \definecolor{linkblue}{rgb}{0,0.2,0.6}
\else
    \definecolor{linkred}{cmyk}{0.0,0.0,0.0,1.0}
    \definecolor{linkblue}{cmyk}{0,0.0,0.0,1.0}
\fi

\usepackage{mathtools}
\usepackage{stmaryrd}
\usepackage{bbm}
\usepackage{booktabs}
\usepackage{verbatim}
\usepackage{enumerate}
\usepackage{subcaption}  
\usepackage[all]{xy}
\CompileMatrices

\usepackage{tikz}
\usetikzlibrary{arrows,calc,positioning,decorations.pathreplacing}
\usepackage{amssymb}

\PassOptionsToPackage{hyphens}{url} 
\usepackage[
    setpagesize=false,
    pagebackref,
	pdfpagelabels,
	plainpages=false,
    pdfstartview={FitH 1000},
    bookmarksnumbered=false,
    linktoc=all,
    colorlinks=true,
    anchorcolor=black,
    menucolor=black,
    runcolor=black,
    filecolor=black,
    linkcolor=linkblue,
	citecolor=linkblue,
	urlcolor=linkred,
]{hyperref}
\usepackage[backrefs,msc-links,initials,nobysame]{amsrefs}

\customizeamsrefs 

\theoremstyle{plain}

\newtheorem{theorem}[equation]{Theorem}

\newtheorem{proposition}[equation]{Proposition}
\newtheorem{lemma}[equation]{Lemma}

\newtheorem{corollary}[equation]{Corollary}

\theoremstyle{definition}
\newtheorem{definition}[equation]{Definition}

\newtheorem{example}[equation]{Example}
\newtheorem{exercise}[equation]{Exercise}

\theoremstyle{remark}

\newtheorem{remark}[equation]{Remark}
\newtheorem*{remark*}{Remark}

\newcommand{\cC}{\mathcal C}
\DeclarePairedDelimiter{\set}{\{}{\}}

\newcommand{\lgray}{gray!40!white}
\newcommand{\hC}{h\cC_1^\R}
\newcommand{\ghC}{\hC[(\hC)^{-1}]}
\newcommand{\pt}{\mathrm{pt}}

\newcommand{\boundingrect}[2]{
	\draw (0, 0) -- (0, #1);
	\draw (#2, 0) -- (#2, #1);
}
\newcommand{\boundingwalls}[1]{
	\boundingrect{#1}{#1}
}
\newcommand{\stuff}{\footnotesize\bfseries\textcolor{cmykred}{(?)}}
\newcommand{\textstuff}{{\bfseries\textcolor{cmykred}{(?)}}\!}

\newcommand{\singledisc}[1]{
	\begin{tikzpicture} 
	\boundingwalls{#1}
	\draw (#1/2, #1/2) circle (0.25 * #1);
	\end{tikzpicture}
}

\newcommand{\twodiscs}[1]{
	\begin{tikzpicture} 
	\boundingrect{#1}{2*#1/3}
	\draw (#1/3, #1/4) circle (0.125 * #1);
	\draw (#1/3, 3*#1/4) circle (0.125 * #1);
	\end{tikzpicture}
}

\newcommand{\twoleftsemicircles}[1]{
	\begin{tikzpicture} 
	  \begin{scope}
		\clip (0, 0) rectangle (#1/2, #1);
		\draw (0, #1/4) circle (0.125 * #1);
		\draw (0, 3*#1/4) circle (0.125 * #1);
	  \end{scope}
	\boundingrect{#1}{#1/2}
	\end{tikzpicture}
}

\newcommand{\tworightsemicircles}[1]{
	\begin{tikzpicture} 
	  \begin{scope}
		\clip (0, 0) rectangle (#1/2, #1);
		\draw (#1/2, #1/4) circle (0.125 * #1);
		\draw (#1/2, 3*#1/4) circle (0.125 * #1);
	  \end{scope}
	\boundingrect{#1}{#1/2}
	\end{tikzpicture}
}

\newcommand{\leftmacaroni}[1]{
	\begin{tikzpicture}
	  \begin{scope}
		\clip (0, 0) rectangle (2*#1/3, #1);
		\draw (0, #1/2) circle (0.35 * #1);
		\draw (0, #1/2) circle (0.15 * #1);
	  \end{scope}
	  \boundingrect{#1}{2*#1/3}
	\end{tikzpicture}
}

\newcommand{\wg}[1]{\begin{gathered}\ #1\ \end{gathered}}
\newcommand{\CircNum}[1]{\ooalign{\hfil\raise .00ex\hbox{\scriptsize #1}\hfil\crcr\mathhexbox20D}}

\newcommand{\bC}{\mathbb{C}}
\newcommand{\bD}{\mathbb{D}}

\newcommand{\bN}{\mathbb{N}}

\newcommand{\bP}{\mathbb{P}}

\newcommand{\bR}{\mathbb{R}}

\newcommand{\bZ}{\mathbb{Z}}

\newcommand\Diff{\mathrm{Diff}}
\newcommand\Emb{\mathrm{Emb}}

\newcommand\colim{\operatorname*{colim}}

\renewcommand\mathscr[1]{\overline{\mathcal{#1}}\vphantom{\overline{\mathcal{#1}}}}

\newcommand{\Sing}{\mathrm{Sing}}

\newcommand{\Fr}{\mathrm{Fr}}
\newcommand{\R}{\bR}
\newcommand{\N}{\mathbbm{N}}
\newcommand{\Z}{\mathbbm{Z}}
\newcommand{\Q}{\mathbbm{Q}}

\newcommand{\Int}{\mathrm{int}}

\newcommand{\Hom}{\mathrm{Hom}}

\newcommand{\Aut}{\mathrm{Aut}}

\renewcommand{\epsilon}{\varepsilon}

\newcommand{\Mod}{\mathrm{Mod}}

\newcommand{\id}{\mathrm{id}}

\newcommand{\Map}{\mathrm{Map}}
\newcommand{\Gr}{\mathrm{Gr}}

\newcommand{\Ob}{\mathrm{Ob}}

\newcommand{\Sets}{\mathrm{Sets}}
\newcommand{\Top}{\mathrm{Top}}

\newcommand{\F}{\mathbb{F}}

\makeatletter
\def\@setcontribs{%
  \@xcontribs
  {\xcontribs}%
}
\let\@wraptoccontribs\wraptoccontribs
\def\@setauthors{%
  \begingroup
  \def\thanks{\protect\thanks@warning}%
  \trivlist
  \centering\footnotesize \@topsep30\p@\relax
  \advance\@topsep by -\baselineskip
  \item\relax
  \author@andify\authors
  \def\\{\protect\linebreak}%
{\large\authors\par\vspace*{-9pt}}
  \ifx\@empty\contribs
  \else
  \@setcontribs
    \@closetoccontribs
  \fi
  \endtrivlist
  \endgroup
}
\makeatother
\begin{document}

%
%
%
%
%
%

\newcommand\bettertitle{Lectures on Invertible Field Theories}
\newcommand\bettershorttitle{Invertible Field Theories}
\newcommand\exauthors{Arun Debray, S{\o}ren Galatius, and Martin Palmer} 

\title[\bettershorttitle]{\bettertitle}

%
%

\author[S{\o}ren Galatius]{S{\o}ren Galatius\\[4pt]}
\email{galatius@math.ku.dk}
\address{Department of Mathematics, University of Copenhagen, Denmark}

\contrib[with Exercises by Arun Debray, S{\o}ren Galatius, and Martin Palmer]{}
%
%
\keywords{Park City Mathematics Institute, mapping class group, surface bundles}

\thispagestyle{empty}
\begin{abstract}
  Four lectures on invertible field theories at the Park City Mathematics Institute 2019.
  
  Cobordism categories are introduced both as plain categories and topologically enriched.  We then discuss localization of categories and its relationship to classifying spaces, and state the main theorem of classification of invertible field theories in these terms.  We also discuss symmetric monoidal structures and their relationship to actions of the little disk operads.  In the final lecture we discuss an application of cobordism categories to characteristic classes of surface bundles.

  Emphasis will be on self-contained definitions and statements, referring to original literature for proofs.
\end{abstract}  

%
%
\maketitle 

\tableofcontents

%
%

\setcounter{section}{-1}
\section*{Introduction}\refstepcounter{section}
\addcontentsline{toc}{section}{Introduction}

The notion of \emph{cobordism} between smooth manifolds and its relationship with \emph{homotopy theory} goes back at least to the work of Pontryagin and Thom.  In these lectures we will outline some more recent developments, including connections to \emph{diffeomorphism groups}.

\subsection*{Acknowledgments}

I would like to thank Arun Debray and Martin Palmer for their help with the exercise sessions during the summerschool including their coauthorship on the exercises at the end of these notes.  I also thank David Ayala and Ulrike Tillmann for helpful conversations and feedback on draft versions of these notes.  Finally, I thank the PCMI and all participants for creating a unique and inspiring atmosphere.

The author was supported by the European Research Council (ERC) under the European Union's Horizon 2020 research and innovation programme (grant agreement No.\ 682922),  the Danish National Research Foundation through the Centre for Symmetry and Deformation (DNRF92), and the EliteForsk Prize.

\section*{Lecture 1: Cobordisms}\refstepcounter{section}
\label{sec:cobordisms}
\addcontentsline{toc}{section}{Lecture 1: Cobordisms}

Recall that a smooth manifold $M$ is \emph{closed} if it is compact and $\partial M = \emptyset$.  We do not require $M$ to be connected.  For any $d \in \Z$, the empty set is a manifold of dimension $d$.

\subsection{Classical definitions}
\label{sec:classical}

\begin{definition}
  Let $M_0$ and $M_1$ be smooth closed manifolds of dimension $d-1$.  An (abstract) \emph{cobordism} from $M$ to $N$ is a triple $(W,c_\mathrm{in},c_\mathrm{out})$, where $W$ is a compact manifold, and
  \begin{align*}
    c_\mathrm{in}: [0,\infty) \times M_0 &\hookrightarrow W\\
    c_\mathrm{out}: (-\infty,0] \times M_1 &\hookrightarrow W
  \end{align*}
  are smooth embeddings with disjoint images, such that
  \begin{equation*}
    \partial W = c_\mathrm{in}(\{0 \} \times M_0) \cup c_\mathrm{out}(\{0\} \times M_1).
  \end{equation*}
  Sometimes we shall omit the two charts $c_\mathrm{in}$ and $c_\mathrm{out}$ from the notation and simply say ``$W$ is a cobordism from $M_0$ to $M_1$'' and write $W: M_0 \leadsto M_1$.
  
  Two manifolds $M_0$ and $M_1$ are \emph{cobordant} if there exists a cobordism from $M_0$ to $M_1$.
\end{definition}

\newcommand{\cin}{c_\mathrm{in}}
\newcommand{\cout}{c_\mathrm{out}}

\begin{lemma}\label{lem:equiv-rel}
  Being cobordant is an equivalence relation.
\end{lemma}
\begin{proof}[Proof sketch]
    We leave reflexivity and symmetry to the reader and focus on transitivity.  If $(W,\cin,\cout)$ is a cobordism from $M_0$ to $M_1$ and $(W',\cin',\cout')$ is a cobordism from $M_1$ to $M_2$, then we may define a cobordism $(W'',\cin'',\cout'')$, where the underlying space $W''$ is
  \begin{equation*}
    W'' = \frac{W \amalg W'}{\cout(0,x) \sim \cin'(0,x), \,\, \forall x \in M_1}
  \end{equation*}
  in the quotient topology, i.e.\ the coarsest topology making the canonical maps
  \begin{align*}
    i: W & \to W''\\
    i': W' & \to W''
  \end{align*}
  continuous.  Charts near a glued point $\cout(0,x) = \cin'(0,x)$ for $x \in M_1$ are obtained as
  \begin{equation*}
    (t,u) \mapsto
    \begin{cases}
      \cout(t,\phi(u)) & \text{for $t \leq 0$}\\
      \cin'(t,\phi(u)) & \text{for $t \geq 0$}
    \end{cases}
  \end{equation*}
  for a chart $\phi: U \to M_1$ defined on an open $U \subset \R^{d-1}$.  Around other points in $W''$ charts are obtained by composing a chart in $W$ with $i$ or a chart in $W'$ with $i'$.  It is easy to check that this defines a smooth structure on $W''$.  We obtain a cobordism $M_0 \leadsto M_2$ by setting $\cin'' = i \circ \cin$ and $\cout'' = i' \circ \cout$.
\end{proof}

\begin{remark}
  The set of cobordism classes of $k$-dimensional manifolds is often denoted $\mathfrak{N}^k$.  The direct sum
  \begin{equation*}
    \mathfrak{N}^* = \bigoplus_{k \geq 0} \mathfrak{N}^k
  \end{equation*}
  has the structure of a commutative graded ring (addition induced by disjoint union and multiplication by product of smooth manifolds) in which $2 = 0$.  In his 1954 paper \emph{Quelques propri\'et\'es globales de vari\'et\'es differentiables}, Ren\'e Thom completely determined this ring: there is an isomorphism of commutative graded rings
  \begin{equation*}
    \F_2[x_k \mid k \neq 2^i - 1] \xrightarrow{\cong} \mathfrak{N}^*,
  \end{equation*}
  from the polynomial ring over $\F_2$ with one generator in each dimension not of the form $k = 2^i - 1$.

  The citation for Ren\'e Thom's Fields medal at the 1958 ICM is: ``In 1954 [he] invented and developed the theory of cobordism in algebraic topology. This classification of manifolds used homotopy theory in a fundamental way and became a prime example of a general cohomology theory.''

  The work presented in these lectures owes an intellectual debt to Thom's profound 1954 paper, although we shall not make direct use of his calculations.
\end{remark}

In these lectures we shall be concerned with the \emph{reasons} that two closed manifolds $M_0$ and $M_1$ are cobordant, rather than merely \emph{whether} they are cobordant as in Thom's 1954 paper.

\subsection{Cobordism categories, first attempt}
\label{sec:first-attempt}

\begin{definition}
  Let $(W,\cin,\cout)$ and $(W',\cin',\cout')$ be two cobordisms between $M_0$ and $M_1$.  We say that these are \emph{diffeomorphic as cobordisms} if there exists a diffeomorphism $\phi: W \to W'$ of underlying smooth manifolds, such that there exists an $\epsilon > 0$ with $\phi(\cin(t,x)) = \cin'(t,x)$ for all $(t,x) \in [0,\epsilon] \times M_0$ and $\phi(\cout(t,x)) = \cout'(t,x)$ for all $(t,x) \in [-\epsilon,0] \times M_1$.
\end{definition}


\begin{definition}
  Let $W$ be a cobordism from $M_0$ to $M_1$ and $W'$ a cobordism from $M_1$ to $M_2$.  (Here we suppress the boundary collars from the notation, but emphasize that they are an important part of the structure.)  The \emph{composition} of $W$ and $W'$ is the cobordism $W'': M_0 \leadsto M_2$ constructed in the proof of transitivity in Lemma~\ref{lem:equiv-rel}.  We shall also write
  \begin{equation*}
    W \cup_{M_1} W'
  \end{equation*}
  for this cobordism.  (This is sloppy notation, since it depends on the collars.)
\end{definition}

\begin{lemma}
  Let $M_0$ and $M_1$ be closed $(d-1)$-manifolds and let $W_1, W_1': M_0 \leadsto M_1$ be two cobordisms between the same two manifolds, and similarly $W_2, W_2': M_1 \leadsto M_2$.  If $W_1$ and $W_1'$ are diffeomorphic as cobordisms and also $W_2$ and $W_2'$ are diffeomorphic as cobordisms, then the compositions $W_1 \cup_{M_1} W_2$ and $W_1' \cup_{M_1} W_2'$ are diffeomorphic as cobordisms.
\end{lemma}

\begin{lemma}
  Let $M_0$, $M_1$, $M_2$, and $M_3$ be closed smooth $(d-1)$-manifolds, and let $W_1: M_0 \leadsto M_1$, $W_2: M_1 \leadsto M_2$ and $W_3: M_2 \leadsto M_3$ be cobordisms.  Then the two cobordisms
  \begin{align*}
    (W_1 \cup_{M_1} W_2) \cup_{M_2} W_3: M_0 & \leadsto M_3\\
    W_1 \cup_{M_1} (W_2 \cup_{M_2} W_3): M_0 & \leadsto M_3
  \end{align*}
  are diffeomorphic as cobordisms.
\end{lemma}

The first lemma is proved by gluing together two diffeomorphisms of cobordisms and verifying that the result is a diffeomorphism of cobordisms, the second by verifying that the canonical map of underlying sets is a diffeomorphism of cobordisms.  Let us remark that the two cobordisms $(W_1 \cup_{M_1} W_2) \cup_{M_2} W_3$ and $W_1 \cup_{M_1} (W_2 \cup_{M_2} W_3)$ are likely not \emph{equal} (their underlying sets aren't equal).

\newcommand{\Cob}{\mathrm{Cob}}
\begin{definition}\label{def:composition}
  For closed $(d-1)$-manifolds $M_0$ and $M_1$, let $\Cob_d(M_0,M_1)$ denote the set of equivalence classes of cobordisms from $M_0$ to $M_1$, up to diffeomorphism of cobordisms.  The induced operation
  \begin{align*}
    \Cob_d(M_0,M_1) \times \Cob_d(M_1,M_2) &\to \Cob_d(M_0,M_2)\\
    ([W],[W']) \quad\quad\quad\quad& \mapsto [W \cup_{M_1} W']
  \end{align*}
  is then well defined and associative, by the two lemmas above.
\end{definition}

An easy argument shows that it is also \emph{unital}: in fact the ``cylinder'' $[0,1] \times M$ gives rise to an identity element in $\Cob_d(M,M)$.

\begin{definition}
  Let $\Cob_d$ be the category whose objects are closed smooth $(d-1)$-manifolds, whose morphism set $M_0 \to M_1$ is $\Cob_d(M_0,M_1)$, and whose composition is as in Definition~\ref{def:composition}.
\end{definition}

We will later enhance this definition to one which for many purposes turns out to be more natural.  The enhanced definition has more structure (it is not just a plain category), and therefore at first sight will seem more complicated.  In the rest of this lecture we shall stick to the elementary definition above.

\begin{remark}
  A \emph{small category} is one whose collections of objects and morphisms form sets.  For example, the category of sets is not small because there is no ``set of all sets''.

  As it stands, the category $\Cob_d$ defined above is not small for $d \geq 0$.  The problem is that there are too many possibilities for the underlying set of the manifolds.  This issue may be handled in multiple ways, for example as follows.  For any set $\Omega$, let $\Cob_d^\Omega$ be the subcategory of $\Cob_d$ in which the manifolds $M$ (objects) and $W$ (whose diffeomorphism class gives a morphism) are required to have underlying set a subset of $\Omega$.  Then $\Cob_d^\Omega$ is a small category, and, if the cardinality of $\Omega$ is at least that of $\R$ then the inclusion $\Cob_d^\Omega \hookrightarrow \Cob_d$ is an equivalence of categories.  If $\Omega'$ is another such set, then the small categories $\Cob_d^\Omega$ and $\Cob_d^{\Omega'}$ are \emph{equivalent}, in fact both inclusions of subcategories
  \begin{equation*}
    \Cob_d^{\Omega} \hookrightarrow \Cob_d^{\Omega \cup \Omega'} \hookleftarrow \Cob_d^{\Omega'}
  \end{equation*}
  are equivalences of categories.

  We shall not dwell further on this issue.  To be definite, let's say we choose once and for all a set $\Omega$ of sufficiently large cardinality and henceforth by abuse of notation write $\Cob_d$ for the small category $\Cob_d^\Omega$.
\end{remark}

\subsection{Categories, groupoids, and spaces}
\label{sec:categories-and-groupoids}

For $p \geq 0$, the set $[p] = \{0, \dots, p\}$ is totally ordered by the usual ordering of natural numbers, and may be considered the objects of a category, whose morphism set $i \to j$ is empty if $i > j$ and a singleton if $i \leq j$.  If $C$ is any small category, the set of all functors $f: [p] \to C$ may be identified with the set
\begin{equation*}
  N_pC = \{(f_1, \dots, f_p) \mid \text{target$(f_i)$ = source$(f_{i+1})$ for all $i$}\}
\end{equation*}
of composable $p$-tuples of morphisms: $f_1$ is given by the value of $f: [p] \to C$ on the unique morphism $0 \to 1$, $f_2$ the value on the unique morphism $1 \to 2$, etc.  

\begin{definition}
  The \emph{classifying space} of a small category $C$ is the space
  \begin{equation*}
    BC = |NC|,
  \end{equation*}
  the geometric realization of its nerve.
\end{definition}

More explicitly, $BC$ is the CW complex with one 0-cell for each object of $C$; one 1-cell for each non-identity morphism $f: x \to y$ in $C$, with one end attached to the 0-cell $x$ and the other end to $y$; one 2-cell for each pair $(f,g)$ of composable non-identity morphisms, etc.

In particular, the set $N_0C = \mathrm{Ob}(C)$, regarded as a topological space in the discrete topology, canonically sits inside $BC$ via
\begin{equation*}
  N_0C \hookrightarrow BC,
\end{equation*}
the inclusion of 0-simplices.

Let us briefly summarize some convenient properties of this construction; see e.g.\ \cite{SegalNerve}.
\begin{itemize}
\item Functoriality $\mathrm{Cat} \to \mathrm{Top}$: a functor $F: C \to D$ is sent to a continuous map $BF: BC \to BC$.
\item A natural transformation $T: F \rightarrow G$ between functors $C \to D$ induces a homotopy between the maps $BF, BG: BC \to BD$.
\item As a consequence, a functor $F: C \to D$ admitting a left or a right adjoint is sent to a homotopy equivalence.  In particular, equivalences of categories are sent to homotopy equivalences.  Likewise, $BC$ is contractible if $C$ admits a terminal or an initial object.
\item Products: the canonical map $B(C \times D) \to BC \times BD$ is a homeomorphism\footnote{Fine print: in this statement the product of $BC$ and $BD$ should be formed in compactly generated spaces.  Otherwise we can only say it is a continuous bijection, although still a weak equivalence.}.
\end{itemize}

An overarching theme in these lectures is the question of how much information about $C$ is retained by the space $BC$.

\begin{example}\label{ex:groupoid}
  For $c_0 \in \mathrm{Ob}(C)$, let $C_{/c_0}$ be the \emph{over category}: its objects are morphisms $f: c \to c_0$ in $C$ and its morphisms $(f: c \to c_0) \to (f': c' \to c_0)$ are morphisms $g: c \to c'$ in $C$ such that $f = f' \circ g$.  There is a forgetful functor $C_{/c_0} \to C$ which induces a map of spaces
  \begin{equation*}
    u: BC_{/c_0} \to BC.
  \end{equation*}
  It is not hard to show that $BC_{/c_0}$ is contractible (it has the terminal object $\mathrm{id}: c_0 \to c_0$, and any category with a terminal object has contractible classifying space).

  If in addition $C$ is a \emph{groupoid}, i.e.\ all morphisms are invertible, then the map $u$ can be shown to be a covering map, and it is not hard to use this to construct an isomorphism of groups
  \begin{equation*}
    \mathrm{Aut}_C(c_0) \to \pi_1(BC,c_0).
  \end{equation*}
  Since all path components of $BC$ admit a contractible covering space, all higher homotopy groups vanish.  Thus we have completely understood the homotopy type of $BC$ when $C$ is a groupoid: it is a disjoint union of $K(\pi,1)$'s, one for each isomorphism class in $C$.
\end{example}

Taking classifying space turns a category into a space.  One way to go in the other direction is the \emph{fundamental groupoid}.
\begin{definition}
  For a space $X$, let $\pi_1(X)$ denote\footnote{Other commonly used notation for the fundamental groupoid is $\Pi_1(X)$ or $\pi_{\leq 1}(X)$.} the category whose objects are the points of $x$, whose morphisms $x \to y$ are homotopy classes of paths $[0,1] \to X$ starting at $x$ and ending at $y$, up to homotopy relative to the endpoint, and  whose composition is induced by concatenation of paths.

  For any subset $A \subset X$ we write $\pi_1(X,A)$ for the full subcategory of $\pi_1(X)$ whose objects are the points in $X$.
\end{definition}
As the name indicates, the category $\pi_1(X)$ is in fact a groupoid.  It simultaneously encodes the fundamental groups $\pi_1(X,x)$ for all $x \in X$ and the change-of-basepoint isomorphisms $\pi_1(X,x) \to \pi_1(X,x')$ induced by paths between $x$ and $x'$.
\begin{lemma}
  Let $C$ be a small category and let
  \begin{equation*}
    \gamma: C \to \pi_1(BC)
  \end{equation*}
  send an object $x \in C$ to the corresponding 0-simplex in $BC$, regarded as an object of $\pi_1(BC)$, and a morphism $f: x \to y$ to the homotopy class represented by the 1-simplex corresponding to $f: x \to y$ (and the constant path if $f$ is an identity morphism).  Then $\gamma$ is a (well defined) functor.
\end{lemma}
\begin{proof}[Proof sketch]
  We must prove that $\gamma$ preserves composition.  While the 1-simplex corresponding to a composition $h = f \circ g$ may not be \emph{equal} to the concatenation of the paths corresponding to $f$ and $g$, the 2-simplex corresponding to the pair $(f,g) \in N_2C$ gives rise to a homotopy relative to end-points.
\end{proof}

If $C$ is a groupoid it is not hard to show, as in Example~\ref{ex:groupoid} above, that $\gamma$ is an equivalence of categories.  In fact, if we keep track of the subset $N_0C \subset BC$ we may reconstruct $C$ up to isomorphism of groupoids, as follows.
\begin{proposition}\label{prop:groupoid-case}
  If $C$ is a groupoid, then $\gamma: C \to \pi_1(BC)$ is an equivalence of categories.  In fact it is an isomorphism of categories onto $\pi_1(BC,N_0C)$, the full subcategory of $\pi_1(BC)$ on the objects $N_0C \subset BC$.\qed
\end{proposition}

In fact $\gamma$ is an equivalence of categories if and only if  $C$ is a groupoid, since $\pi_1(BC)$ is always a groupoid.  In general we have the following universal property, expressing that $\gamma$ is the ``universal functor to a groupoid''.
\begin{proposition}
  The functor
  \begin{equation*}
    \gamma: C \to \pi_1(BC,N_0C)
  \end{equation*}
  is the universal groupoid under $C$: for any groupoid $D$ and functor $F: C \to D$, there is a unique functor $G: \pi_1(BC,N_0C) \to D$ such that $F = G \circ \gamma$.
\end{proposition}

\begin{proof}[Proof sketch]
  We first explain uniqueness.  Since $\gamma$ is a bijection of object sets it is fixed what the functor $G$ must do on objects.  A morphism in $\pi_1(BC,N_0C)$ is represented by a path $\lambda: [0,1] \to BC$ starting and ending in $N_0C \subset BC$, and any such path is homotopic relative to its endpoints to a path running entirely in the 1-skeleton of $BC$.  Since the 1-simplices of $BC$ are (non-identity) morphisms in $C$, any morphism in $\pi_1(BC,N_0 C)$ may be written as a finite composition of the form
  \begin{equation*}
    (\gamma(\phi_1))^{\epsilon_1} \dots (\gamma(\phi_n))^{\epsilon_n},
  \end{equation*}
  where $\epsilon_i \in \{\pm 1\}$ and $\phi_i$ are morphisms in $C$.  The requirement $F = G \circ \gamma$ forces $G$ to take this element to $(F(\phi_1))^{\epsilon_1} \dots (F(\phi_n))^{\epsilon_n}$.

  To see existence, the easiest is to use proposition~\ref{prop:groupoid-case}: define $G: \pi_1(BC,N_0 C) \to D$ as the inverse of the isomorphism $D \to \pi_1(BD,N_0 D)$, composed with $\pi_1(BF): \pi_1(BC,N_0 C) \to \pi_1(BD,N_0 D)$.
\end{proof}
As with any universal property, $\pi_1(BC,N_0C)$ is determined up to unique isomorphism of categories by this proposition.  It therefore also determines $\pi_1(BC)$ up to canonical equivalence of categories (since the inclusion $\pi_1(BC,N_0C) \hookrightarrow \pi_1(BC)$ is an equivalence).

\begin{remark}
  The natural transformation $\gamma: C \to \pi_1(BC,N_0C)$ is a special case of \emph{localization of categories}.  More generally, given a small category $C$ and a subcategory $S \subset C$, there exists a functor $\gamma: C \to C[S^{-1}]$ which is a bijection on object sets, sends any morphism in $S$ to an isomorphism in $C[S^{-1}]$, and is initial with this property (i.e.\ if $f: C \to D$ is a functor sending morphisms in $S$ to isomorphisms, then it factors as $f = g \circ \gamma$ for a unique $g: C[S^{-1}] \to D$).  The universal functor which sends all morphisms to isomorphisms is then $C \to C[C^{-1}]$, so the result above may be summarized as an isomorphism of categories
  \begin{equation*}
    C[C^{-1}] \to \pi_1(BC,N_0C)
  \end{equation*}
  and hence an equivalence of categories
  \begin{equation*}
    C[C^{-1}] \xrightarrow{\simeq} \pi_1(BC).
  \end{equation*}
\end{remark}

\begin{remark}
  We have explained that, up to equivalence of categories, no information is lost in the passage $C \mapsto BC$, as long as $C$ is a groupoid.

  If $C$ is not a groupoid, the passage from $C$ to $BC$ should be expected to lose lots of information.  For example, if $C$ is the two-object category with object set $\{0,1\}$, one morphism $0 \to 1$, and no other non-identity morphisms, then $BC$ is homeomorphic to the interval $[0,1]$.  If $C' = \{0\} \hookrightarrow C$ is the inclusion of the 1-object category, then $BC' \to BC$ is a homotopy equivalence (namely the inclusion $\{0 \} \subset [0,1]$), even though $C' \to C$ is not an equivalence of categories.

  More examples of this flavor are easily constructed.  For example, to any simplicial complex $X$ there is an associated ``poset of simplices'' of $X$, regarded as a category (i.e.\ $\mathrm{Hom}(\sigma,\tau)$ is a singleton if $\sigma$ is a face of $\tau$ and empty otherwise).  Then it is a standard result that $BC$ is homeomorphic to $|X|$.
\end{remark}



\subsection{Invertible field theories (poor man's version)}
\label{sec:invertible}

In light of what we have discussed so far, understanding functors from a small category $C$ into groupoids is essentially the same thing as understanding the universal such $C \to C[C^{-1}]$.  As we have seen, this is the same as understanding the fundamental groupoid of $BC$.  If we want an up-to-isomorphism-of-categories answer we should also keep track of the subset $N_0 C \subset BC$, but for an up-to-equivalence-of-categories answer we only need the space $BC$ itself, and then we only need it up to homotopy equivalence.  In fact we need less than the homotopy type of $BC$, we only need to understand the space $BC$ up to maps that induce equivalence of fundamental groupoids (i.e.\ the ``1-type'' of $BC$).  For example, if we can find a space $X$ and a 2-connected map $X \to BC$ (meaning all homotopy fibers are simply connected), then the induced functor $\pi_1(X) \to \pi_1(BC)$ is an equivalence.  Let us spell out this  observation.
\begin{corollary}
  Let $C$ be a small category and let $f: X \to BC$ be a 2-connected map.  Then the functors
  \begin{equation*}
    C \xrightarrow{\gamma} \pi_1(BC) \xleftarrow{\pi_1(f)} \pi_1(X)
  \end{equation*}
  induce, after choosing an inverse to the equivalence $\pi_1(f): \pi_1(X) \to \pi_1(BC)$, an equivalence of categories
  \begin{equation*}
    C[C^{-1}] \to \pi_1(X)
  \end{equation*}
  and for any small groupoid $D$ it induces an equivalence
  \begin{equation*}
    \mathrm{Fun}(\pi_1(X),D) \to \mathrm{Fun}(C,D).
  \end{equation*}
  In turn, these groupoids are also canonically equivalent to the fundamental groupoids of the mapping spaces $\Map(X,BD) \simeq \Map(BC,BD)$, at least if $X$ is a CW complex.
\end{corollary}
  
We shall be interested in the case $C = \Cob_d$.  In this case functors from $\Cob_d$ to groupoids may be regarded as an approximation to the notion of \emph{invertible field theories} with values in the groupoid $D$.  The actual definition involves\footnote{Depending on who you ask, it may likely also involve higher categories.} symmetric monoidal functors into rigid symmetric monoidal groupoids, as we shall explain later in the week.  The approach we shall take, namely to find a 2-connected map from an ``understandable'' space, also makes sense in the simpler case of plain categories and groupoids, and it seems instructive to discuss this setting first.

\section*{Lecture 2: Topologically enriched (cobordism) categories}\refstepcounter{section}
\label{sec:enriched}
\addcontentsline{toc}{section}{Lecture 2: Topologically enriched (cobordism) categories}

Recall that a (small) \emph{topologically enriched} category consists of a category $C$ and a specified topology on each morphism set $C(x,y)$ such that the composition map $C(x,y) \times C(y,z) \to C(x,z)$ is continuous\footnote{For technical reasons it is probably better to require all morphism spaces $C(x,y)$ to be \emph{compactly generated} topological spaces.  If this is not already the case for $C$, it may be achieved by applying $k$-ification to each morphism space.
  Similarly, the domain of the composition map should be the product in compactly generated topological spaces.  In these notes we will gloss over this and many other details.}
  for all $x$, $y$, and $z$.
The underlying (unenriched) category has the same object set, morphism set $x \to y$ given by $C(x,y)^\delta$, the underlying set of $C(x,y)$, and composition as in $C$.


A \emph{topologically enriched} (also known as \emph{continuous}) functor $F: C \to D$ between topologically enriched categories is a functor of underlying categories with the property that the induced map of morphism spaces is continuous.


\subsection{Topologically enriched cobordism categories}

In this section we shall define a version of the cobordism category $\Cob_d$ which is a topologically enriched category.  It is very similar in spirit to $\Cob_d$, except that the objects and morphisms are now \emph{submanifolds} of a fixed vector space, instead of being abstract manifolds.

If $V$ is a finite-dimensional real vector space and $M \subset V$ is a subset, then being a $d$-dimensional submanifold (without boundary) is a \emph{property} of the subset: it is the condition that it locally be the inverse image of a regular value of a smooth function to $\R^n$ for $n = \dim(V) - d$.

For $t > 0$ we shall say that a compact subset $W \subset [0,t] \times V$ is a \emph{cobordism} between two compact submanifolds $M_0 \subset V$ and $M_1 \subset V$, provided there exists an $\epsilon > 0$ such that
\begin{align*}
  W \cap ([0,\epsilon] \times V) & = [0,\epsilon] \times M_0\\
  W \cap ([t-\epsilon,t] \times V) & = [t-\epsilon,t] \times M_1,
\end{align*}
and such that the subset
\begin{equation*}
  \widehat{W} = \big((-\infty,0] \times M_0 \big) \cup W \cup \big( [t,\infty) \times M_1\big)
\end{equation*}
is a smooth submanifold of $\R \times V$.

\begin{definition}
  For a finite-dimensional real vector space $V$, the objects of $\mathcal{C}_d^V$ are the compact subsets $M \subset V$ which are smooth $(d-1)$-dimensional submanifolds.  Morphisms $M_0 \leadsto M_1$ are pairs $(t,W)$ where $t > 0$ and $W \subset [0,t] \times V$ is a cobordism in the sense above.  Composition is defined as
  \begin{equation*}
    (t_1,W_1) \circ (t_2, W_2) = (t_1 + t_2,W_1 \cup (t_1 + W_2)),
  \end{equation*}
  where $t_1 + W_2 = \{(t_1,0) + w \mid w \in W_2\}$ is the parallel translate of $W_2$.
\end{definition}
\begin{lemma}
  The composition defined above is well defined, and makes $\mathcal{C}_d^V$ into a (small, non-unital
  category.
\end{lemma}
\begin{proof}
  This is in fact much easier than for the ``abstract'' cobordisms considered earlier.  The power set of $V$ is a set, and $\mathrm{Ob}(\mathcal{C}_d^V)$ is a subset thereof, hence it is again a set; morphisms form a set for a similar reason.  Composition is given by \emph{union of subsets} which is strictly associative (unlike the gluing operation of abstract manifolds, which is only ``associative up to (preferred) diffeomorphism'').
\end{proof}
As defined above there are no units for composition, due to the requirement $t > 0$.  We could fix this by allowing $t = 0$ with some care, but instead we shall just consider $\mathcal{C}_d(V)$ a \emph{non-unital category}.

Next we define a topology on the morphism spaces of $\mathcal{C}_d^V$.
If $M_0$ and $M_1$ are objects of $\mathcal{C}_d(V)$ and $W$ is an \emph{abstract} cobordism between them, then we have a map
\begin{equation}\label{eq:1}
  \begin{aligned}
  \R_{>0} \times \Emb_\partial(W,[0,1] \times V) & \to \mathcal{C}_d^V(M_0,M_1)\\
  (t,\phi) & \mapsto (t,\lambda_t(\phi(W))),
\end{aligned}
\end{equation}
where $\lambda_t$ stretches by a factor of $t$ in the first factor, i.e.\ it is the map $\R \times V \to \R \times V$ is given by $(s,v) \mapsto (st,v)$.  The subscript ``$\partial$'' denotes that the embeddings are required to agree with the canonical $M_0 \hookrightarrow \{0\} \times V$ and $M_1 \hookrightarrow \{1\} \times V$ when restricted to $\partial W$, and also that they must have a product behavior on a collar neighborhood of the boundary.

The point is now that the domain of~(\ref{eq:1}) has a natural topology: the Euclidean topology on $\R_{>0}$ and the $C^\infty$ topology on the embedding space.  The latter is the topology in which convergence means uniform convergence of all partial derivatives of any order (with respect to a chart on $W$ it should be uniform convergence of partial derivatives on compact subsets of the chart).  We give the codomain of~(\ref{eq:1}) the quotient topology (as $W$ ranges over all abstract cobordisms).

\begin{definition}
  Let $\mathcal{C}_d^V$ be the topologically enriched category whose objects are the closed $(d-1)$-dimensional submanifolds $M \subset V$ and whose morphism spaces are the sets of embedded cobordisms $(t,W)$ in the quotient topology from the $C^\infty$ topology, as defined above.
\end{definition}

At first sight the definition of $\mathcal{C}_d^V$ may seem a bit ad hoc: for example, the particular way that ``collar conditions'' near incoming and outgoing boundaries is handled; one could imagine other definitions (e.g.\ objects could come with a germ of a $d$-dimensional thickening) which would lead to a topological category which is likely not isomorphic to the one defined above.  However, we only want to declare this particular model to be ``the'' canonical model up to a suitable notion of \emph{equivalence}, namely the so-called \emph{DK-equivalence}.  We shall return briefly to this homotopy theory question later in the lectures.

Before proceeding, let us discuss the geometric meaning of the homotopy type of the morphisms spaces $\mathcal{C}_d^V(M_0,M_1)$.

\subsection{Manifold bundles and bundles of cobordisms}
\label{sec:bundles}

We will only give a cursory discussion of this material, referring the reader to the original sources for further details.

To simplify matters we postpone the discussion of boundaries.  Hence, let $W$ be a closed manifold, let $V$ be a finite dimensional real vector space, and consider the orbit space
\begin{equation}\label{eq:2}
  \Emb(W,V)/\Diff(W),
\end{equation}
where the diffeomorphism group of $W$ acts on the embedding space by precomposition.  As a set, this quotient may be identified with a subset of the power set of $V$, namely those subsets of $V$ which have the property of being submanifolds which additionally are diffeomorphic to $W$.  The space is topologized as a quotient of the embedding space, which is topologized in the $C^\infty$ topology.

How does one produce continuous maps \emph{into} the space~(\ref{eq:2})?  A good way is to start with an \emph{embedded bundle}.  Suppose $X$ is a smooth manifold without boundary, but not necessarily compact, and that $E \subset X \times V$ is a subset with the following properties
\begin{itemize}
\item $E$ is a smooth submanifold of $X \times V$, without boundary,
\item the composition
  \begin{equation*}
    E \xrightarrow{\text{inclusion}} X \times V \xrightarrow{\text{projection}} X
  \end{equation*}
  is proper (inverse image of compact is compact), smooth ($C^\infty$), and a submersion (derivative at any point is surjective),
\item the inverse image $E_x \subset V$ of $x \in X$ is diffeomorphic\footnote{$E_x$ is given the smooth structure arising as the inverse image of a regular point, which agrees with the smooth structure it inherits as a submanifold of $V$.} to $W$, for all $x \in X$.
\end{itemize}
Let us say that a subset $E \subset X \times V$ with these properties is an \emph{smooth $W$-bundle over $X$ embedded in $X \times V$.}
For a subset $E \subset X \times V$ with these properties, the map
\begin{align*}
  X & \to\Emb(W,V)/\Diff(W)\\
  x & \mapsto E_x
\end{align*}
can be shown to be continuous.  Moreover, any continuous map is homotopic to one arising this way, and in fact this correspondence gives rise to a bijection
\begin{equation}\label{eq:3}
  [X,\Emb(W,V)/\Diff(W)] \cong \frac{\{\text{$E \subset X \times V$ as above}\}}{\simeq}
\end{equation}
where the equivalence relation $\simeq$ on the right denotes \emph{concordance}: a concordance from $E$ to $E'$ is an embedded smooth $W$ bundle over $\R \times X$  whose restriction to $\{0\} \times X$ is $\{0\} \times E$ and to $\{1\} \times X$ is $\{1\} \times E'$.

In this sense, $\Emb(W,V)/\Diff(W)$ is a classifying space for embedded smooth $W$-bundles, in much the same way as Grassmannians are classifying spaces for vector bundles embedded in trivial bundles.

One way to prove~(\ref{eq:3}) proceeds by giving $\Emb(W,V)/\Diff(W)$ the structure of an \emph{infinite dimensional manifold} and proving that any continuous map from $X$ may be homotoped to a smooth one.  The usual reference seems to be \cite{MR613004}.

The space $\mathcal{C}_d(M_0,M_1)$ may be understood similarly, as a classifying space for ``embedded bundles of cobordisms''.
\begin{definition}
  Let $M_0$ and $M_1$ be objects of $\mathcal{C}_d$, and let $X$ be a smooth manifold without boundary (possibly non-compact).  A subset $E \subset X \times ([0,1] \times V)$ is a \emph{bundle of cobordisms} from $M_0 \subset V$ to $M_1 \subset V$ provided
  \begin{enumerate}[(i)]
  \item $E$ is a smooth submanifold,
  \item the projection map $\pi: E \to X$ is a proper submersion,
  \item the boundary of $E$ equals $X \times (\{0\} \times M_0 \cup \{1\} \times M_1)$, and for some $\epsilon > 0$, 
    \begin{align*}
      E \cap (X \times ([0,\epsilon] \times V)) & = X \times ([0,\epsilon] \times M_0)\\
      E \cap (X \times ([1-\epsilon,1] \times V)) & = X \times ([1-\epsilon,1] \times M_1).\\
    \end{align*}
  \end{enumerate}
\end{definition}
As for closed manifolds, such a bundle of embedded cobordisms gives rise to a continuous map
\begin{equation*}
  X \to \mathcal{C}^V_d(M_0,M_1),
\end{equation*}
any continuous map is homotopic to one arising this way, and this construction induces a natural bijection
\begin{equation*}
  [X,\mathcal{C}^V_d(M_0,M_1)] \cong \frac{\text{$E \subset X \times ([0,1] \times V)$ bundle of cobordisms from $M_0$ to $M_1$}}{\simeq}
\end{equation*}
where the equivalence relation $\simeq$ on the right denotes \emph{concordance}: a concordance from $E$ to $E'$ is a bundle of cobordisms over $\R \times X$ whose restriction to $\{0\} \times X$ is $\{0\} \times E$ and to $\{1\} \times X$ is $\{1\} \times E'$.

\subsection{Infinite dimensional ambient space}
\label{sec:infinite}

To get a ``more universal'' object we may take colimit over $V$.
\begin{definition}
  Let
  \begin{equation*}
    \mathcal{C}_d = \colim_{V \subset \R^\infty} \mathcal{C}_d^V
  \end{equation*}
  be the topologically enriched category defined as the colimit over finite-dimensional linear subspaces $V \subset \R^\infty = \oplus_\N \R$.
\end{definition}

For any other infinite-dimensional real vector space $\mathcal{U}$ we could similarly define a cobordism category $\mathcal{C}_d^{\mathcal{U}}$ by taking colimit over finite-dimensional $V \subset \mathcal{U}$ but this would not be especially interesting: the comparison functors
\begin{equation*}
  \mathcal{C}_d^{\mathcal{U}} \to \mathcal{C}_d^{\mathcal{U} \oplus \R^\infty} \leftarrow \mathcal{C}_d^{\R^\infty} = \mathcal{C}_d.
\end{equation*}
are both ``$DK$-equivalences'' in the sense explained below.

For reasons of time we shall mostly stick to finite-dimensional $V$ in these lectures, but let us record one interesting connection to diffeomorphism groups of manifolds.
\begin{proposition}
  There is a weak equivalence
  \begin{equation}\label{eq:9}
    \mathcal{C}_d(M_0,M_1) \simeq \coprod_W B\Diff_\partial(W),
  \end{equation}
  where $W$ ranges over abstract cobordisms from $M_0$ to $M_1$, one in each diffeomorphism class, and $\Diff_\partial(W)$ denotes the group of diffeomorphisms of $W$ which act as the identity near $\partial W$, topologized in the $C^\infty$ topology.
\end{proposition}


  We will again not prove this in any detail, but remark that it boils down to the embedding space $\Emb_\partial(W,\R^\infty)$ being weakly contractible and the quotient map to $\Emb_\partial(W,\R^\infty)/\Diff_\partial(W)$ being a fiber bundle.

\subsection{Categories versus enriched categories}


In order to discuss the relationship between $\Cob_d$ and $\mathcal{C}_d$, we first discuss the relationship between enriched and unenriched categories in the abstract.
\begin{definition}
  Let $D$ be a small (ordinary) category.  Then we may regard $D$ as a topologically enriched category by giving each set $D(x,y)$ the discrete topology.  This topologically enriched category shall be denoted $\iota D$.

  Let $\mathrm{TCat}$ denote the category of small topologically enriched small categories, and topologically enriched functors between them.  Thus we have defined a functor $\iota: \mathrm{Cat} \to \mathrm{TCat}$.
\end{definition}

\begin{definition}
  For a small topologically enriched category $C$, let $h  C$ denote the (unenriched) category with the same object set as $C$, but with morphisms sets given by
  \begin{equation*}
    h  C(x,y) = \pi_0(C(x,y)),
  \end{equation*}
  the set of path components of the morphism space in $C$.  Composition is induced by composition in $C$ (using that $\pi_0$ preserves products).
\end{definition}

\begin{definition}
  Let $C$ and $D$ be topologically enriched categories.  Then a functor $F: C \to D$ is a \emph{DK-equivalence}\footnote{This terminology is more commonly used in the setting of \emph{simplicially enriched} categories, but makes sense for topologically enriched as well.  If desired, a topologically enriched category can of course be turned into a simplicially enriched one by taking $\Sing$ of all morphism spaces.} if the induced functor $(h F): (hC) \to (hD)$ is an equivalence in the usual sense and if the maps
  \begin{equation*}
    C(x,y) \to D(Fx,Fy)
  \end{equation*}
  are weak equivalences for all objects $x,y \in C$.
\end{definition}

\begin{example}\label{ex:Sing-real}
  For a topologically enriched category $C$ we can define another topologically enriched category $C'$ with the same object set, but with $C'(x,y) = |\Sing(C(x,y))|$, where $\Sing$ denotes the total singular complex.  Composition is induced by composition in $C$, using that both $\Sing$ and geometric realization of simplicial sets preserves finite products.  The evaluation maps $|\Sing(C(x,y))| \to X(x,y)$ are compatible with composition and hence define a DK-equivalence $C' \to C$.  In particular any topologically enriched category is equivalent to one whose morphism spaces are CW complexes.  See Exercise~\ref{Postnikov} for more details.
\end{example}

\begin{remark}
  The topologically enriched categories appearing in these lectures should be understood as homotopy theoretic objects.  In the same way as a weak equivalence $X \to Y$ of spaces lets us consider them as ``models'' for the same homotopy type, we shall consider two topologically enriched categories as ``models'' for the same platonic object whenever we have given a DK-equivalence between them, or at least a zig-zag of such.

  The point-set definition of $\mathcal{C}_d^V$ and $\mathcal{C}_d$ that we wrote down above should be understood in this sense, as reference models for the (embedded) cobordism category.  
  
  Topologically enriched categories form one model for $(\infty,1)$-categories.  Presumably everything in these notes could have been presented in any of the other popular settings (simplicially enriched categories, quasi-categories, complete Segal spaces, etc.)
\end{remark}

The relationship between $\mathcal{C}_d$ and $\Cob_d$ may  be expressed succintly as an equivalence of categories $h  \mathcal{C}_d \simeq \Cob_d$, as follows from the description~(\ref{eq:9}).  More precisely the following holds.
\begin{proposition}
  Let $\mathcal{C}_d^V \to \iota \Cob_d$ be the functor which sends an object $M_0 \subset V$ to the underlying abstract smooth manifold $M_0$ and an embedded cobordism $W \subset [0,t] \times V$ to its diffeomorphism class as a cobordism.  Then the induced functor
  \begin{equation*}
    h \mathcal{C}_d^V \to \Cob_d
  \end{equation*}
  is an equivalence of categories, provided $\dim(V) \gg d$.
\end{proposition}
\begin{proof}[Proof sketch]
  It is essentially surjective if $\dim(V) \geq 2d-1$, since any $(d-1)$-manifold $M$ admits an embedding into $\R^{2d-1}$.  It is full when $\dim(V) \geq 2d$, since any cobordism from $M_0$ to $M_1$ may then be embedded into $[0,t] \times \R^{2d}$ relative to any specified embeddings of $M_0$ and $M_1$, and it is faithful when $\dim(V) \geq 2d+1$ since any two embeddings of a cobordism are then isotopic, so represent the same element in $\pi_0\mathcal{C}_d^V(M_0,M_1)$.
\end{proof}

Since $\Cob_d$ is obtained from the topologically enriched category $\mathcal{C}_d^V$ by the functorial construction $C \mapsto hC$, we take the point of view that the latter ``contains more information''.  Once one becomes sufficiently advanced, one might decide that the enriched category $\mathcal{C}_d$, or perhaps some even more sophisticated object, is a more natural object to consider than $\Cob_d$.  Nevertheless, we shall in the rest of this section discuss how information about $\mathcal{C}_d^V$, such as the main theorem below, have implications for $\Cob_d$.

Under a mild assumption, the construction $C \mapsto h C$ has the following universal property.
\begin{lemma}
  Let $C$ be a topologically enriched category such that all morphism spaces $C(x,y)$ are locally path connected\footnote{On a technical level it is often more convenient to work with categories enriched in simplicial sets, instead of in topological spaces.  For example, the local path connectedness assumption here would go away: the analogue of taking the discrete topology is taking the constant simplicial set which is right adjoint to ``$\pi_0$''.}.  Then, for any small (unenriched) category $D$ there is a bijection
  \begin{equation*}
    \mathrm{Fun}(h  C,D) \cong \mathrm{Fun}^\mathrm{enr}(C,\iota D),
  \end{equation*}
  natural in $D \in \mathbf{Cat}$, given by composing with the canonical functor $C \to \iota (h  C)$.  That is, $C \to \iota (h  C)$ is the initial enriched functor to an ordinary category (more precisely, to a topologically enriched category in which all morphism spaces have the discrete topology).
\end{lemma}
In particular, functors out of $\Cob_d$ may be translated to enriched functors out of $\mathcal{C}_d$.

The lemma is easily deduced from the corresponding statement about topological spaces: if $X$ is locally path connected, then $X \to \pi_0(X)$ is the initial continuous map to a space which has the discrete topology.  (Without local path connectedness, such a space might not exist!  See Exercise~\ref{adjoints}.)

It also makes sense to take the classifying space of a topologically enriched category $C$: in this case the set $N_p(C)$ of composable $p$-tuples of morphisms inherits a topology from the morphism spaces of $C$, and hence $NC$ is a \emph{simplicial topological space}.  The geometric realization
\begin{equation*}
  BC = |NC| = \bigg(\coprod_{p \geq 0} \Delta^p \times N_pC\bigg)/\sim
\end{equation*}
may be defined as for simplicial sets, except that $\Delta^p \times N_pC$ is given the product topology and the whole space $BC$ is given the quotient topology.  Actually, we shall work with the \emph{thick} geometric realization, i.e.\ take only the quotient involving the face maps of $NC$, not the degeneracy maps.  With this definition, the space $BC$ still makes sense when $C$ is a non-unital topologically enriched category: the nerve $NC$ is then a \emph{semi-simplicial space} (it has no degeneracy maps defined).
\begin{lemma}
  Let $C$ be a topologically enriched category, possibly non-unital, with locally path connected morphism spaces, and apply ``$B$'' to the canonical functor $C \to \iota(h  C)$; then the resulting map
  \begin{equation*}
    BC \to B(h  C)
  \end{equation*}
  is 2-connected.  Hence we get equivalences of groupoids
  \begin{equation*}
    hC[(hC)^{-1}] \xrightarrow{\simeq} \pi_1(B(h  C)) \xleftarrow{\simeq} \pi_1(BC).
  \end{equation*}
\end{lemma}

\begin{proof}
  We claim that the induced map of simplicial nerves has
  \begin{equation*}
    N_pC \to N_p(\iota h C)
  \end{equation*}
  a $(2-p)$-connected map for all $p \geq 0$.  For $p=0$ it is in fact a homeomorphism since the functor $C \to (\iota h C)$ is a bijection (the identity) on object sets.  For $p = 1$ it follows because the maps
  \begin{equation*}
    C(x,y) \to (\iota h C)(x,y) = \pi_0(C(x,y))
  \end{equation*}
  of morphisms spaces are 1-connected for all $x$ and $y$.  This is just the statement that for any locally path connected space the continuous map $X \to \pi_0(X)$ is 1-connected.  For $p=2$ we assert that $N_2C \to N_2(\iota h C)$ induces a surjection on $\pi_0$, but in fact it is even surjective on the point-set level.  For $p > 2$ the claim is vacuous.

  It is then a general fact that the geometric realization of a map of semi-simplicial spaces is $n$-connected if the map of $p$-simplices is $(n-p)$-connected.
\end{proof}

The point of the lemma (and of considering topologically enriched categories in this context) is that it may be easier to describe the whole homotopy type of $BC$ than to try to calculate $h C[(hC)^{-1}]$ ``directly''.  With more experience, it may also turn out that one is actually interested in the extra information contained in $C$ which is forgotten in $hC$: that extra information is the homotopy type of each path component of the morphism spaces $C(x,y)$.

\begin{corollary}
  The functors $\mathcal{C}_d \to h \mathcal{C}_d \simeq \Cob_d$ induce equivalences of groupoids
  \begin{equation*}
    \Cob_d[\Cob_d^{-1}] \xrightarrow{\simeq} \pi_1(B\Cob_d) \xleftarrow{\simeq} \pi_1(B\mathcal{C}_d)
  \end{equation*}
  and hence a functor $\gamma: \Cob_d \to \pi_1(B\mathcal{C}_d)$, universal among functors to groupoids, up to equivalence of categories.
\end{corollary}


\subsection{The main theorem}
\label{sec:main-thm}

It turns out there is a nice formula for the homotopy type of $B\mathcal{C}_d^V$, the classifying space of the cobordism category.  As we have seen above, such a formula implies a classification of invertible field theories (poor man's version; we return to symmetric monoidal structures later).
\begin{definition}
  Let $V$ be a finite dimensional real inner product space and $\Gr_d(V)$ denote the Grassmannian of $d$-planes in $\R \oplus V$.  Let $U_{d,V}$ be the universal $d$-dimensional vector bundle over $\mathrm{Gr}_d(V)$, and let
  \begin{equation*}
    U_{d,V}^\perp = \{(X,v) \in \mathrm{Gr}_d(V) \mid v \in X^\perp\}
  \end{equation*}
  be its orthogonal complement in $V$.  Finally, let $T_{d,V}$ be the \emph{Thom space} of $U_{d,V}^\perp$, i.e.\ the one-point compactification of its total space.  Let $S^V$ denote the one-point compactification of $V$, and let
  \begin{equation*}
    \Omega^V T_{d,\R \oplus V} = \mathrm{Map}_*(S^V,T_{d,\R \oplus V})
  \end{equation*}
  denote the space of basepoint preserving maps, in the compact-open topology.  (I.e., if $V = \R^n$ this is the $n$-fold based loop space.)
\end{definition}

We can now state the main theorem of these lectures.
\begin{theorem}\label{thm:gmtw}
  There is a weak equivalence
  \begin{equation}\label{eq:4}
    B\mathcal{C}_d^V \simeq \Omega^V T_{d,\R \oplus V}.
  \end{equation}
\end{theorem}

\begin{remark}
  We will not have time to discuss the proof of this theorem in these lectures, but let us make some historical remarks.  There is presumably not a direct map in either direction, and in any case the known proofs construct a \emph{zig-zag} of weak equivalences.  In suitable models it is compatible with taking colimit over finite-dimensional $V \subset \R^\infty$ and gives a weak equivalence
  \begin{equation*}
    B\mathcal{C}_d \simeq \colim_{V \subset \R^\infty} \Omega^V T_{d,\R \oplus V}.
  \end{equation*}
  The object on the right hand side is an \emph{infinite loop space}, and the corresponding spectrum is often called a Madsen--Tillmann spectrum and denoted $MTO(d)$.  In this form the weak equivalence was first proved in \cite[Main Theorem]{GMTW}.

  The proof for finite-dimensional $V$ was first outlined in \cite[Section 6]{galatius-2006}, and in more detail in \cite{GR-W}.  An outline is also contained in my lecture notes from the 2011 PCMI \cite{MR3114684}.
\end{remark}
\begin{corollary}
  There is\footnote{The proof of course gives a specified equivalence, or at least a zig-zag of such, not just existence of an equivalence.} an equivalence of groupoids
  \begin{equation*}
    h\mathcal{C}_d^V[(h\mathcal{C}_d^V)^{-1}] \simeq \pi_1\big(\Omega^V T_{d,\R \oplus V}\big).
  \end{equation*}
  Hence for $\dim(V) \gg d$ there is an equivalence of groupoids
  \begin{equation*}
    \Cob_d[\Cob_d^{-1}] \simeq \pi_1\big(\Omega^V T_{d,\R \oplus V}\big).
  \end{equation*}
\end{corollary}
The point is that it is now purely a homotopy theoretic question to determine the fundamental groupoid of $\Omega^V T_{d,\R \oplus V}$, which is closely related to determining $\pi_n(T_{d,\R \oplus V})$ and $\pi_{n+1}(T_{d,\R \oplus V})$ for $n  = \dim(V)$.  Let us illustrate this with a low-dimensional example.
\begin{example}\label{example:formerly-2-19}
  Let us take $d = 1$ and $V = \R$ so that objects and morphisms may be visualized in the plane.  Then $\Gr_1(\R^2) = \R P^1$ is homeomorphic to a circle, while the bundle $U^\perp_{1,\R}$ may be identified with the M\"obius strip.  Therefore $T_{1,\R}$ may be identified with $\R P^2$, so that $\Omega T_{1,\R}$ has two components, each of which is homotopy equivalent to $\Omega S^2$ which has fundamental group $\pi_2(S^2) = \Z$.  The fundamental groupoid therefore has two isomorphism classes of objects, and all objects have infinite cyclic automorphism groups.
\end{example}



\begin{remark}\label{rem:tangential}
  In fact the references given above prove slightly more, about a more general notion of cobordism category, where the objects and morphisms are equipped with  a \emph{tangential structure}, given by specifying a space $\Theta$ with a continuous action of the Lie group $\mathrm{GL}_d(\R)$.  One then defines a topologically enriched category $\mathcal{C}_{\Theta}^V$ in which morphisms are cobordisms $W \subset [0,t] \times V$ as before, equipped with a $\mathrm{GL}_d(\R)$-equivariant continuous map $\ell: \Fr(W) \to \Theta$, where $\Fr(W)$ denotes the frame bundle of $TW$.  Objects are $M \subset V$ equipped with $\mathrm{GL}_d(\R)$-equivariant maps from the frame bundle of $\R \oplus TM$.  A popular choice is $\Theta = \{\pm 1\}$ on which $\mathrm{GL}_d(\R)$ acts by multiplication by the sign of the determinant; then an equivariant map from the frame bundle is the same as an \emph{orientation}.  A version of the main theorem holds in this generality, the only difference is that $T_{d,\R \oplus V}$ should be replaced with a space $T_{\Theta,\R \oplus V}$ in which the Grassmannian $\Gr_d(\R \oplus V)$ has been replaced by a space whose points are pairs of a $d$-dimensional linear subspace $X \subset \R \oplus V$ together with an equivariant map $\ell: \Fr(X) \to \Theta$.

  Another interesting variant concerns manifolds equipped with an action of a finite group $G$.  In this case the embedded version of the cobordism category is defined for a finite dimensional inner-product space $V$, equipped with an isometric action of $G$; objects are then submanifolds of $V$ invariant under the action and morphisms are submanifolds of $[0,t] \times V$, invariant when $[0,t]$ is given the trivial action.  Under the assumption that $\dim(V^G) > d$, there is again a homotopy theoretic formula for the homotopy type of the classifying space of this ``equivariant bordism category''.  Passing to the colimit over subrepresentations $V$ of the direct sum of countably many copies of the regular representation $\R[G]$, one gets a theorem for a cobordism category of ``abstract'' $G$-equivariant manifolds.  See \cite{SzucsGalatius} for more details on all this, including precise statements.
\end{remark}

To finish this section, let us state the classification result in yet another form.
\begin{corollary}
  Let $D$ be a small groupoid, then there are equivalences of groupoids
  \begin{align*}
    \mathrm{Fun}(h \mathcal{C}_d^V,D) & \simeq \pi_1 \Map(B(h \mathcal{C}_d^V),BD)\\
                                      & \simeq \pi_1\Map(B\mathcal{C}_d^V,BD)\\
    & \simeq \pi_1\Map(\Omega^V T_{d,\R \oplus V},BD).
  \end{align*}
\end{corollary}
Imprecisely the poor man's field theories, namely functors from $\Cob_d$ to a groupoid $D$, are up to natural equivalence of functors the same as continuous maps from $\Omega^VT_{d,\R \oplus V}$ to $BD$, up to homotopy.

\section*{Lecture 3: More structure}\refstepcounter{section}
\addcontentsline{toc}{section}{Lecture 3: More structure}
\label{sec:more-structure}

\subsection{Symmetric monoidal structures}
\label{sec:symmetric-monoidal}

The ``abstract'' cobordism category $\Cob_d$ admits the structure of a \emph{symmetric monoidal category} with respect to disjoint union.  This monoidal structure is given by the \emph{data} of
\begin{itemize}
\item an object $\emptyset \in \Cob_d$,
\item a functor $\amalg: \Cob_d \times \Cob_d \to \Cob_d$,
\item isomorphisms $M_0 \amalg (M_1 \amalg M_2) \cong (M_0 \amalg M_1) \amalg M_2$ forming a natural transformation of functors $\Cob_d \times \Cob_d \times \Cob_d \to \Cob_d$,
\item isomorphisms $M_0 \amalg M_1 \cong M_1 \amalg M_0$ forming a natural transformation of functors $\Cob_d \times \Cob_d \to \Cob_d$,
\item isomorphisms $\emptyset \amalg M_0 \cong M_0$ forming a natural transformation of functors $\Cob_d \to \Cob_d$,
\end{itemize}
and this data satisfies the axioms of a symmetric monoidal structure.  See e.g.\ MacLane's book for the full list of axioms (or wikipedia, or many other places on the internet).

Another example of a symmetric monoidal category is that of complex vector spaces, or more generally modules over some commutative ring $R$, with symmetric monoidal structure given by $\otimes_R$, the trivial 1-dimensional module $1 = R$, all the usual isomorphisms $N_0 \otimes_R (N_1 \otimes_R N_2) \cong (N_0 \otimes_R N_1) \otimes_R N_2$, $R \otimes_R N \cong N$, etc.  Let us write $\mathrm{Mod}_R$ for the category of $R$-modules, equipped with this symmetric monoidal structure.  One version of the classical definition of a topological field theory then goes as follows.
\begin{definition}
  An ($R$-linear) \emph{topological field theory}\footnote{To distinguish from more elaborate notions, this is nowadays sometimes called a $(d-1,d)$-theory, because $\Cob_d$ involves manifolds of dimension $d-1$ and $d$.} is a symmetric monoidal functor $Z: \Cob_d \to \Mod_R$.

  More generally, for any symmetric monoidal category $D$, a $D$-valued topological field theory is a symmetric monoidal functor $Z: \Cob_d \to D$.
\end{definition}
Let us again skip the detailed definition, but recall that a (strong) symmetric monoidal functor is a functor together with specified isomorphisms $1 \to Z(\emptyset)$ and $Z(M_0) \otimes_R Z(M_1) \to Z(M_0 \amalg M_1)$, natural in $M_0$ and $M_1$, satisfying three axioms (having to do with the associators, unitors, and symmetries in the symmetric monoidal categories).

There are also versions of the definition in which objects and cobordisms are equipped with tangential structures, as in Remark~\ref{rem:tangential}.

\begin{definition}
  Let $D$ be a symmetric monoidal category.  An object $x$ is \emph{invertible} if there exists another object $\overline{x}$ such that $x \otimes \overline{x}$ is isomorphic to $1$.

  A field theory $Z: \Cob_d \to D$ is \emph{invertible} if it sends all object of $\Cob_d$ to invertible objects of $D$ and all morphisms to isomorphisms.
\end{definition}
In fact the condition on objects is automatic: the cylinder $[0,1] \times M$ as a cobordism from $M \amalg M$ to $\emptyset$ will be sent to an (iso)-morphism $Z(M) \otimes Z(M) \to Z(\emptyset) \cong 1$, so $Z(M)$ is invertible with inverse $Z(M)$.  (In tangentially structured versions, the inverse of $Z(M)$ can be taken as $Z(-M)$, the value of $Z$ on $M$ given the opposite tangential structure.)

There is a subcategory $D^\sim \subset D$ whose objects are the invertible objects and whose morphisms are the isomorphisms.  Then a $D$-valued invertible field theory is the same thing as a symmetric monoidal functor $\Cob_d \to D^\sim$.  Hence we may forget about the ambient $D$ and just consider $D^\sim$-valued topological field theories.  Thus an invertible topological field theory is essentially one whose target category is a rigid symmetric monoidal groupoid, defined as follows\footnote{For a general monoidal \emph{category} the adjective ``rigid'' is the requirement that all objects admit \emph{duals} but that is the same as admitting inverses up to isomorphism in the case where all morphisms are isomorphisms.  Another name for ``rigid symmetric monoidal groupoid'' is ``Picard groupoid''.}.
\begin{definition}
  A symmetric monoidal groupoid $D$ is \emph{rigid} if all objects are invertible.
\end{definition}

\subsection{Symmetric monoidal structure on the universal groupoid under $C$}

We have already studied how to classify functors from $\Cob_d$ into the underlying groupoid of $D$, ignoring the symmetric monoidal structure.  Indeed, they are ``the same thing'' as functors from $\Cob_d[\Cob_d^{-1}] \simeq \pi_1(B\Cob_d)$, which we identified with the fundamental groupoid of another space $\Omega^V T_{d,v}$ for $\dim(V) \gg d$.  This answer may be enhanced to a classification of invertible topological field theories provided we have good answers to the following two (vaguely worded) questions.
\begin{itemize}
\item Let $X$ be a space and $C = \pi_1(X)$ the fundamental groupoid.  What extra structure on $X$ is necessary to functorially induce a symmetric monoidal structure on $C$?
\item Let $C$ be a small category and $X = BC$.  What extra structure on $X$ does a symmetric monoidal structure on $C$ functorially induce?
\end{itemize}

There is a quite straightforward answer to these questions, based on the fact that $C \mapsto BC$ preserves products and takes natural transformations to homotopies, and conversely $X \mapsto \pi_1(X)$ takes maps to functors and homotopies to natural isomorphisms\footnote{A possible pitfall here is that $C \mapsto BC$ does not send vertical composition of natural transformations to concatenation of homotopies, it only does so up to homotopy relative to endpoints.  Therefore it is better to only consider \emph{homotopy classes} of homotopies: identify maps $[0,1] \times X \to Y$ that are homotopic relative to $\{0,1\} \times X$.  Conversely, such a homotopy class of homotopies is sufficient to induce a natural transformation of functors $\pi_1(X) \to \pi_1(Y)$.}.  Indeed, the following list of structure on $X$ will induce a symmetric monoidal structure on $\pi_1(X)$.
\begin{itemize}
\item a base point $1 \in X$,
\item a ``product'' map $\mu: X \times X\to X$,
\item a (homotopy class of a) homotopy $[0,1] \times X \times X \times X \to X$, between $(x_0,x_1,x_2) \mapsto \mu(x_0,\mu(x_1,x_2))$ and $(x_0,x_1,x_2) \mapsto \mu(\mu(x_0,x_1),x_2)$,
\item a (homotopy class of a) homotopy $[0,1] \times X \times X \to X$ between $(x_0,x_1) \mapsto \mu(x_0,x_1)$ and $(x_0,x_1) \mapsto \mu(x_1,x_0)$,
\item a (homotopy class of a) homotopy of maps $X \to X$ from the identity to the map $x \mapsto \mu(1,x)$,
\end{itemize}
subject to certain conditions, obtained by translating the axioms of a symmetric monoidal category (according to ``category'' $\mapsto$ ``space'', ``functor'' $\mapsto$ ``continuous map'', ``natural transformation'' $\mapsto$ ``homotopy class of a homotopy'', ``object'' $ \mapsto$ ``point'', ``vertical composition of natural transformations'' $\mapsto$ ``concatenation of homotopies'').  Conversely if $C$ is given a symmetric monoidal structure then $X = BC$ acquires this kind of structure.  Finally, these processes are inverse in the sense that $\gamma: C \to \pi_1(BC)$ is a symmetric monoidal functor, and a symmetric monoidal equivalence if all morphisms in $C$ are isomorphisms.

What seems to occur in nature is usually not the exact list above of data on a space $X$, but rather an ``$E_n$-structure''.  Let us briefly recall what that is, and explain why it is sufficient for inducing a symmetric monoidal structure on $\pi_1(X)$, for $n \geq 3$.

\subsection{Little disks}
\label{sec:little-disks}

Let $\mathcal{D}_n(k)$ be the space of embeddings $j:\{1, \dots, k\} \times \mathrm{int}(D^n) \to \mathrm{int}(D^n)$ of the form
\begin{equation*}
  (i,x) \mapsto v_i + t_i x,
\end{equation*}
for some $t_i \in (0,1)$ and $v_i \in \mathrm{int}(D^n)$, i.e.\ embeddings that on each open disk is given by a translation and a scaling by a positive number.  Use $v_i$ and $t_i$ to topologize this as a subspace
\begin{equation*}
  \mathcal{D}_n(k) \subset \big((0,1) \times \mathrm{int}(D^n)\big)^k.
\end{equation*}
These spaces come with the following structure
\begin{itemize}
\item an identity element $1 \in \mathcal{D}_n(1)$
\item an action of $S_k$ on $\mathcal{D}_n(k)$, by permuting input disks
\item composition maps for $k,k' \geq 0$ and $i \in \{1, \dots, k\}$
  \begin{equation*}
    \circ_i: \mathcal{D}_n(k) \times \mathcal{D}_n(k') \to \mathcal{D}_n(k + k' -1),
  \end{equation*}
  where $j \circ_i j'$ is defined by inserting the codomain $\Int(D^n)$ of $j'$ into the $i$th copy of $\Int(D^n)$ in the domain of $j$.
\end{itemize}
We will write $\mathcal{D}_n$ for the data of the spaces $\mathcal{D}_n(k)$ for $k \geq 0$ with their symmetric group actions, the identity $1 \in \mathcal{D}_n(1)$, and the compositions $\circ_i$.  It satisfies the associativity axioms
\begin{align*}
  (j \circ_i j') \circ_{i + i' -1} j'' & = j \circ_i (j' \circ_{i'} j'')\\
  (j \circ_{i'} j') \circ_{i + k' -1} j'' & = (j \circ_i j') \circ_{i'} j''
\end{align*}
for $j \in \mathcal{D}_n(k)$, $j' \in \mathcal{D}_n(k')$, and $j'' \in \mathcal{D}_n(k'')$, as well unit axioms
\begin{align*}
  1 \circ_1 j & = j\\
  j \circ_i 1 & = j,
\end{align*}
and an axiom expressing how composition interacts with the action of symmetric groups.  Together, these axioms precisely express that $\mathcal{D}_d$ is an \emph{operad}\footnote{See e.g.\ \cite[chapter 2]{FresseI} or \cite{Markl} for the axioms given here, and a discussion of how they are equivalent to May's axioms.} in spaces.  This operad is the \emph{little $n$-disk operad}.

Another example of an operad is the \emph{endomorphism operad} $\mathrm{End}_X$ of a space $X$.  Its $k$th space is
\begin{equation*}
  \mathrm{End}_X(k) = \Map(X^k,X),
\end{equation*}
the symmetric group acts by permuting factors in $X^k$, the identity $1 \in \mathrm{End}_X(1)$ is the identity map, and the composition $f \circ_i f'$ is defined by inserting the codomain $X$ of $f'$ into the $i$th factor in the domain of $f$, i.e.\
\begin{equation*}
  (f \circ_i f')(x_1, \dots, x_{k + k'-1}) = (x_1, \dots, x_{i-1},f(x_i,\dots, x_{i + k'-1}),x_i,\dots, x_{k + k'-1}).
\end{equation*}

\begin{definition}
  An \emph{algebra} for $\mathcal{D}_n$ is a space $X$ together with an operad map $\mathcal{D}_n \to \mathrm{End}_X$, i.e.\ continuous maps
  \begin{equation*}
    \mu_k: \mathcal{D}_n(k) \to \mathrm{End}_X(k) = \Map(X^k,X),
  \end{equation*}
  for each $k \in \Z_{\geq 0}$, preserving all the structure (symmetric group action, identity, compositions $\circ_i$).

  A \emph{morphism} of $\mathcal{D}_n$-algebras is a continuous map $f: X \to Y$ with the property that the diagrams
  \begin{equation*}
    \xymatrix{
      \mathcal{D}_n(k) \times X^k \ar[r]^-{\mu_k} \ar[d]_{\mathrm{id} \times f^k} & X \ar[d]^f\\
      \mathcal{D}_n(k) \times Y^k \ar[r]^-{\mu_k} & Y
    }
  \end{equation*}
  commute for all $k \in \Z_{\geq 0}$.
\end{definition}

\begin{proposition}\label{prop:associate-symmetric-to-En}
  For $n \geq 3$ there is a functorial way to associate a symmetric monoidal structure on $\pi_1(X)$ to a space $X$ with $E_n$ structure.
\end{proposition}

\begin{lemma}\label{lem:simple-connectivity}
  The space $\mathcal{D}_n(k)$ is homotopy equivalent to the \emph{ordered configuration space}
  \begin{equation*}
    \mathrm{Conf}_k(\R^n) = \{(x_1, \dots, x_k) \in (\R^n)^k \mid \text{$x_i \neq x_j$ for $i \neq j$}\}.
  \end{equation*}
  In particular it is simply connected for all $k$ when $n \geq 3$.
\end{lemma}
\begin{proof}[Proof sketch]
  Compose the inclusion $\mathcal{D}_n(k) \subset ((0,1) \times \mathrm{int}(D^n))^k$ with the projection to $(\mathrm{int}(D^n))^k$.  Its image is the set of $k$-tuples of distinct elements of $\mathrm{int}(D^n) \subset \R^n$, and the fiber over a $k$-tuple of distinct elements is a non-empty convex open subset of $(0,1)^k$.  After choosing a diffeomorphism $\mathrm{int}(D^n) \to \R^n$, we get a submersion $\mathcal{D}_n(k) \to \mathrm{Conf}_k(\R^n)$ with contractible fibers, which implies being a homotopy equivalence.

  Forgetting the last point gives a map $\mathrm{Conf}_{k+1}(\R^n) \to \mathrm{Conf}_k(\R^n)$ which is a fiber bundle whose fiber over $(x_1,\dots, x_k)$ is
  \begin{equation*}
    \R^n \setminus \{x_1, \dots, x_k\} \simeq \bigvee^k S^{n-1}.
  \end{equation*}
  This fiber is simply connected for $n \geq 3$, so in this case the long exact sequence of homotopy groups gives simple connectivity of $\mathrm{Conf}_k(\R^n)$ by induction on $k$.
\end{proof}

\begin{proof}[Proof of Proposition~\ref{prop:associate-symmetric-to-En}]
  Applying $\pi_1$ to the structure maps gives
  \begin{equation}\label{eq:7}
    \pi_1(\mathcal{D}_n(k)) \to \mathrm{Fun}(\pi_1(X) \times \dots \times \pi_1(X),\pi_1(X)),
  \end{equation}
  associating functors $(\pi_1(X))^k \to \pi_1(X)$ to objects of $\pi_1(\mathcal{D}_n(k))$ and natural isomorphisms to isomorphisms in $\pi_1(\mathcal{D}_n(k))$.
  
  Choose\footnote{In case the reader feels unease about this choice, let us remind them that defining e.g.\ the disjoint union functor $\amalg: \Sets \times \Sets \to \Sets$ also involves a choice, viz.\ of how to replace two sets by ``disjoint copies''.  Here and there, any two choices give naturally isomorphic functors.} an object of $\pi_1(\mathcal{D}_n(2))$, i.e.\ a point in the space $\mathcal{D}_n(2)$, and let $\otimes: \pi_1(X) \times \pi_1(X) \to \pi_1(X)$ be the image of $m$ under~(\ref{eq:7}).

  The two points $m \circ_1 m$ and $m \circ_2 m$ of $\mathcal{D}_n(3)$ are likely not equal, but there's a unique isomorphism between them in $\pi_1(\mathcal{D}_n(3))$ by Lemma~\ref{lem:simple-connectivity}.  Since the structure maps preserve $\circ_1$ and $\circ_2$, the image under~(\ref{eq:7}) of this isomorphism $m \circ_1 m \to m \circ_2 m$ in $\pi_1(\mathcal{D}_n(3))$ is a natural transformation between the two functors
  \begin{align*}
    \pi_1(X) \times \pi_1(X) \times \pi_1(X) & \to \pi_1(X)\\
    (x,y,z) \quad \quad & \mapsto
                          \begin{cases}
                            x \otimes (y \otimes z)\\
                            (x \otimes y) \otimes z,
                          \end{cases}
  \end{align*}
  and we take this to be the associator.

  The action of the transposition $(1 2) \in S_2$ on $m$ gives a new point $(1 2).m \in \mathcal{D}_n(2)$ but since $\mathcal{D}_n(2)$ is simply connected (in fact it is homotopy equivalent to $S^{n-1}$) there is a unique isomorphism $m \to (1 2).m$ in $\pi_1(\mathcal{D}_n(2))$.  Take the image under~(\ref{eq:7}) of this isomorphism as the symmetry.  Finally, as monoidal unit we take the image of the unique point in $\mathcal{D}_n(0)$
  under~(\ref{eq:7}).

  It remains to verify that the axioms are satisfied, but that is in fact easy.  For example, the pentagon axiom asserts that two natural transformations of functors $(\pi_1(X))^4 \to \pi_1(X)$ are equal.  The two natural transformations come from the two ways of using the associator to give a natural isomorphism $x \otimes(y \otimes (z \otimes w)) \to ((x \otimes y) \otimes z) \otimes w$.  Tracing definitions, we see that they are the image under~(\ref{eq:7}) of two particular isomorphisms $m \circ_2 (m\circ_2 m) \to (m \circ_1 m) \circ_1 m$ in  $\pi_1(\mathcal{D}_n(4))$.  But since $\mathcal{D}_n(4)$ is simply connected, \emph{any} two isomorphisms between two objects in its fundamental groupoid must be equal.  The other axioms of a symmetric monoidal category are proved by the same method.
\end{proof}

\begin{remark}
  A similar argument gives a \emph{monoidal structure} on the fundamental groupoid of an $E_1$ space, and a \emph{braided monoidal structure} on the fundamental groupoid of an $E_2$ space.
\end{remark}


\begin{example}\label{ex:loop-space}
  Finally, let us give the \emph{standard example} of an $E_n$ algebra, which was in fact the historical motivation for the development of operads.

  Any embedding $j: \{1, \dots, k\} \times \mathrm{int}(D^n) \to \mathrm{int}(D^n)$ may be one-point compactified to a based map in the other direction
  \begin{equation*}
    S^n \to S^n \vee \dots \vee S^n,
  \end{equation*}
  where we regard $S^n$ as the one-point compactification of $\Int(D^n)$.  Hence, if $Y$ is any based space we obtain a map
  \begin{equation*}
    j_*: \Omega^n Y \times \dots \Omega^n  Y\to \Omega^n Y,
  \end{equation*}
  depending continuously on $j \in \mathcal{D}_n(k)$.  It is not hard to verify that this construction makes $X = \Omega^n Y$ into a $\mathcal{D}_n$-algebra.  Moreover, it is \emph{group-like} (the monoidal category $\pi_1(X)$ is rigid).

  More surprisingly, it turns out that up to homotopy the space $Y$ may be reconstructed from $X$ and its $E_d$ structure, at least if $Y$ is $(n-1)$-connected.  There is a functorial procedure, May's \emph{two-sided bar construction}, which to an $E_n$-space $X$ associates a pointed space $B(S^n,\mathcal{D}_n,X)$.  For $X = \Omega^n Y$ there is a canonical weak equivalence between $B(S^n,\mathcal{D}_n,X)$ and the $(n-1)$-connected cover of $Y$.  Conversely, if $X$ is any group-like $E_n$ space and we define $Y = B(S^n,\mathcal{D}_n,X)$, then there is a zig-zag of weak equivalences which are $E_n$-maps, between $X$ and $\Omega^n Y$.
\end{example}

\subsection{Structure on embedded cobordism categories}
\label{sec:structure-on-embedded}

We can ``almost'' give the space $B\mathcal{C}_d^V$ an $E_n$ structure with $n = \dim(V)$, as follows.  Any embedding $j: \{1, \dots, k\} \times V \to V$ gives a continuous functor
\begin{equation}\label{eq:8}
  j_*: (\mathcal{C}_d^V)^k \to \mathcal{C}_d^V,
\end{equation}
by taking images of objects $M_1 \subset V$ and morphisms $W \subset [0,t] \times V$ under $j$, and hence a continuous map
\begin{equation*}
  B (j_*): (B\mathcal{C}_d^V)^k \to B \mathcal{C}_d^V.
\end{equation*}
Choose a diffeomorphism $\phi: V \approx \mathrm{int}(D^n)$ and use it to define maps
\begin{align*}
  \mathcal{D}_n(k) \hookrightarrow \Emb(\{1, \dots, k\} \times \mathrm{int}(D^n),\mathrm{int}(D^n)) &\xrightarrow{\phi} \Emb(\{1, \dots, k\} \times V,V)\\
  & \to \Map((B\mathcal{C}_d^V)^k,(B\mathcal{C}_d^V),
\end{align*}
where the last map is $j \mapsto B(j_*)$.  It is easy to verify that these maps preserve the operad structure: symmetric group actions, identity, and compositions $\circ_i$.  Unfortunately they are not continuous, because the functor~(\ref{eq:8}) does not ``depend continuously'' on $j$.  The problem is that $\mathcal{C}_d^V$ has a set (as opposed to a space) of objects, and hence $N_0\mathcal{C}_d^V$ has the discrete topology in which the action is not continuous.

The obvious fix, which is the approach we shall take, is to topologize $\mathrm{Ob}(\mathcal{C}_d^V)$.  In fact this set has a very natural topology, which does make the actions explained above be continuous.  The result is no longer a topologically enriched category, though.

\subsection{Topological categories}
\label{sec:topological-categories}

A \emph{topological category}, or a \emph{category internal to topological spaces} is a category $\overline{C}$ together with a topology on $\mathrm{Ob}(\overline{C})$ and on the total set $\mathrm{Mor}(\overline{C})$ of all morphisms, such that the four structure maps
\begin{align*}
  \text{identity}: \mathrm{Ob}(\overline{C}) & \to \mathrm{Mor}(\overline{C})\\
  \text{source, target}: \mathrm{Mor}(\overline{C}) & \to \mathrm{Ob}(\overline{C})\\
  \text{composition}: \mathrm{Mor}(\overline{C}) \times_{\Ob{\overline{C}}} \mathrm{Mor}(\overline{C}) & \to \mathrm{Mor}(\overline{C})
\end{align*}
are all continuous.  The nerve $N\overline{C}$ is then a simplicial topological space, with $N_0 \overline{C} = \mathrm{Ob}(\overline{C})$, $N_1 \overline{C} = \mathrm{Mor}(\overline{C})$, etc.  There is also a notion of non-unital topological categories, which have no identities; the nerve is then a semi-simplicial topological space (no degeneracy maps).

There is an easy way to extract a topologically enriched category $C$ out of a topological category $\overline{C}$, namely take the underlying set of $\mathrm{Ob}(\overline{C})$, and let $C(x,y) = \{f \in  \mathrm{Mor}(\overline{C}) \mid \text{source}(f) = x, \text{target}(f) = y\}$ with the subspace topology.  Composition and identity are as in $\overline{C}$.  There is an induced map of simplicial topological spaces $NC \to N\overline{C}$.  If we assume that $(d_0,d_1): N_1\overline{C} \to (N_0 \overline{C})^2$ is a Serre fibration, then it follows that the diagrams
\begin{equation}\label{eq:16}
  \begin{aligned}
  \xymatrix{
    N_p C \ar[r] \ar[d] & N_p \overline{C}\ar[d]\\
    (N_0 C)^{p+1} \ar[r] & (N_0 \overline{C})^{p+1}
  }
\end{aligned}
\end{equation}
are homotopy pullbacks for all $p$.  In this case
not much is lost by this operation, from a homotopy theoretic point of view\footnote{In the unital case a concise statement is the map $\Sing(NC) \to \Sing(N\overline{C})$ is a weak equivalence in the \emph{complete Segal space} model structure \cite{Rezk} on simplicial spaces.  Hence they may be viewed as two models for the same homotopy type of $(\infty,1)$-categories.}.  For example it follows from the Bousfield--Friedlander theorem that the induced map of classifying spaces $BC \to B\overline{C}$ is a weak equivalence.  It is also a bijection, so $\pi_1(BC) \to \pi_1(B\overline{C})$ is an isomorphism of groupoids.  If $hC$ is unital (but $C$ possibly not), then the canonical functors
\begin{equation}\label{eq:17}
  hC \xrightarrow{\gamma} \pi_1 BC \xrightarrow{\cong} \pi_1(B\overline{C})
\end{equation}
induces an isomorphism $(hC)[(hC)^{-1}] \cong \pi_1(B\overline{C})$.

We may define a version of the embedded cobordism category which is a topological category, as follows.  The set of objects of $\mathcal{C}_d^V$ is in bijection with
\begin{equation*}
  \coprod_M \Emb(M,V)/\Diff(M),
\end{equation*}
where the disjoint union ranges over closed $(d-1)$-manifolds, one in each diffeomorphism class.  Hence we may topologize it in the quotient topology from the $C^\infty$ topology on the embedding space.  Keeping the same underlying category as $\mathcal{C}_d^V$, there is a compatible way to topologize the total set of morphisms, where each connected component is topologized as a subspace of $\R_{>0} \times \Emb_\partial(W,[0,1] \times V)/\Diff_\partial(W)$.  See \cite{GMTW} for more details on such a definition.  Let us temporarily write $\mathscr{C}_d^V$ for the version as a topological category.

\subsection{Conclusion}
\label{sec:conclusion}

Since $\mathcal{C}_d^V$ and $\mathscr{C}_d^V$ have the same underlying (non-topologized) category, the discussion in \S\ref{sec:structure-on-embedded} applies equally well to the topological category $\mathscr{C}_d^V$, but now the functor
\begin{equation*}
  j_*: (\mathscr{C}_d^V)^k \to \mathscr{C}_d^V
\end{equation*}
depends continuously on $j \in \Emb(\{1, \dots, k\} \times V,V)$ and we get continuous maps
\begin{align*}
  \mathcal{D}_n(k) \hookrightarrow \Emb(\{1, \dots, k\} \times \mathrm{int}(D^n),\mathrm{int}(D^n)) &\xrightarrow{\phi} \Emb(\{1, \dots, k\} \times V,V)\\
  & \to \Map((B\mathscr{C}_d^V)^k,(B\mathscr{C}_d^V)),
\end{align*}
which satisfy the axioms of a map of operads.  Hence they make the space $B\mathscr{C}_d^V$ into an $E_n$-algebra (depending on the diffeomorphism $V \approx \mathrm{int}(D^n)$).  For $n = \dim(V) \geq 3$ the fundamental groupoid $\pi_1(B\mathscr{C}_d^V)$ therefore inherits a symmetric monoidal structure, from a choice of $m \in \mathcal{D}_n(2)$.  This choice also promotes $h\mathcal{C}_d^V$ to a symmetric monoidal category and the universal functor (as in~(\ref{eq:17}))
\begin{equation*}
  h\mathcal{C}_d^V \xrightarrow{\gamma} \pi_1(B\mathscr{C}_d^V)
\end{equation*}
to a symmetric monoidal functor, factoring over a symmetric monoidal equivalence $h\mathcal{C}_d^V[(h\mathcal{C}_d^V)^{-1}] \simeq \pi_1(B\mathscr{C}_d^V)$.

By Example~\ref{ex:loop-space}, the space $\Omega^V T_{d,\R \oplus V} \approx \Omega^n T_{d,\R \oplus V}$ also canonically has the structure of an $E_n$ algebra.  Hence it makes sense to ask whether the weak equivalence from the main theorem, between $\Omega^V T_{d,\R \oplus V}$ and $B \mathcal{C}_d^V \simeq B \mathscr{C}_d^V$ preserves this structure.  The proof of this equivalence goes through a rather long zig-zag of weak equivalences.  The original proofs of the main theorem did not discuss compatibility with $E_n$-structures, but more recently C.\ Schommer-Pries \cite[Chapter 4]{1712.08029} has shown how to find a zig-zag of maps, in our notation between  $B\mathscr{C}_d^V$ and $\Omega^V T_{d,\R \oplus V}$, which are both maps of $E_n$ spaces and weak equivalences of underlying spaces.  (Nguyen \cite{1505.03490} had shown this earlier in the limiting case $n \to \infty$, using Segal's $\Gamma$-spaces as a model for $E_\infty$ spaces.)  More concisely, he constructed a weak equivalence of $E_n$ spaces $B\mathscr{C}_d^V \simeq \Omega^V T_{d,\R \oplus V}$.

Combining everything, we get symmetric monoidal functors (for $\dim(V) \gg d$)
\begin{equation*}
  \Cob_d \xleftarrow{\simeq} h\mathcal{C}_d^V \to h\mathcal{C}_d^V [(h\mathcal{C}_d^V)^{-1}] \xrightarrow{\simeq} \pi_1(B\mathscr{C}_d^V) \xleftrightarrow{\simeq} \pi_1(\Omega^n T_{d,\R \oplus V}),
\end{equation*}
where the symmetric monoidal structure on $\pi_1(\Omega^n T_{d,\R \oplus V})$ arises from the $E_n$-structure on $\Omega^n T_{d,\R \oplus V}$, as explained in \S\ref{ex:loop-space}.

\begin{remark}\label{remark:postnikov-for-spaces}
  Let $X$ be a (based) space and let $n \geq 3$.  Let $X \to \tau_{\leq n} X$ denote the Postnikov truncation\footnote{See Exercise~\ref{Postnikov}, especially (d).}, and $\tau_{> n} X = \tau_{\geq n+1} X \to X$ its homotopy fiber over the base point.  Then the two maps $\tau_{\leq n+1}\tau_{\geq n} X \to \tau_{\leq n+1}X \leftarrow X$ induce symmetric monoidal equivalences
  \begin{equation*}
    \pi_1 \Omega^n(\tau_{\leq n+1}\tau_{\geq n} X) \xrightarrow{\simeq} \pi_1 \Omega^n (\tau_{\leq n+1} X ) \xleftarrow{\simeq} \pi_1 \Omega^n X,
  \end{equation*}
  so the symmetric monoidal groupoid $\pi_1(\Omega^n X)$ depends only on the homotopy type of the based space $\tau_{\leq n+1}\tau_{\geq n} X$.  This space is a ``two-stage Postnikov tower'', meaning that it has only two non-zero homotopy groups, namely $A = \pi_n(X,\ast)$ and $B = \pi_{n+1}(X,\ast)$.  Moreover, it is canonically homotopy equivalent to the homotopy fiber of a map of Eilenberg--MacLane spaces
  \begin{equation*}
    K(A,n) \xrightarrow{k} K(B,n+2),
  \end{equation*}
  whose homotopy class in $[K(A,n),K(B,n+2)] \cong H^{n+2}(K(A,n);B)$ is one of the \emph{$k$-invariants} of $X$.  It is a measure of how far $\tau_{\leq n+1}\tau_{\geq n} X$ is from being homotopy equivalent to the product $K(A,n) \times K(B,n+1)$.

  To summarize, the universal symmetric monoidal groupoid equipped with a symmetric monoidal functor from $\Cob_d$ is entirely encoded by the two groups $A = \pi_n(T_{d,\R^{n+1}},\ast)$ and $B = \pi_{n+1}(T_{d,\R^{n+1}},\ast)$ as well as the $k$-invariant\footnote{See also Exercise~\ref{exerc:k-invariant}.} connecting them.  A similar discussion applies to all the other flavors of cobordism categories (e.g.\ with extra tangential structure).
\end{remark}

With only a little more work, the classification of symmetric monoidal functors $\Cob_d \to D$ with values in a rigid symmetric monoidal groupoid $D$ can be turned entirely into homotopy theory of spaces.  The classical ``infinite loop space machines'' allow us to write the classifying space $BD$ as an iterated loop space
\begin{equation*}
  BD \simeq \Omega^n (B^{n+1} D),
\end{equation*}
for any $n$, and for $n \geq 3$ such a delooping in turn determines the symmetric monoidal structure on $D \simeq \pi_1(BD)$, as explained in
\S\ref{sec:little-disks}.  Using the main theorem with $V = \R^n$ for $n \gg d$ there is then an equivalence of groupoids
\begin{equation*}
  \mathrm{Fun}^\otimes(\Cob_d,D) \simeq \pi_1\Map_*(\tau_{\geq n} T_{d,\R^{n+1}},B^{n+1} D),
\end{equation*}
where the domain denotes the category of symmetric monoidal functors and monoidal natural transformations between them and the codomain is the fundamental groupoid of the space of pointed maps.

\section*{Lecture 4: Cobordism classes and characteristic classes}\refstepcounter{section}
\label{sec:lecture4}
\addcontentsline{toc}{section}{Lecture 4: Cobordism classes and characteristic classes}

As explained in Lecture 2, each path component of a morphism space in $\mathcal{C}_d$ may be interpreted as a classifying space for bundles of manifolds.  In this section we shall explain how the main result stated in Lecture 2 may be used to study the cohomology of these spaces, i.e.\ characteristic classes of bundles of manifolds.  This will make use of more of the homotopy type of $B\mathcal{C}_d$, beyond its fundamental groupoid.

\subsection{Stable homology of moduli spaces of surfaces}

Let us discuss the case $d=2$.  In this case, classification of surfaces asserts that a compact connected non-orientable smooth 2-manifold $W$ with $\partial W \approx \amalg_1^k S^1$ must be diffeomorphic to
\begin{equation*}
  W_{h,k}= (h \R P^2) \# (k D^2),
\end{equation*}
the connected sum of $h \in \Z_{\geq 1}$ copies of the real projective plane and $k \in \Z_{\geq 0}$ copies of the 2-disk.  For $h' \geq h$, $k \geq 1$, and $k' \geq 0$, there exist embeddings
\begin{equation*}
  W_{h,k} \stackrel{j}\hookrightarrow W_{h',k'}
\end{equation*}
and any such embedding gives rise to a continuous group homomorphism
\begin{equation*}
  j_*: \Diff_\partial(W_{h,k}) \to \Diff_\partial(W_{h',k'}),
\end{equation*}
defined by ``extension by the identity'': for $\varphi \in \Diff_\partial(W_{h,k})$ the diffeomorphism $j_*\varphi$ is given by $j \circ \varphi \circ j^{-1}$ on the image of $j$, and by the identity on the complement of the image of $j$.  The resulting map $j_* \varphi$ is again a diffeomorphism, because $\varphi$ acts as the identity near $\partial W_{h,k}$.  Taking induced map of classifying spaces gives a map of spaces $B\Diff(W_{h,k}) \to B\Diff(W_{h',k'})$ and hence a map on homology
\begin{equation}\label{eq:11}
  H_i(B\Diff(W_{h,k})) \to H_i(B\Diff(W_{h',k'})).
\end{equation}
The following \emph{homological stability} theorem was proved by Nathalie Wahl\footnote{Wahl's paper works with the discrete groups $\pi_0(\Diff_\partial(W_{h,k}))$, but that point of view is equivalent by the earlier result \cite[Th\'eor\`eme 1]{Gramain}.}.
\begin{theorem}[\cite{Wahl}]
  For any choice of embedding $j: W_{h,k} \hookrightarrow W_{h',k'}$ the homomorphism~(\ref{eq:11}) is an isomorphism for $h \gg i$.
\end{theorem}
In fact Wahl proved that $h \geq 4i + 5$ is sufficient, which was later improved  by Randal-Williams to $h \geq 3i + 3$ being sufficient.

For any connected, compact, and non-orientable surface $W$, the homology of $B\Diff_\partial(W)$ is therefore ``independent of $W$'' at least up to homological degree $\frac{h-3}{3}$, where $h$ is the largest natural number for which there exists an embedding $j: W_{h,1} \hookrightarrow W$.  (This number may be characterized homologically as the dimension of the cokernel of the natural map $H_1(\partial W;\F_2) \to H_1(W;\F_2)$.)
By the homological stability theorem the group $H_i(B\Diff_\partial(W);A)$ may be identified with the colimit of the direct system
\begin{equation*}
  \dots \to H_i(B\Diff_\partial(W_{h,1});A) \to H_i(B\Diff_\partial(W_{h+1,1});A) \to \dots,
\end{equation*}
which is manifestly independent\footnote{The diligent reader and writer ought to raise the question of whether the isomorphism is independent of choice of $j: W_{h,1} \hookrightarrow W$.  This is indeed the case: assuming the image of $j$ is contained in $W \setminus \partial W$, as may be arranged after isotoping $j$, it may be deduced from the classification of surfaces that the homomorphisms of diffeomorphism groups arising from two choices of $j$ are conjugate, and hence the maps of classifying spaces are homotopic.  By a similar argument the maps in the direct system are independent of choice of $W_{h,1} \hookrightarrow W_{h+1,1}$.} of $W$.

The range appearing in Wahl's theorem (improved by Randal-Williams) is known as the \emph{stable range}.  In this range the homology with coefficients in an abelian group $A$ may be identified with the \emph{stable homology} groups
\begin{equation}\label{eq:14}
  \colim_{h \to \infty} H_*(B\Diff_\partial(W_{h,1});A).
\end{equation}
In this lecture we shall discuss how cobordism categories may be used to determine the stable homology.  We shall not directly employ the ``field theory'' language, although the interested reader may interpret some of the constructions in this lecture in terms of invertible field theories.

\begin{remark}
  A similar homological stability result for \emph{oriented} compact connected 2-manifolds was established earlier by Harer in \cite{H} (see the survey \cite{MR3135444} for a geodesic proof).  Surfaces with other tangential structures were considered later by Randal-Williams in \cite{R-WResolution}.  Similar homological stability results are known for manifolds of dimension higher than 2, as long as the dimension is even, see \cite{GR-W3,GR-W4}.
  
  Several homological stability results for moduli spaces of odd dimensional manifolds have been established by Perlmutter \cite{PerlmutterProdSpheres,PerlmutterLink,PerlmutterStabHand}.
\end{remark}


\subsection{Cohomology of morphism spaces}
\label{sec:morphism-spaces}

The connection between cobordism categories and the question of determining $H^*(B\Diff_\partial(W);A)$ for a (non-orientable, connected, compact) surface $W$ is that the space $B\Diff_\partial(W)$ appears, up to homotopy, as a path component of a morphism space in $\mathcal{C}_2$.  Indeed, if we pick an embedding $\partial W \hookrightarrow V \subset \R^\infty$ and denote its image by $M$, then $M$ is an object of $\mathcal{C}_2$, and one of the path components of the space $\mathcal{C}_2(M,\emptyset)$ is weakly equivalent to $B\Diff_\partial(W)$, as explained in Section~\ref{sec:bundles}.  It also appears as a path component of $\mathcal{C}_2(M_0,M_1)$ if $M_0 \amalg M_1 \approx \partial W$.  We wish to deduce information about the cohomology of $B\Diff_\partial(W)$ from this appearance as a path component of morphism spaces in $\mathcal{C}_2$, and the knowledge provided by the main theorem cited above, which gives the weak equivalence
\begin{equation*}
  \Omega B\mathcal{C}_2 \simeq \Omega^\infty MTO(2) = \colim_V \Omega^V T_{2,V}.
\end{equation*}
This homotopy equivalence should be regarded as a calculation of the left hand side and therefore, at least implicitly, gives information about $\mathcal{C}_2$.  The right hand side is a priori more well understood---for example, standard arguments imply that each path component of this space has rational cohomology ring
\begin{equation}\label{eq:15}
  \Q[x_1, x_2, \dots],
\end{equation}
a polynomial ring on generators $x_i$ of cohomological degree $|x_i| = 4i$.

Let us first discuss what we are trying to do, in the abstract setting of a topologically enriched category $C$.
There is a continuous map
\begin{equation}\label{eq:12}
  C(x,y) \to \Omega_{x,y} BC,
\end{equation}
where the codomain denotes the space of paths $[0,1] \to BC$ starting at $x \in N_0 C \subset BC$ and ending at $y$, in the compact-open topology.  Standard arguments, generalizing the discussion in \S\ref{sec:categories-and-groupoids}, show that this map is a weak equivalence for all $x$ and $y$, if and only if $hC$ is a groupoid.  Without further assumptions the map~(\ref{eq:12}) may be very far from an equivalence; for example if $C$ is a partially ordered set the domain is either a singleton or empty, while the space $BC$ may have any homotopy type whatsoever.

The map~(\ref{eq:12}) is not quite a natural transformation in the strict sense: a morphism $f: y \to y'$ induces a map $C(x,y) \to C(x,y')$ by postcomposing with $f$, and a map $\Omega_{x,y} BC \to \Omega_{x,y'} BC$ by concatenating with the path given by the 1-simplex $f \in N_1(C)$, but the square
\begin{equation*}
  \xymatrix{
    C(x,y) \ar[r]\ar[d]^{f \circ -} & \Omega_{x,y} BC\ar[d]\\
    C(x,y') \ar[r] & \Omega_{x,y'} BC
  }
\end{equation*}
does not commute on the nose.  It does commute up to homotopy though, which is sufficient to make
\begin{equation*}
  H_*(C(x,y);A) \to H_*(\Omega_{x,y} BC;A)
\end{equation*}
a natural transformation of functors of $(x,y) \in (hC)^\mathrm{op} \times (hC)$, for any abelian group $A$.  Since concatenation with a path is a homotopy equivalence, the codomain will send any morphism $y \to y'$ to an isomorphism $H_*(\Omega_{x,y} BC;A) \to H_*(\Omega_{x,y'} BC;A)$, and similarly for $x' \to x$.

If we have a filtered category $J$, an object $j_0 \in J$, and a functor $y: J \to C$, we therefore get an isomorphism
\begin{equation*}
  H_*(\Omega_{x,y(j_0)} BC;A) \xrightarrow {\cong} \colim_{j \in J} H_*(\Omega_{x,y(j)} BC;A)
\end{equation*}
and an induced map
\begin{equation}\label{eq:18}
  \colim_{j \in J} H_*(C(x,y(j));A) \to \colim_{j \in J} H_*(\Omega_{x,y(j)} BC;A) \cong H_*(\Omega_{x,y(j_0)} BC;A)
\end{equation}
The following result is a version\footnote{An arguably more conceptual point of view on the map~(\ref{eq:12}) involves the \emph{derived} version of the universal functor $\gamma: C \to C[C^{-1}]$ mentioned earlier for plain categories.  Explicit models in simplicially enriched categories were given by Dwyer and Kan, \cite{MR579087,MR578563}, who also established that morphism spaces in the derived localization are path spaces in the classifying space.  Another model is the localization of quasi-categories \cite[\S5.2.7]{MR2522659}.  Either of these settings promotes~(\ref{eq:12}) to a natural transformation of (representable) functors of $y$.  From this derived point of view, it may also be more natural to replace the ``colimits of homology'' by ``homotopy colimits of chains'' in Proposition~\ref{prop:group-completion}; in such a formulation the assumption that $J$ be filtered may be weakened to $BJ$ being contractible.} of the ``group completion'' theorem for categories.
\begin{proposition}\label{prop:group-completion}
  Let $C$ and $y: J \to C$ be as above, let $A$ be an abelian group, and assume the functor
  \begin{equation}\label{eq:19}
    \begin{aligned}
    hC^\mathrm{op} &\to \mathrm{Ab}\\
    x &\mapsto \colim_{j \in J} H_*(C(x,y(j));A)
  \end{aligned}
  \end{equation}
  sends all morphisms to isomorphisms.  Then the map~(\ref{eq:18}) is an isomorphism, in fact a natural isomorphism of functors of $x \in hC^\mathrm{op}$.
\end{proposition}

This abstract result (with $A = \Z$) is quite close to what we need: the stable homology groups~(\ref{eq:14}) are (filtered)  colimits of homology groups of path components of morphism spaces in the cobordism category $\mathcal{C}_2$, and we understand the space $\Omega B\mathcal{C}_2 \simeq \Omega^\infty MTO(2)$ well, including the (rational co)homology of its path components~(\ref{eq:15}).  Unfortunately the proposition does not apply on the nose to $C = \mathcal{C}_2$, because the many path components of $\mathcal{C}_2(x,y)$ corresponding to disconnected cobordisms prevents the  assumption of the proposition from being satisfied.

\begin{remark}
  Even though we shall not need it in the rest of this lecture, let us make a brief remark about how Proposition~\ref{prop:group-completion} fits into the ``field theory'' paradigm, under the extra assumption that $C$ is symmetric monoidal, $y(j_0) = 1$ is the monoidal unit of $C$, and $A$ is a commutative ring.  Then the functor
  \begin{align*}
    hC^\mathrm{op} &\to \text{graded $A$-modules}\\
    x &\mapsto \colim_{j \in J} H_*(C(x,y(j);A) \cong H_*(\Omega_{x,1} BC;A)
  \end{align*}
  may be promoted to a lax symmetric monoidal functor, where the codomain is symmetric monoidal with respect to $\otimes_A$ and the usual symmetry $x \otimes y \mapsto (-1)^{|x||y|} y \otimes x$.  The monoidality of the functor arises from applying the lax monoidal functor $H_*(-;A)$ to
  \begin{equation*}
    \Omega_{x,1} (BC) \times \Omega_{x',1}(BC) \xrightarrow{\cong} \Omega_{(x,x'),(1,1)} (BC \times BC) \to \Omega_{x \otimes x',1 \otimes 1} BC \xrightarrow{\simeq} \Omega_{x \otimes x',1} BC,
  \end{equation*}
  where the middle arrow is $\Omega_{(x,x'),(1,1)} (B \otimes)$ and the last arrow is concatenating with the path corresponding to the isomorphism $1 \otimes 1 \to 1$.
  
  Since the resulting natural transformation 
  \begin{equation*}
H_*(\Omega_{x,1} BC;A) \otimes_A H_*(\Omega_{x',1} BC;A) \to H_*(\Omega_{x \otimes x',1} BC;A)
\end{equation*}
 and also the
 unit $A \to H_*(\Omega_1 BC;A)$ are likely not isomorphisms,  we  have only constructed a \emph{lax} monoidal functor.  But if $BC$ is path connected it lifts to a (strong) monoidal functor into the category of modules over $H_*(\Omega_1 BC;A)$, endowed with tensor product over that graded-commutative ring.  The module structure is induced by the map of spaces $\Omega_{x,1} BC \times \Omega_1 BC \to \Omega_{x,1} BC$ defined by concatenation of paths, and the monoidality is defined as before but factoring over the tensor product over $H_*(\Omega_1 BC;A)$.  It is not hard to check that this is well defined and gives a strong monoidal functor if $BC$ is path connected\footnote{A more elaborate construction also works for disconnected $BC$, at least when the commutative monoid structure on $\pi_0(BC)$ induced by $\otimes$ is a group.  Then the functor should take values in a symmetric monoidal category of $\pi_0(BC)$-graded $H_*(\Omega_1 BC;A)$-modules.  The symmetry isomorphism involves a ``generalized sign'' which is an anti-symmetric function $c: \pi_0(BC) \times \pi_0(BC) \to \pi_1(BC,1) \subset H_*(\Omega_1 BC;A)^\times$, and the associator involves a normalized 3-cocycle $h: (\pi_0(BC))^3 \to \pi_1(BC,1)$.}.
\end{remark}

\subsection{Cobordisms with connectivity restrictions}
\label{sec:restrictions}

It is a popular endeavour to study and/or define field theories whose domain is \emph{larger} than $\Cob_d$.  The general flavor is to allow objects to have boundary as well: instead of only allowing closed $(d-1)$-manifolds $M$ we could allow $M$ to be a cobordism between $(d-2)$-manifolds.  Cobordisms $M_0 \leadsto M_1$ would then be $d$-manifolds with corners.
If we stop there, the corresponding notion of field theory is a ``$(d-2,d-1,d)$-theory''.  Going for maximal generality, there is a notion of \emph{fully extended} field theory.  See \cite{1712.08029} for a classification result for invertible field theories in this setting or \cite{Lurie} for a non-invertible classification result.

In this section we shall do the \emph{opposite}: study functors defined on something \emph{smaller} than $\mathcal{C}_d$.  Such subcategories turn out to be very useful for understanding homology of components of the morphism spaces $\mathcal{C}_d(M_0,M_1)$ which, as we have seen, have the homotopy type of $B\Diff_\partial(W)$.  For $d=2$ the goal is, roughly speaking, to find a subcategory $C \subset \mathcal{C}_2$ which is large enough to contain all the $B\Diff_\partial(W_{h,k})$ and small enough that Proposition~\ref{prop:group-completion} applies to it.

\begin{definition}
  Let $d \geq 0$ and $k \geq -1$.  For objects $M_0, M_1 \in \Cob_d$, let
  \begin{equation*}
    \Cob_d^k(M_0,M_1) \subset \Cob_d(M_0,M_1)
  \end{equation*}
  be the subset consisting of (diffeomorphism classes of abstract) cobordisms $W$, such that the pair $(W,M_1)$ is $k$-connected.  For $k = -1$ this is no condition, for $k = 0$ it means that $\pi_0(M_1) \to \pi_0(W)$ is surjective, and for $k > 0$ it additionally means that the relative homotopy groups $\pi_i(W,M_1,x)$ vanish for all $i \leq n$ and all basepoints $x \in M_1$.
\end{definition}
As you will show in Exercise~\ref{exercise:cob-cat-with-connectivity}, the connectivity property is preserved under composition of cobordisms, so that we have defined a subcategory $\Cob_d^k \subset \Cob_d$.

Since we have functors
\begin{equation*}
  h\mathcal{C}_d^V \to h\mathcal{C}_d \xrightarrow{\simeq} \Cob_d,
\end{equation*}
we get corresponding subcategories of $\mathcal{C}_d$ and $\mathcal{C}_d^V$ by taking the inverse image.  Similarly in the presence of a tangential structure $\Theta$.  Following \cite{GR-W2} we shall write $\mathcal{C}_d^k \subset \mathcal{C}_d$ for the inverse image of $\Cob_d^k \subset \Cob_d$.
\begin{lemma}
  $\Cob_d^k \subset \Cob_d$ is a symmetric monoidal subcategory.  Relatedly, for $n = \dim(V)$, the $E_n$ structure on $\mathscr{C}_d^V$ restricts to an action on the topological subcategory defined as the inverse image of $\Cob_d^k$.
\end{lemma}
\begin{proof}
  Let $W \in \Cob_d(M_0,M_1)$ and $W' \in \Cob_d(M_0',M_1')$.  If $(W,M_1)$ and $(W',M'_1)$ are both $k$-connected, then the inclusion
  \begin{align*}
    M_1 \amalg M'_1 \hookrightarrow W \amalg W'
  \end{align*}
  is also $k$-connected.
\end{proof}
The case $k = 0$ is especially easy to visualize: the condition on a cobordism $W \in \mathcal{C}_d(M_0,M_1)$ is then that each path component of $W$ must have non-empty outgoing boundary.  (Symmetric monoidal) functors from $\mathcal{C}_d^k$, or from other subcategories of cobordism categories, are sometimes called \emph{restricted field theories}.  
\begin{example}
  For $d=1$ the subcategory $\Cob_1^0 \subset \Cob_1$ has a very simple structure.  Its objects are finite sets, while morphisms from a finite set $M_0$ to another finite set $M_1$ may be viewed as injections $M_0 \hookrightarrow M_1$ together with matchings (equivalence relations where each equivalence class has cardinality 2) on the complement of the image.  Symmetric monoidal functors $\Cob_1^0 \to \Mod_R$ are therefore uniquely given, up to natural isomorphism, by an object $X \in \Mod_R$ together with an element $\omega \in X \otimes_R X$ that is fixed under swapping the factors.  For most choices of $X$ and $\omega$ it is \emph{not} possible to extend to a symmetric monoidal functor on all of $\Cob_1$.  In fact, it can be shown\footnote{This is the 1-dimensional \emph{cobordism hypothesis}, \cite{BaezDolan,Lurie}.} that such an extension exists if and only if $X$ is finitely generated and projective, and the adjoint $\omega: X \to X^\vee$ is an isomorphism.  
\end{example}

Similar examples may be given of symmetric monoidal functors $\Cob_2^0 \to \Mod_R$ which do not extend to $\Cob_2$.


The inclusion $\mathcal{C}_d^k \hookrightarrow \mathcal{C}_d$ induces a map of spaces
\begin{equation}\label{eq:5}
  B\mathcal{C}_d^k \hookrightarrow B\mathcal{C}_d
\end{equation}
which, surprisingly, sometimes turns out to be a weak equivalence.
It is easy to see that~(\ref{eq:5}) can't always be a weak equivalence though.  For example if $k > d$, then we only allow cobordisms $W \in \Cob_d(M_0,M_1)$ for which the inclusion $M_1 \to W$ is a homotopy equivalence.  This prevents the map~(\ref{eq:5}) from inducing a bijection on $\pi_0$, much less being a weak equivalence.  It is perhaps more surprising that it is \emph{ever} a weak equivalence for $k \geq 0$.  The following result was proved in \cite[\S3]{GR-W2} (see also \cite[\S6]{GMTW} for the case $d=2$).
\begin{theorem}\label{thm:surgery-on-morphisms}
  Let $k \leq \frac d2 -1$.  Then the map~(\ref{eq:5}) is a weak equivalence.

  More generally let $\Theta$ be a space with $\mathrm{GL}_d(\R)$ action, let $\mathcal{C}_{\Theta}$ be the corresponding cobordism category of manifolds with $\Theta$-structure, as in Remark~\ref{rem:tangential}, and let $\mathcal{C}_\Theta^k$ be the inverse image of $\Cob_d^k$.  Then the inclusion induces a weak equivalence
  \begin{equation}\label{eq:5bis}
    B\mathcal{C}_\Theta^k \hookrightarrow B\mathcal{C}_\Theta
  \end{equation}
  for $k \leq \frac d2 -1$.
\end{theorem}

\subsection{Cohomology of morphism spaces}
\label{sec:cohomology}

Let us now return to the topic of characteristic classes of smooth manifold bundles.  For concreteness we will again first explain the case of non-orientable surfaces.

This is in fact just a matter of collecting together all the ingredients outlined above.  Indeed, let us we apply Proposition~\ref{prop:group-completion} to the topologically enriched category $C = \mathcal{C}_2$ with $J = (\N,\leq)$ and $y: J \to \mathcal{C}_2$ given on objects by letting each $y(n)$ be the image of some embedding $S^1 \hookrightarrow V$ and on the generating morphisms $n < n+1$ as an embedded cobordism diffeomorphic to $([n,n+1] \times S^1)\# \R P^2$.  Then Wahl's theorem implies that the assumption of Proposition~\ref{prop:group-completion} is satisfied, and the proposition implies an isomorphism between the cohomology of $B\Diff_{\partial}(W_{h,k})$ in the stable range, and the cohomology of a path component of $\Omega^\infty MTO(2)$.  We deduce a map of graded $\Q$-algebras
\begin{equation*}
  \Q[x_1,x_2,\dots] \to H^*(B\Diff_\partial(W_{h,k});\Q),
\end{equation*}
which is an isomorphism in the stable range (cohomological degrees $\leq (h-3)/3$, by Randal-Williams' improved range).

\begin{remark}
  The corresponding statement for oriented surfaces is the celebrated Madsen--Weiss theorem \cite{MW}, establishing the Mumford conjecture.  It may be proved using cobordism categories of oriented surfaces by an argument similar to the one outlined above, replacing Wahl's theorem with Harer's theorem.  This method of using cobordism categories to prove Madsen and Weiss' theorem was first used in \cite{GMTW}.  The insight that restricted cobordism categories are useful for studying stable homology goes back to Ulrike Tillmann's paper \cite{Tillmann}, in the surface case.
\end{remark}
\begin{remark}
  The homological stability results established in \cite{GR-W3,GR-W4} for manifolds of even dimension $d = 2n$ may be used in a similar way to Harer's and Wahl's results, and the restricted cobordism categories $\mathcal{C}_{2n}^{n-1}$ may be used in a similar way, to determine the stable homology.  See op.\ cit.\ for precise statements and detailed proofs.
  
  The situation for odd dimensional manifolds is currently more mysterious.  While Perlmutter's homological stability results \cite{PerlmutterProdSpheres,PerlmutterLink,PerlmutterStabHand} give explicit stable ranges, we do not currently know the corresponding stable homology.  The statement in Theorem~\ref{thm:surgery-on-morphisms} is not quite strong enough in the odd dimensional case $d = 2n+1$: Proposition~\ref{prop:group-completion} does not seem to apply to $\mathcal{C}_{2n+1}^{n-1}$.
\end{remark}


\markboth{\exauthors}{Exercises for \bettershorttitle}

\section*{Exercises (by A.\ Debray, S.\ Galatius, M.\ Palmer)}\refstepcounter{section}
\addcontentsline{toc}{section}{Exercises (by A.\ Debray, S.\ Galatius, M.\ Palmer)}
\label{sec:exercises}

The four lectures at the summer school were accompanied by three exercise sessions, TA'd by Arun Debray\footnote{University of Texas, Austin.  \href{mailto:a.debray@math.utexas.edu}{a.debray@math.utexas.edu}} and Martin Palmer\footnote{University of Bonn, Germany.  Current affiliation: Mathematical Institute of the Romanian Academy.  \href{mpanghel@imar.ro}{mpanghel@imar.ro}}.  The exercises below are edited versions of the problem sets: a few have been omitted and some have been rewritten for clarity.  We thank the participants in the exercise sessions for their feedback on the original versions.

We have mostly preserved the ordering used at the summer school (but added a few suggestions for doing certain exercises out of order).  Exercises labeled with an asterisk are likely more challenging.

\subsection{Exercise set 1}

\begin{exercise}\label{cob-groups}
Recall that $\mathfrak{N}_d$ denotes the abelian group of smooth, closed $d$-manifolds up to cobordism, with respect to the operation of disjoint union.
\begin{itemize}
\item[(a)] Show that the Klein bottle is nullbordant.
\item[(b)] Prove that the disjoint union of two manifolds of positive dimension is cobordant to their connected sum.  (The connected sum depends on a choice of chart in each manifold, but the claim is true for any choice.)  
\item[(c)] Deduce that every cobordism class except $0 \in \mathfrak{N}_0$ contains a connected manifold.
\item[(d)] Prove, without invoking the theorem of Thom (but you may use the classification of surfaces), that the abelian groups $\mathfrak{N}_0$, $\mathfrak{N}_1$ and $\mathfrak{N}_2$ are isomorphic to $\bZ/2\bZ$, $0$ and $\bZ/2\bZ$, respectively.  
\item[] (Hint: to show that $\mathfrak{N}_2 \neq 0$, check that Euler characteristic modulo 2 defines a well defined surjection $\mathfrak{N}_2 \to \bZ/2\bZ$.)
\end{itemize}
Write $\Omega_d$ for the analogous oriented cobordism group: the abelian group of smooth, closed, \emph{oriented} $d$-manifolds up to \emph{oriented} cobordism, with respect to the operation of disjoint union.
\begin{itemize}
\item[(e)] Prove that the abelian groups $\Omega_0$, $\Omega_1$ and $\Omega_2$ are isomorphic to $\bZ$, $0$ and $0$ respectively.
\end{itemize}
\end{exercise}

\begin{exercise}\label{eift}
In this exercise we construct two elementary ``invertible field theories'' (in the naive sense, ignoring symmetric monoidal structures). Recall that a (naive) invertible field theory is a functor from the abstract cobordism category $\mathrm{Cob}_d$ to a groupoid; we will construct two examples where the target is a group.
\begin{itemize}
\item[(a)] Construct a functor
\[
F_d \colon \mathrm{Cob}_d \longrightarrow \mathfrak{N}_d
\]
(where the group $\mathfrak{N}_d$ is considered as a category with one object) such that the restriction to $\mathrm{End}_{\mathrm{Cob}_d}(\varnothing)$ sends a diffeomorphism class of $d$-manifolds to its cobordism class.
\item[] (\emph{Hint}: note that $\mathrm{Cob}_d$ is a disjoint union (coproduct) of subcategories indexed by the elements of $\mathfrak{N}_{d-1}$, so it will suffice to define $F$ on the full subcategory of $\mathrm{Cob}_d$ on nullbordant $(d-1)$-manifolds.)
\item[(b)] Construct a functor
\[
E_d \colon \mathrm{Cob}_d \longrightarrow \bZ
\]
(where the group $\bZ$ is considered as a groupoid with one object) such that the restriction to $\mathrm{End}_{\mathrm{Cob}_d}(\varnothing)$ sends a diffeomorphism class of $d$-manifolds to its Euler characteristic.
\item[(c)] $(*)$ In the case $d=2$, using the classification of surfaces, show that the functor
\[
\mathrm{Cob}_2[\mathrm{Cob}_2^{-1}] \longrightarrow \bZ,
\]
induced by $E_2$, is an equivalence.\footnote{This is not so easy to do by hand, in fact it is a research paper \cite[Theorem 3.7]{JuerTillmann}.  This is a more reasonable exercise if you do Exercise~\ref{exerc:5.3.1} first.}
\end{itemize}
\end{exercise}



\begin{exercise}
  In this exercise as in the text, the fundamental groupoid of an (unbased) space $X$ is denoted $\pi_1(X)$.
\begin{itemize}
\item[(a)] Let $C$ be the poset of all non-empty, proper subsets of $\{0, 1, 2, 3\}$, considered as a category. Prove that $BC$ is homeomorphic to $S^2$ and hence that $\pi_1(BC)$ is equivalent to a trivial category.
\item[(b)] Let $C$ be the category with exactly two objects $a$, $b$ and exactly two non-identity morphisms $f$, $g$, which both have source $a$ and target $b$. Prove that $BC$ is homeomorphic to $S^1$ and describe explicitly an equivalence of groupoids $\pi_1(BC) \to \bZ$.
\item[(c)] Combining the above ideas, find a finite category $C$ such that the abelian group $\pi_2(BC,x)$ is infinitely-generated for any object $x$.
\end{itemize}
\end{exercise}

\begin{exercise}$(*)$  Another variant of the cobordism group is the \emph{group of homotopy $d$-spheres} $\Theta_d$.  To define this, let us first denote by $\mathcal{M}_d$ the abelian monoid of isomorphism classes of smooth, closed, connected and oriented $d$-manifolds, up to oriented diffeomorphism, under the operation of connected sum (which is well defined up to oriented diffeomorphism).  Let $\mathcal{S}_d$ be the subset of diffeomorphism classes of $d$-manifolds that are homotopy equivalent to the sphere $S^d$, which we call \emph{homotopy spheres}.
\begin{itemize}
\item[(a)] Prove that, if $M$ and $N$ are homotopy spheres, so is their connected sum.
\end{itemize}
Hence $\mathcal{S}_d$ is a submonoid. Two manifolds $M,N \in \mathcal{M}_d$ are called \emph{$h$-cobordant} if there is a cobordism $W$ between them such that the inclusions $c_{\mathrm{in}}$ and $c_{\mathrm{out}}$ are both homotopy equivalences.
\begin{itemize}
\item[(b)] Prove that $h$-cobordism $\sim_h$ induces an equivalence relation on $\mathcal{S}_d$.
\item[(c)] Prove that this equivalence relation is compatible with the connected sum operation.
\end{itemize}
Hence we have a well-defined quotient monoid $\Theta_d \coloneqq \mathcal{S}_d/{\sim_h}$ with respect to the operation of connected sum.
\begin{itemize}
\item[(d)] Prove that $\Theta_d$ is a group.\footnote{This is Theorem 1.1 of the influential paper \cite{KervaireMilnor}, which also studies the group $\Theta_d$ in depth.  It is always a finite abelian group, and the smallest $d$ for which it is non-trivial is $\Theta_7 \cong \bZ/28\bZ$.  In fact, except possibly for $d=4$, the quotient $\mathcal{S}_d \to \Theta_d$ is an isomorphism, in which case $\mathcal{S}_d$ is also a (finite) group. }
\end{itemize}
\end{exercise}

\subsection{Exercise set 2}

\begin{exercise}\label{diffeomorphism}
  This exercise studies the difference between being diffeomorphic and being diffeomorphic as cobordisms.
  \begin{itemize}
  \item[(a)] Give examples of the following data: two closed $(d-1)$-manifolds $M_0$ and $M_1$ and two cobordisms $(W,\cin,\cout)$ and $(W',\cin',\cout')$, both between $M_0$ and $M_1$, which are \emph{not} diffeomorphic as cobordisms, yet $W$ is diffeomorphic to $W'$.  (Hint: start with $d=1$.)
  \item[(b)] Let $M = M_0 \amalg M_1$ be the disjoint union.  Construct a bijection $\Cob_d(M_0,M_1) \cong \Cob_d(M,\emptyset)$, by sending a triple $(W,\cin,\cout)$ to a cobordism of the form $(W,?,\emptyset)$.  To study the difference between being diffeomorphic and being diffeomorphic as cobordisms, it therefore suffices to consider the case $M_1 = \emptyset$.
  \item[(c)] Let $\Diff(M)$ denote the group of diffeomorphisms $M \to M$.  Construct an action of $\Diff(M)$ on the set $\Cob_d(M,\emptyset)$ by ``reparametrizing the boundary''.  Show that two cobordisms $M \leadsto \emptyset$ are diffeomorphic (as abstract manifolds) if and only if they represent elements of $\Cob_d(M,\emptyset)$ which are in the same $\Diff(M)$-orbit.
  \end{itemize}
\end{exercise}

\begin{exercise}\label{groupoid}
Let $D$ be the groupoid defined abstractly as $\mathrm{Ob}(D) = \bZ/2\bZ$, morphism sets
\begin{equation*}
D(a,b) = \begin{cases}
\bZ & a=b \\
\varnothing & a \neq b,
\end{cases}
\end{equation*}
and composition given by addition in $\bZ$. As explained in Example~\ref{example:formerly-2-19} in the notes, this groupoid is equivalent to the fundamental groupoid of $\Omega T_{1,\mathbb{R}^2} \cong \Omega \mathbb{R}\bP^2$, so it should also be the target of a universal functor from $h \mathcal{C}_1^{\mathbb{R}}$ to a discrete groupoid.  The goal of this exercise is to verify this directly, by geometric constructions.

\begin{itemize}
\item[(a)] Construct a functor
\begin{equation*}
f \colon h \mathcal{C}_1^{\mathbb{R}} \longrightarrow D,
\end{equation*}
defined by sending an object (finite subset of $\mathbb{R}$) to its cardinality modulo $2$, and a morphism $W \subset [0,t] \times \mathbb{R}$ from $M_0 \subset \mathbb{R}$ to $M_1 \subset \mathbb{R}$ to the integer
\begin{equation*}
\chi(X) - \chi(X \cap (\{0\} \times \mathbb{R})),
\end{equation*}
where $X \subset [0,t] \times \mathbb{R}$ is the union of components of the complement $([0,t] \times \mathbb{R}) \smallsetminus W$ obtained by coloring them green or red in an alternating way, starting with green for the unbounded component in the $[0,t] \times \{-\infty\}$ direction, and setting $X$ to be the union of the red components.
\item[] More rigorously: a component $C$ of $([0,t] \times \mathbb{R}) \smallsetminus W$ is included in the union $X$ if and only if a ray starting from the interior of $C$, intersecting $W$ transversely and asymptotically equal to $s \mapsto (t/2,-s)$, has an odd number of intersections with $W$. For example, in the following picture, the shaded region is $X$ and its boundary is $W$:
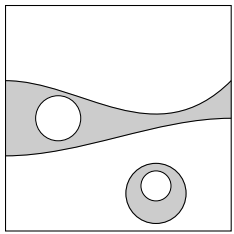
\begin{center}
\begin{tikzpicture}
[x=1mm,y=1mm]
\draw[fill,black!20] (0,10) -- (0,20) .. controls (10,20) and (20,10) .. (30,20) -- (30,15) .. controls (20,15) and (10,10) .. (0,10);
\draw[fill,black!20] (20,5) circle (4);
\draw[fill,white] (20,6) circle (2);
\draw[fill,white] (7,15) circle (3);
\draw (0,0) rectangle (30,30);
\draw (0,10) .. controls (10,10) and (20,15) .. (30,15);
\draw (0,20) .. controls (10,20) and (20,10) .. (30,20);
\draw (20,5) circle (4);
\draw (20,6) circle (2);
\draw (7,15) circle (3);
\end{tikzpicture}
\end{center}
\item[] Check that this indeed defines a functor as claimed.
\item[(b)] Verify by pictures that $f$ factors over an equivalence $h \mathcal{C}_1^{\mathbb{R}}[(h \mathcal{C}_1^{\mathbb{R}})^{-1}] \simeq D$.  
\item[] (Hint: it is recommended to first look at Exercise~\ref{exerc:5.3.1} below.  For a solution to this exercise see the end of these notes.)
\end{itemize}
\end{exercise}

\begin{exercise}[Cobordism categories with tangential structures.]\label{tangent}
\hspace{0pt}

As mentioned in the lectures, there are versions of the cobordism category for manifolds equipped with tangential structures, where a \emph{tangential structure} is a space $\Theta$ equipped with a continuous action of $\mathrm{GL}_d(\mathbb{R})$. A \emph{$\Theta$-structure} on a real vector bundle $E \to X$ is a $\mathrm{GL}_d(\mathbb{R})$-equivariant map $\mathrm{Fr}(E) \to \Theta$, where $\mathrm{Fr}(E)$ is the total space of the frame bundle of the vector bundle. As explained in the lectures, the topological cobordism category $\mathcal{C}_\Theta^V$ is defined similarly to $\mathcal{C}_d^V$, except that each object $M$ is equipped with a $\Theta$-structure on $\R \oplus TM$ and each morphism $W$ is equipped with a $\Theta$-structure on $TW$.

\begin{itemize}
\item[(a)] Unwind the definition of $\Theta$-structure in the cases:
  \begin{itemize}
  \item[(i)] $\Theta = \{*\}$,
  \item[(ii)] $\Theta = \{ \pm 1 \}$, where an element $A$ of $\mathrm{GL}_d(\mathbb{R})$ acts by multiplication by $\mathrm{det}(A)/\lvert \mathrm{det}(A) \rvert$,
  \item[(iii)] $\Theta = \mathrm{GL}_d(\mathbb{R})$, where $\mathrm{GL}_d(\mathbb{R})$ acts on itself by left-multiplication,
 \item[(iv)] $\Theta = \mathrm{GL}_{n}(\mathbb{R})/O(n)$, 
   \item[(v)] $\Theta = \mathrm{GL}_{2n}(\mathbb{R})/\mathrm{GL}_n(\bC)$, when $d=2n$ is even,
  \item[(vi)] $\Theta = Z$, where $Z$ is any topological space, and $\mathrm{GL}_d(\mathbb{R})$ acts trivially on $Z$.
  \end{itemize}
\item[(b)] Fill in the details of the definition of $\mathcal{C}_\Theta^V$ above. As before, $\mathcal{C}_\Theta$ is then defined as the colimit of $\mathcal{C}_\Theta^V$ over all finite-dimensional linear subspaces $V \subset \mathbb{R}^\infty$.
\item[(c)] Define the abstract cobordism category $\mathrm{Cob}_\Theta$ and show that it is equivalent to $h\mathcal{C}_\Theta$.
\item[(d)] Describe explicitly the categories $\mathrm{Cob}_{\{\pm 1\}}$ and $\mathrm{Cob}_Z$ for $d=1$ ($\mathrm{GL}_1(\mathbb{R})$ acts trivially on $Z$).
\end{itemize}

\end{exercise}

\begin{exercise}\label{discrete}
Let $\widetilde{\mathcal{C}}_d^V$ be the (ordinary) category obtained by giving the morphism spaces of $\mathcal{C}_d^V$ the discrete topology. What can you say about the connectivity of the map
\[
B\widetilde{\mathcal{C}}_d^V \longrightarrow B\mathcal{C}_d^V?
\]
\end{exercise}

\begin{exercise}\label{Postnikov}
  Let $F: \Top \to \Top$ be a functor from spaces\footnote{In this exercise it is somewhat important that ``spaces'' means ``compactly generated topological spaces''.} to spaces.
  Applying $F$ to the projections $\pi_0: X_0 \times X_1 \to X_0$ and $\pi_1: X_0 \times X_1 \to X_1$ gives a canonical map $F(X_0 \times X_1) \to F(X_0) \times F(X_1)$.  We say that the functor \emph{preserves finite products} if this canonical map is a homeomorphism for all spaces $X_0$ and $X_1$, and $F$ of a one-point space is again a one-point space.

  If $F$ admits a left adjoint then it preserves all small limits and in particular finite products, but there are interesting examples not of that form.
  \begin{itemize}
  \item[(a)] If $C$ is a topologically enriched category and $F: \Top \to \Top$ preserves finite products, explain how to define a new topologically enriched category $FC$, with the same object set as $C$, but with $(FC)(x,y) = F(C(x,y))$.
  \item[(b)] If $F \to F'$ is a natural transformation between functors preserving finite products, explain how to define an induced topologically enriched functor $FC \to F'C$.
  \item[(c)] Let $\tau_{\leq \infty}(X) = |\Sing(X)|$ denote the geometric realization of the total singular complex.  Recall why $\tau_{\leq \infty}: \Top \to \Top$ preserves finite products, and why the evaluation map $|\Sing(X)| \to X$ defines a natural transformation from $\tau_{\leq \infty}$ to the identity functor, which is always a weak equivalence.  This justifies some details of Example~\ref{ex:Sing-real}, giving a DK equivalence
    \begin{equation*}
      \tau_{\leq \infty} C \xrightarrow{\simeq} C
    \end{equation*}
    to any topologically enriched category $C$ from one whose morphism spaces are all CW complexes.
  \item[(d)] Let $\mathrm{Cosk}_{n}\Sing(X)$ be the simplicial set whose $p$-simplices are continuous maps $\mathrm{sk}_{n}(\Delta^p) \to X$ from the $n$-skeleton of $\Delta^n$, and let $\tau_{\leq n} X = |\mathrm{Cosk}_{n+1}\Sing(X)|$.  Recall why the canonical map $\tau_{\leq \infty} X \to \tau_{\leq n} X$ is a model for \emph{Postnikov truncation}: an $(n+1)$-connected map whose codomain has vanishing homotopy groups above degree $n$, with any basepoint.

    Explain that we get topologically enriched functors
    \begin{equation*}
      C \xleftarrow{\simeq} \tau_{\leq \infty} C \to \tau_{\leq n} C.
    \end{equation*}
    This gives a way of ``Postnikov truncating each morphism space of $C$'', retaining a topologically enriched category.
  \item[(e)]  For a space $X$, regard $\pi_0(X)$ as a topological space using the discrete topology (not the quotient topology from $X$).  Prove that the resulting functor $\pi_0: \Top \to \Top$ preserves finite products.  Construct a natural weak equivalence $\tau_{\leq 0}(X) \to \pi_0(X)$ and deduce a DK-equivalence
    \begin{equation*}
      \tau_{\leq 0} C \xrightarrow{\simeq} hC
    \end{equation*}
    for any topologically enriched category $C$.  In this way the construction $C \mapsto hC$ may be regarded as a special case of the Postnikov truncations $C \mapsto \tau_{\leq n} C$.
  \end{itemize}
\end{exercise}

\begin{exercise}\label{adjoints}
\hspace{0pt}
\begin{itemize}
\item[(a)] Prove that the ``constant simplicial set'' functor $\mathrm{Sets} \to \mathrm{sSets}$ is right adjoint to the functor $\pi_0 \colon \mathrm{sSets} \to \mathrm{Sets}$.
\item[(b)] Let us temporarily write $\Top^\mathrm{lpc}$ for the category of locally path connected topological spaces and continuous maps between such.  Prove that $\pi_0: \Top^\mathrm{lpc} \to \mathrm{Sets}$ is left adjoint to the functor $\mathrm{Sets} \to \Top^\mathrm{lpc}$ assigning a set its discrete topology.
\item[(c)] Prove that the ``discrete topology'' functor $\mathrm{Sets} \to \mathrm{Top}$ does not admit a left adjoint.
\item[] (\emph{Hint}: any functor which admits a left adjoint must preserve all small limits and in particular infinite products.)
\end{itemize}
\end{exercise}

\subsection{Exercise set 3}

\begin{exercise}
  \label{rigidsymm}
  Let $D$ be a rigid symmetric monoidal groupoid, let $1$ denote the monoidal unit, and let $x \in D$ be any object.  Prove that
  \begin{align*}
    \mathrm{End}_D(1) &\to \mathrm{End}_D(1 \otimes x)\\
    f & \mapsto f \otimes \mathrm{id}_x
  \end{align*}
  is an isomorphism of groups.  Using the unitor and its inverse, the codomain may be identified with $\mathrm{End}_D(x)$.

  Does this say anything useful when $D = \mathrm{Cob}_d[\mathrm{Cob}_d^{-1}]$?  (Hint: first show that $D$ is rigid.)
\end{exercise}

\begin{exercise}\label{exerc:5.3.1}
  Let $C$ be a small category and let $\gamma: C \to C[C^{-1}]$ be the universal functor to a groupoid.  In particular we have, for each object $x \in C$,
  \begin{equation}\label{eq:gamma}
    \gamma: \mathrm{End}_C(x) \to \mathrm{End}_{C[C^{-1}]}(x),
  \end{equation}
  a monoid homomorphism into a group.  The purpose of this exercise is to work out some useful rules for determining the group $\mathrm{End}_{C[C^{-1}]}(x)$, under the additional assumption on $C$ that for any other object $y$, both $C(x,y)$ and $C(y,x)$ are non-empty.  Unless stated otherwise,  we shall in the rest of this exercise make this assumption on $C$ and $x$.
  \begin{itemize}
  \item[(a)]  Prove that the image of~(\ref{eq:gamma}) generates.
  \item[(b)]  Let $y \in C$ be any object, let $w_1, w_2 \in C(y,x)$ and $w_3,w_4 \in C(x,y)$ be any morphisms, and define $a, b, c, d \in \mathrm{End}_C(x)$ by
    \begin{equation*}
      a = w_1 \circ w_3, \quad b = w_2 \circ w_3\quad
      c = w_1 \circ w_4, \quad 
      d = w_2 \circ w_4.
    \end{equation*}
    Then $\gamma(a)$, \dots, $\gamma(d)$ are elements of the group $\mathrm{End}_{C[C^{-1}]}(x)$.  Prove that\footnote{See \cite[Section 6]{BokstedtSvane} for further discussion.}
    \begin{equation}\label{eq:relation}
      \gamma(a) \circ (\gamma(b))^{-1} = \gamma(c) \circ (\gamma(d))^{-1}.
    \end{equation}
  \end{itemize}
  Let us now specialize to $C \subset h\mathcal{C}_d^V$ the full subcategory on those objects admitting a morphism from $\emptyset$, and set $x = \emptyset$.  To avoid special cases let us also take $\dim(V) > 0$ (see Remark~\ref{rmk:dimension-zero} for a discussion of that special case).
  \begin{itemize}
  \item[(c)] Prove that this $C$ and $x$ satisfy the assumption above.
  \item[(d)] Convince yourself that the domain of~(\ref{eq:gamma}) is a commutative monoid, and deduce that the codomain is an abelian group.
  \item[(e)] Convince yourself that, moreover, the domain of~(\ref{eq:gamma}) is a \emph{free} commutative monoid\footnote{An easier exercise is to prove this under the assumption $\dim(V) > 2d+1$ where the domain of~(\ref{eq:gamma}) is free on the set of classes of path connected manifolds.  For a solution in the general case see the end of these notes.}.
  \item[(f)] Use the tools developed above to show that, in $\mathrm{Cob}_2[(\mathrm{Cob}_2)^{-1}]$, the endomorphisms of $\emptyset$ given by the torus, by the Klein bottle and by the empty surface are equal.
  \item[(g)] Return to Exercise \ref{eift}(c) and Exercise \ref{groupoid}(b).
  \end{itemize}
\end{exercise}



\begin{exercise}\label{exercise:cob-cat-with-connectivity}
  In this exercise we shall study the subcategories $\Cob_d^k \subset \Cob_d$ and $\mathcal{C}_d^k$ from Section~\ref{sec:restrictions}.
\begin{itemize}
\item[(a)] Verify that $\Cob_d^k \subset \Cob_d$ is indeed a subcategory.  Deduce that $\mathcal{C}_d^k$ is a subcategory of $\mathcal{C}_d$.
\end{itemize}
In the lectures it was stated that the inclusion $\mathcal{C}_d^k \hookrightarrow \mathcal{C}_d$ induces a weak homotopy equivalence of classifying spaces
\begin{equation}\label{inclusion}
B\mathcal{C}_d^k \longrightarrow B\mathcal{C}_d
\end{equation}
as long as $k \leq \frac{d-2}{2}$. The purpose of this exercise is to investigate, by more elementary means, when the map \eqref{inclusion} induces a bijection on $\pi_0$.
\begin{itemize}
\item[(b)] Prove ``by hand'' that, when $d\geq 2$ and $k=0$, the map \eqref{inclusion} induces a bijection on $\pi_0$.
\item[(c)] Rephrase bijectivity of \eqref{inclusion} on $\pi_0$ as the statement that, given any cobordism $W \colon M_0 \rightsquigarrow M_1$ in $\mathcal{C}_d$, there is a zig-zag of cobordisms between $M_0$ and $M_1$ that each satisfy the $k$-connectivity condition on the outgoing boundary.
\item[(d)] Prove bijectivity of \eqref{inclusion} on $\pi_0$ more generally whenever $k < d/2$. 
\item[] (\emph{Hint}: if you haven't already, learn about \emph{elementary cobordisms} e.g.\ from Milnor's book on the $h$-cobordism theorem.)
\end{itemize}
\end{exercise}

\begin{exercise}
  Let $D$ be the groupoid with one object $\ast$ and $\mathrm{End}_D(\ast) = \bZ$.  Define
  \begin{align*}
    E: \mathrm{Cob}_d \to D    
  \end{align*}
  by sending any object to $\ast$ and a morphism $W: M_0 \leadsto M_1$ to $\chi(W) - \chi(M_0) \in \bZ$.
  \begin{itemize}
    
  \item[(a)] Briefly explain why this is a functor.
  \item[] (You may have already done this in Exercise \ref{eift}(b).)
  \item[(b)] Prove that $E$ is isomorphic to the trivial functor (sending any object to $\ast$ and any morphism to 0) when $d$ is odd.
  \item[(c)] Prove that $E$ is not isomorphic to the trivial functor when $d$ is even.
  \item[(d)] Can you promote $E$ to a symmetric monoidal functor?  (Hint: use addition in $\bZ$ as $\otimes$.  In this case the associator, symmetry, and unitor can all be taken to be the identity.) 
  \end{itemize}
  This is sometimes called the ``Euler TQFT''.
  \begin{itemize}
  \item[(e)] Let $V = \mathbb{R}^m$, let $e \colon T_{d,\mathbb{R}^{m+1}} \to K(\bZ,m+1)$ be a based map, and let $\Omega^m T_{d,\mathbb{R}^{m+1}} \to \Omega^m K(\bZ,m+1)$ be the $m$-fold loop of $e$. Prove that the fundamental groupoid of $\Omega^m K(\bZ,m+1)$ is equivalent to $D$, and explain how any such map $e$ gives rise to a symmetric monoidal invertible field theory $\mathrm{Cob}_d \to D$ if $m \geq 3$.
  \item[(f)] $(*)$ For $d$ even, let $e \in \widetilde{H}^{m+1}(T_{d,\mathbb{R}^{m+1}})$ be the cohomology class corresponding to the Euler class of the universal bundle under Thom isomorphism
    \begin{equation*}
      H^d(\Gr_d(\R \oplus V);\widetilde\bZ) \xrightarrow{\cong} H^{m+1}(T_{d,\R^{m+1}};\Z),
    \end{equation*}
    where $\widetilde{\bZ}$ denotes the local system corresponding to the orientation character of the canonical bundle (this same local coefficient system shows up in both the Thom isomorphism and in the Euler class).  Use the same notation $e: T_{d,\mathbb{R}^{m+1}} \to K(\bZ,m+1)$ for a pointed map classifying this cohomology class, and prove that the TQFT defined in (e) agrees with the Euler TQFT.  (Hint: in this part of the exercise it may be necessary to know a geometric description of the map $\mathcal{C}_d(\emptyset,\emptyset) \to \Omega^{n+1} T_{d,\R^{n+1}}$.  This is the \emph{Pontryagin--Thom construction}, see e.g.\ \cite[p.~197]{GMTW}.)
  \end{itemize}
\end{exercise}

\begin{exercise}\label{exerc:k-invariant}
  If $X$ is a rigid symmetric monoidal groupoid, it is determined up to equivalence by three pieces of data:
  $\pi_0X$ (the abelian group of isomorphism classes of objects), $\pi_1X = \mathrm{Aut}(1_X)$, and something
  called the \emph{$k$-invariant}, which we proceed to define (see Remark~\ref{remark:postnikov-for-spaces} for the related notion of $k$-invariants of spaces). Given any $x\in X$, there is a canonical isomorphism
  $\text{--}\otimes\mathrm{id}_x \colon \mathrm{Aut}(1_X)\to\mathrm{Aut}(x)$. The $k$-invariant of $X$ is the map
  $\pi_0X\otimes\mathbb Z/2\to\pi_1 X$ which to $x\in\pi_0X$ assigns the image of the
  symmetry $\sigma\colon x\otimes x\to x\otimes x$ in $\mathrm{Aut}(x\otimes x)\cong\mathrm{Aut}(1_X) = \pi_1X$.

  Compute $\pi_0$, $\pi_1$, and $k$ for the following rigid symmetric monoidal groupoids.
  \begin{itemize}
  \item[(a)] The category $\mathrm{Vect}_k^\sim$ of invertible vector spaces over a field $k$.
  \item[(b)] Assuming $\mathrm{char}(k)\ne 2$, the category $\mathrm{sVect}_k^\sim$ of invertible \emph{super
  vector spaces}, i.e.\ the category of invertible $\mathbb Z/2$-graded vector spaces with the symmetry $a\otimes b \mapsto (-1)^{\deg a\deg b}b\otimes a$.
  \item[(c)] $\mathrm{Cob}_1[\mathrm{Cob}_1^{-1}]$.
  \item[(d)] The same as in (c), but with the oriented $1$-dimensional cobordism category.
  \end{itemize}
\end{exercise}

\begin{exercise}
Let us denote by $h\mathring{\mathcal{C}}_d^V$ the oriented version of the category $h\mathcal{C}_d^V$, where both $(d-1)$-manifolds and cobordisms are equipped with compatible orientations.\footnote{This was called $h\mathcal{C}_{\{\pm 1\}}^V$ in Exercise \ref{tangent}.} There is a functor
\[
F_{d,V} \colon h\mathring{\mathcal{C}}_d^V \longrightarrow h\mathcal{C}_d^V
\]
that forgets all orientations.
\begin{itemize}
\item[(a)] When $d=\mathrm{dim}(V)$ or $d=0$, construct a section of $F_{d,V}$.    (``Section'' means a functor $G$ in the other direction, such that $F_{d,V} \circ G$ is naturally isomorphic to the identity.)
\item[(b)] When $0<d<\mathrm{dim}(V)$, prove that the functor $F_{d,V}$ does not admit a section.
\end{itemize}
(\emph{Suggestion}: first consider the cases $(d,V) = (1,\mathbb{R})$ and $(d,V) = (1,\mathbb{R}^2)$, and look at the picture in Exercise \ref{groupoid}.)
\end{exercise}



\section*{Solutions to selected exercises (by A.\ Debray, S.\ Galatius, M.\ Palmer)}\refstepcounter{section}
\addcontentsline{toc}{section}{Solutions to selected exercises (by A.\ Debray, S.\ Galatius, M.\ Palmer)}

\subsection{Solution to Exercise~\ref{groupoid}(b)}
\label{sec:5-2-2-b}

In part (a) of the exercise, we defined a functor $f\colon \cC_1^\bR\to D$ by sending a $0$-manifold $M\subset\bR$
to $\# M\bmod 2$ and a bordism $W$ to $\chi(X) - \chi(X\cap \bR\times\set 0)\in\bZ$. This is well-defined because
if $M$ and $N$ are bordant, their cardinalities mod 2 agree, so $\Hom_D(f(M), f(N)) = \bZ$; then, the additivity
formula for Euler characteristic implies that $f$ is a functor.

Since $D$ is a discrete category, $f$ factors as a functor $f\colon\hC\to D$; since $D$ is a groupoid, $f$ factors
further into a functor $f\colon\ghC\to D$, where $f(W^{-1})\coloneqq f(W)^{-1}$ for a bordism $W$. For part (b),
we show $f\colon \ghC\to D$ is an equivalence of groupoids, meaning that it is fully faithful and essentially
surjective.  Essential surjectivity, i.e.\ that $f$ induces a surjection on the sets of isomorphism classes of
objects, is clear: the empty set in $\ghC$ maps to $0\in\mathrm{Ob}(D)$, and a single point in $\ghC$ maps to
$1\in\mathrm{Ob}(D)$.

First we reduce to considering the empty manifold, using an argument similar to Exercise~\ref{rigidsymm}.
\begin{lemma}
\label{just_auts}
Let $f\colon X\to Y$ be a map of small groupoids. If $\pi_0 f\colon \pi_0X\to\pi_0Y$ (i.e.\ the induced map on sets of
isomorphism classes) is a bijection, then $f$ is fully faithful if and only it induces isomorphisms $\Aut_X(x) \to \Aut_Y(f(x))$ for all $x \in X$.
\end{lemma}
\begin{proof}
If $\psi\colon x\to x'$ is any map in $X$, composing with $\psi$ is an isomorphism $\Aut_X(x)
\overset\cong\to\Hom_X(x, x')$. In the diagram
\begin{equation}
\begin{gathered}
	\xymatrix{
		\Aut_{X}(x)\ar[r]^-\psi_-\cong\ar[d]^f & \Hom_X(x,x')\ar[d]^f\\
		\Aut_Y(f(x))\ar[r]^-{f(\psi)}_-\cong & \Hom_Y(f(x), f(x')),
	}
\end{gathered}
\end{equation}
the horizontal arrows, given by composing with $\psi$ and $F(\psi)$ respectively, are bijections. Hence if the
left vertical arrow is an isomorphism, so is the right vertical arrow. This suffices: there cannot be any $x$ and
$x'$ such that $\Hom_X(x, x') = \varnothing$ and $\Hom_Y(f(x), f(x'))\ne\varnothing$ because $\pi_0f$ is an
isomorphism.
\end{proof}
\begin{lemma}
\label{reduce_to_unit}
Suppose the map $\Aut_{\ghC}(\varnothing)\to \Aut_D(0)$ induced by $f$ is an isomorphism. Then $f$ is fully
faithful.
\end{lemma}
\begin{proof}
By Lemma~\ref{just_auts}, we have only to consider automorphism groups of objects in $\ghC$; there are only two
isomorphism classes of objects, given by the empty manifold and a single point, and the assumption on $\pi_0f$ is
satisfied. We will define an endofunctor $F'$ of $\ghC$ and show that it induces an isomorphism
$\Aut(\varnothing)\to\Aut(\pt)$.

First, we define an endofunctor $F$ of $\cC_1^\R$. The idea of $F$ is to disjoint union all objects with a point,
and all bordisms with an interval regarded as a bordism between the new points. Fix a diffeomorphism
$\phi\colon\R\to (-\infty, 0)$ isotopic to the identity on $\R$, and let $\Phi\colon
[0,t]\times\R\to[0,t]\times(-\infty, 0)$ denote $(s,x)\mapsto
(s, \phi(x))$.
\begin{itemize}
	\item Given an object $M\subset\R$, $F(M)\coloneqq \phi(M)\cup \set 1$.
	\item Given a bordism $W\subset [0,t]\times\R$, let $F(W)\coloneqq \Phi(W)\cup [0,t]\times\set 1$.
\end{itemize}
This is continuous on spaces of objects and morphisms, hence passes to an endofunctor of $\hC$, and this descends to
an endofunctor $F'$ of $\ghC$ by $F'(x^{-1})\coloneqq F(x)^{-1}$.

Choose an isotopy $H\colon [0,1]\times\R\to\R$ from $\phi\circ\phi$ to the identity. We define a natural
transformation $T\colon F^2\Rightarrow\id$, where $F$ is regarded in $\hC$, by ``capping off'' the two new points.
For every object $x$ of $C_1^\R$, we produce a bordism $T_x\colon F^2 x\to x$, defined to be the union of
$\set{H(s, p): s\in[0,1], p\in x}$ and a semicircle in $[0,1]\times\R$ containing the two new points of
$F^2 x$ as its boundary, as in the following picture.
\newcommand{\stretchamt}{1.8}
\begin{equation}
\wg{\begin{tikzpicture}
	\begin{scope}
		\clip (0, 0) rectangle (1.5, 1.5);
		\foreach \y in {0.2, 0.425, 0.65} {
			\draw (0, \y) to[out=0, in=180] (1.5, \stretchamt*\y);
		}
		\draw (0, 1.17) circle (0.17);
	\end{scope}
	\foreach \y in {0.2, 0.425, 0.65} {
		\fill (0, \y) circle (0.05);
		\fill (1.5, \stretchamt*\y) circle (0.05);
	}
	\foreach \x in {0, 1} {
		\draw (1.5*\x, 0) -- (1.5*\x, 1.5);
	}
	\fill (0, 1) circle (0.05);
	\fill (0, 1.34) circle (0.05);
	\node[below] at (0, 0) {$F^2x$};
	\node[below] at (1.5, 0) {$\vphantom{F^2}x$};
\end{tikzpicture}}
\end{equation}
To check that this is in fact a natural transformation after passing to $\hC$, we need that for every bordism
$W\colon x\to y$, the diagram
\begin{equation}
\begin{gathered}
\xymatrix{
	F^2x\ar[r]^-{T_x}\ar[d]_{F^2W} & x\ar[d]^W\\
	F^2y\ar[r]^-{T_y} & y
}
\end{gathered}
\end{equation}
commutes. The picture proof is that the following two bordisms are equal in $\hC$.
\begin{equation}
\wg{\begin{tikzpicture}
	\begin{scope} 
		\clip (0, 0) rectangle (1.5, 1.5);
		\foreach \y in {0.2, 0.425, 0.65} {
                  \draw (0, \y) to[out=0, in=180] (1.5, \stretchamt*\y);
                      }
		\draw (0, 1.17) circle (0.17);
	\end{scope}
	\begin{scope} 
		\clip (1.5, 0) rectangle (3, 1.5);
		\draw (1.5, \stretchamt*0.2) -- (3, \stretchamt*0.2);
		\draw (1.5, \stretchamt*0.5375) circle (\stretchamt*0.1125);
		\draw (3, \stretchamt*0.6) circle (\stretchamt*0.15);
	\end{scope}
	\foreach \y in {0.2, 0.425, 0.65} {
		\fill (0, \y) circle (0.05);
		\fill (1.5, \stretchamt*\y) circle (0.05);
	}
	\foreach \x in {0, 1} {
		\draw (1.5*\x, 0) -- (1.5*\x, 1.5);
	}
	\draw (3, 0) -- (3, 1.5);
	\fill (0, 1) circle (0.05);
	\fill (0, 1.34) circle (0.05);
	\foreach \y in {\stretchamt*0.2, \stretchamt*0.45, \stretchamt*0.75} {
		\fill (3, \y) circle (0.05);
	}
	\node[below] at (0, 0) {$F^2x$};
	\node[below] at (1.5, 0) {$\vphantom{F^2}x$};
	\node[below] at (3, 0) {$\vphantom{F^2}y$};
\end{tikzpicture}}\qquad =\qquad \wg{\begin{tikzpicture}
	\begin{scope} 
		\clip (0, 0) rectangle (1.5, 1.5);
		\draw (0, 0.2) -- (1.5, 0.2);
		\draw (0, 0.5375) circle (0.1125);
		\draw (1.5, 0.6) circle (0.15);
		\foreach \y in {1, 1.34} {
			\draw (0, \y) -- (1.5, \y);
		}
	\end{scope}
	\begin{scope} 
		\clip (1.5, 0) rectangle (3, 1.5);
		\draw (1.5, 0.2) to[out=0, in=180] (3, \stretchamt*0.2);
		\draw (1.5, 1.17) circle (0.17);
	\end{scope}
	\foreach \y in {0.45, 0.75} {
		\draw (1.5, \y) to[out=0, in=180] (3, \stretchamt*\y);
		\fill (1.5, \y) circle (0.05);
		\fill (3, \stretchamt*\y) circle (0.05);
	}
	\foreach \x in {0, 1.5, 3} {
		\draw (\x, 0) -- (\x, 1.5);
	}
		\fill (0, 0.2) circle (0.05);
		\fill (1.5, 0.2) circle (0.05);
		\fill (3, \stretchamt*0.2) circle (0.05);
	\fill (0, 0.425) circle (0.05);
	\fill (0, 0.65) circle (0.05);
	\foreach \x in {0, 1.5} {
		\foreach \y in {1, 1.34} {
			\fill (\x, \y) circle (0.05);
		}
	}
	\node[below] at (0, 0) {$F^2x$};
	\node[below] at (1.5, 0) {$F^2y$};
	\node[below] at (3, 0) {$\vphantom{F^2}y$};
\end{tikzpicture}}.
\end{equation}
(Though we drew a specific bordism $x\to y$, the argument works in general.)

$T$ also descends to a natural transformation $T'\colon (F')^2\Rightarrow\id$, and every natural transformation
between functors into a groupoid is a natural isomorphism. Thus applying $F'$ defines an isomorphism
\begin{equation}
	F'_*\colon \Aut_{\ghC}(\varnothing)\to\Aut_{\ghC}(\pt).
\end{equation}
This is compatible with the identifications $\Aut_D(0) = \bZ = \Aut_D(1)$, in that if
$W\in\Aut_{\ghC}(\varnothing)$, then $f(W) = f(F_*'(W))$, because applying $F$ increases the Euler characteristics
of $X$ and $X\cap(\set 0\times\R)$ by the same amount. Therefore if
\begin{equation}
	f_*\colon \Aut_{\ghC}(\varnothing)\to\Aut_D(0)
\end{equation}
is an isomorphism, so is
\begin{equation}
	f_*\colon \Aut_{\ghC}(\pt)\to\Aut_D(1).
	\qedhere
\end{equation}
\end{proof}
To show $\Aut_{\ghC}(\varnothing)\to\Aut_D(0) = \bZ$ is surjective, it suffices to produce a preimage of the
generator $1\in\Aut_D(0)$, and a single circle, interpreted as a bordism $W\colon \varnothing\to\varnothing$,
works: $X$ is a disc and does not intersect the incoming boundary, so $f(W) = 1$.

Then, we will show that for every morphism
$W\colon \varnothing\to\varnothing$ in $\hC$, the image of $W$ in $\ghC$ is equal to a composition of $n$ circles,
where $n\in\bZ$ (here $n < 0$ is interpreted as a composition of $-n$ formal inverses of circles). By
Exercise~\ref{exerc:5.3.1}, parts~(a) and~(c), $\Aut_{\ghC}(\varnothing)$ is generated by such $W$, so it will follow that $\Aut_{\ghC}(\varnothing)$ is cyclic.  Since every surjective homomorphism from a cyclic group to $\bZ$ is an
isomorphism, this suffices for Lemma~\ref{reduce_to_unit} to apply and conclude the proof.

For any bordism $W\colon\varnothing\to\varnothing$, the complement of $W$ in $[0,t]\times\R$ has exactly one
unbounded path component $A$: there is at least one because $[0,t]\times\R$ is unbounded, and at most one because
$W$ does not intersect $\partial([0,t]\times\R)$ and is compact, hence is contained in some disc in
$(0,t)\times\R$. Let $\overline A$ denote the closure of $A$ in $(0,t)\times\R$. Then $\overline A$ is a
$2$-manifold with boundary and $\partial\overline A$ is bounded in $[0,t]\times\R$, so $\partial\overline A$ is a
compact $1$-manifold with empty boundary, and therefore a disjoint union of $k$ circles. Call $W$ simple if $k =
1$.
\begin{lemma}
\label{simplecomp}
Every bordism $W\colon\varnothing\to\varnothing$ in $\hC$ is a composition of simple bordisms.
\end{lemma}
\begin{proof}
Let $C_1,\dotsc,C_k$ be the components of $\partial\overline A$.
Each $C_i$ is diffeomorphic to a circle, so by the smooth Schoenflies theorem, the complement of $C_i$ in
$(0,t)\times\R$ has a unique bounded component $D_i$, and $\overline{D_i}$ is diffeomorphic to a closed disc.
Then $Y \coloneqq \overline{D_1}\cup\dots\cup\overline{D_k}$ is the complement of $A$ in $[0,t]\times\R$, so $W$ is
contained in $Y$. We can now apply a diffeomorphism of $[0,t]\times\R$ fixing the boundary and such that
$\overline{D_i}$ lands in $((i-1)t/k, it/k)\times\R$, factoring $W$ as the composition of the simple bordisms
$W\cap\overline{D_i}$.
\end{proof}
Therefore we are done if we can show that for every simple bordism $W\colon \varnothing\to\varnothing$ in $\hC$,
its image in $\ghC$ is equal to a composition of $n$ circles for some $n\in\bZ$. Assume $W$ has $m$ path components;
$m = 0$ would not be simple, and if $m = 1$ we are already done. We induct on $m$. A general simple bordism $W$
with $m > 1$ is of the form
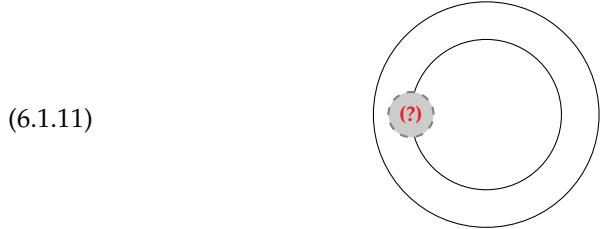
\begin{equation}
\begin{gathered}
\begin{tikzpicture}
	\draw (0, 0) circle (1.5);
	\draw (0, 0) circle (1);
	\draw[thick, dashed, gray, fill=\lgray] (-1, 0) circle (0.3);
	\node at (-1, 0) {\stuff};
\end{tikzpicture}
\end{gathered}
\end{equation}
where the \textstuff{} indicates more components of $W$ that may be present. We use Exercise~\ref{exerc:5.3.1} to
simplify this picture. Recall that Exercise~\ref{exerc:5.3.1} tells us for any object $y$ of $\hC$ admitting a map
from $x\coloneqq \varnothing$ and any morphisms $w_1,w_2\colon y\to x$ and $w_3,w_4\colon x\to y$ in $\hC$, that in
$\Aut_{\ghC}(y)$,
\begin{equation}
\label{w1w2w3w4}
	w_1\circ w_3\circ (w_2\circ w_3)^{-1} = w_1\circ w_4\circ (w_2\circ w_4)^{-1}.
\end{equation}
Apply this to $y \coloneqq \pt^{\amalg 4}$ and the four morphisms
\begin{equation}
w_1 = 
	\wg{\begin{tikzpicture}
		\begin{scope}
			\clip (-1, -1) rectangle (0, 1);
			\draw (0, 0) circle (0.75);
			\draw (0, 0) circle (0.4);
			\draw[thick, dashed, gray, fill=\lgray] (-0.4, 0) circle (0.225);
			\node at (-0.4, 0) {\stuff};
		\end{scope}
		\draw (0, -1) -- (0, 1);
	\end{tikzpicture}},\quad
w_2 = \wg{\tworightsemicircles{2}},\quad
w_3 = \wg{\leftmacaroni{2}},\quad
w_4 = \wg{\twoleftsemicircles{2}}.
\end{equation}
Then~\eqref{w1w2w3w4} tells us
\begin{equation}
	\underbracket{\wg{\begin{tikzpicture}
		\draw (0, 0) circle (0.75);
		\draw (0, 0) circle (0.4);
		\draw[thick, dashed, gray, fill=\lgray] (-0.4, 0) circle (0.225);
		\node at (-0.4, 0) {\stuff};
	\end{tikzpicture}}\circ \wg{\singledisc{2}}^{-1}}_{w_1\circ w_3\circ (w_2\circ w_3)^{-1}} =
	\underbracket{
		\wg{\begin{tikzpicture}
			\draw (0, 0) circle (0.6);
			\draw[thick, dashed, gray, fill=\lgray] (0.5, 0) circle (0.35);
			\node at (0.5, 0) {\stuff};
		\end{tikzpicture}}\circ \wg{\twodiscs{2}}^{-1}
	}_{w_1\circ w_4\circ (w_2\circ w_4)^{-1}}.
\end{equation}
Precomposing with a single circle,
\begin{equation}
\wg{\begin{tikzpicture}
		\draw (0, 0) circle (0.75);
		\draw (0, 0) circle (0.4);
		\draw[thick, dashed, gray, fill=\lgray] (-0.4, 0) circle (0.225);
		\node at (-0.4, 0) {\stuff};
	\end{tikzpicture}}
	= 	\underbracket{
		\wg{\begin{tikzpicture}
			\draw (0, 0) circle (0.6);
			\draw[thick, dashed, gray, fill=\lgray] (0.5, 0) circle (0.35);
			\node at (0.5, 0) {\stuff};
		\end{tikzpicture}}}_{w_1\circ w_4} \circ \wg{\singledisc{2}}^{-1}.
\end{equation}
The bordism $w_1\circ w_4$ has fewer path components than $W$. It is not necessarily true that $w_1\circ w_4$ is
simple, but by Lemma~\ref{simplecomp} it is a composition of simple bordisms, and each has fewer path components
than $W$, so we induct and conclude that every simple bordism can be written as a union of circles and their
inverses.

\subsection{Solution to Exercise~\ref{rigidsymm}(e)}
\label{sec:5-3-1-e}

\renewcommand{\coloneqq}{=}
\renewcommand{\subseteq}{\subset}

The goal of this exercise is to show that the monoid $h\mathcal{C}_d^V(\emptyset,\emptyset)$ is \emph{free commutative} when $\mathrm{dim}(V) > 0$. The fact that it is commutative (which was Exercise 5.3.1(d)) follows from Corollary \ref{coro:model} below, and the fact that it is \emph{free} commutative is Corollary \ref{coro:freely-generating}. Remark \ref{rmk:tangential} describes the corresponding statements when manifolds are equipped with tangential structures and Remark \ref{rmk:Em-algebras} comments on whether the space $\mathcal{C}_d^V(\emptyset,\emptyset)$ is free as an $E_m$-algebra.

\subsubsection*{Description of the monoid.}

Let us write $U = \bR \times V$ and define $m = \mathrm{dim}(U) = \mathrm{dim}(V) + 1$. Note that, unwinding the definitions, we may see that $h\mathcal{C}_d^V(\emptyset,\emptyset)$ is isomorphic to the monoid
\begin{equation*}
\begin{split}
\mathcal{M}_d^V = \left\lbrace (W,\varphi) \Biggm| \begin{matrix}
W \text{ is a smooth, closed $d$-manifold} \\
\varphi \text{ is a smooth embedding } W \hookrightarrow U
\end{matrix} \right\rbrace / {\sim},
\end{split}
\end{equation*}
where $(W,\varphi) \sim (W',\varphi')$ if and only if
\[
[\varphi' \circ \theta] = [\varphi] \in \pi_0 \left( \mathrm{Emb}(W,U)/\mathrm{Diff}(W) \right)
\]
for a diffeomorphism $\theta \colon W \cong W'$, and the operation in $\mathcal{M}_d^V$ is given by
\begin{equation}
\label{eq:operation}
[(W_1,\varphi_1)] \cdot [(W_2,\varphi_2)] = [(W_1 \sqcup W_2 , i \circ (\varphi_1 \sqcup \varphi_2))],
\end{equation}
where $i = j \times \mathrm{id}_V$ for any orientation-preserving embedding $j \colon \bR \sqcup \bR \hookrightarrow \bR$. In other words, $\mathcal{M}_d^V$ is the monoid of isotopy classes of (unparametrized) closed, $d$-dimensional submanifolds of $U$, under the operation of ``placing two submanifolds side-by-side'' in $U$.

\subsubsection*{Embedding and diffeomorphism results.}

For the solution, we will need the following facts. For smooth manifolds $M,N$, recall that $\mathrm{Emb}(M,N)$ denotes the space of smooth embeddings of $M$ into $N$ and $\mathrm{Diff}(N)$ the topological group of self-diffeomorphisms of $N$, both equipped with the smooth Whitney topology. We also write $\mathrm{Diff}_c(N)$ for the subgroup of $\mathrm{Diff}(N)$ of those diffeomorphisms $\varphi$ such that $\{ x \in N \mid \varphi(x) \neq x \} \subset K$ for some compact $K \subset N$ (topologized as a colimit over larger and larger compact $K \subset N$) and $\mathrm{Diff}_c(N)_1$ for the path-component of the identity in $\mathrm{Diff}_c(N)$.
\begin{proposition}[\emph{Isotopy extension theorem}]
\label{prop:isotopy-extension}
Let $M,N$ be smooth manifolds without boundary, where $M$ is compact and $\mathrm{dim}(M) < \mathrm{dim}(N)$, and let $i \colon M \hookrightarrow N$ be a smooth embedding. Then the map
\begin{equation}
\label{eq:restriction}
\mathrm{Diff}_c(N)_1 \longrightarrow \mathrm{Emb}(M,N)/\mathrm{Diff}(M),
\end{equation}
defined by $\Phi \mapsto [\Phi \circ i]$, is a Serre fibration, and thus in particular admits path-lifting.
\end{proposition}
\begin{proof}
We may factor \eqref{eq:restriction} into two maps
\[
r_1 \colon \mathrm{Diff}_c(N)_1 \longrightarrow \mathrm{Emb}(M,N) \qquad\text{and}\qquad r_2 \colon \mathrm{Emb}(M,N) \longrightarrow \mathrm{Emb}(M,N)/\mathrm{Diff}(M).
\]
The first map $r_1$ is a fibre bundle, and hence a Serre fibration, by combining Theorems A and B of \cite{Palais} (see also Corollaire 2 of \S II.2.2.2 of \cite{Cerf} and \cite{Lima}). The second map $r_2$ is a principal $\mathrm{Diff}(M)$-bundle, and hence in particular a Serre fibration, by Theorem 13.11 of \cite{Michor} (see also \cite{MR613004} and Theorem 44.1 of \cite{KrieglMichor}).
\end{proof}

\begin{proposition}[\emph{Multi-disc theorem}]
\label{prop:multi-disc-theorem}
Let $N$ be a smooth, connected, oriented manifold of dimension $m\geq 2$, let $i_1,\ldots,i_r \colon \bD^m \hookrightarrow N$ be a finite collection of orientation-preserving smooth embeddings of the closed $m$-disc into $N$ with pairwise disjoint images, and let $j_1,\ldots,j_r$ be a second such collection. Then there exists $\Phi \in \mathrm{Diff}_c(N)_1$ such that $\Phi \circ i_s = j_s$ for all $s \in \{ 1,\ldots,r \}$.
\end{proposition}
\begin{proof}
When $r=1$ this is exactly the \emph{Disc theorem} of Palais \cite[Theorem~B and Corollary~1]{Palais-extending}. We show how this implies the result for general $r$. First, note that we may assume, without loss of generality, that the images of $i_1,\ldots,i_r$ are also disjoint from the images of the $j_1,\ldots,j_r$.

For each $s \in \{ 1,\ldots,r \}$ we may choose a smooth path $\gamma_s$ in $N$ that is disjoint from $i_t(\bD^m)$ and $j_t(\bD^m)$ for all $t \neq s$ and such that $\gamma_s(0) = i_s(0)$ and $\gamma_s(1) = j_s(0)$. We may also assume that the paths $\gamma_1,\ldots,\gamma_r$ have \emph{pairwise disjoint} images.\footnote{For $m\geq 3$ it suffices to ensure that the $\gamma_1,\ldots,\gamma_r$ are pairwise \emph{transverse}. For $m=2$, some extra care is needed: after ensuring that these smooth paths are pairwise transverse and therefore intersect in finitely many points, we apply finitely many isotopies to remove these intersection points one at a time by ``pushing them off the end of the arc''.} We may then choose pairwise disjoint, connected, open subsets $A_1,\ldots,A_r$ of $N$ such that
\[
i_s(\bD^m) \cup \gamma_s([0,1]) \cup j_s(\bD^m) \subseteq A_s
\]
for $s \in \{ 1,\ldots,r \}$. By the Disc theorem of Palais, we may find $\Phi_s \in \mathrm{Diff}_c(A_s)_1$ such that $\Phi_s \circ i_s = j_s$ for all $s \in \{ 1,\ldots,r \}$. We may then define $\Phi$ to be $\Phi_s$ on each $A_s$ and the identity elsewhere.
\end{proof}


\subsubsection*{Another description of the monoid.}

The next two lemmas give a more convenient description of $\mathcal{M}_d^V$.

\begin{lemma}
\label{lem:equivalence}
An alternative description of the equivalence relation in the definition of $\mathcal{M}_d^V$ is $(W,\varphi) \sim (W',\varphi')$ if and only if there is a diffeomorphism $\Phi \in \mathrm{Diff}_c(U)_1$ such that $\Phi(\varphi(W)) = \varphi'(W')$.
\end{lemma}
\begin{proof}
Suppose first that there is a diffeomorphism $\Phi \in \mathrm{Diff}_c(U)_1$ such that $\Phi(\varphi(W)) = \varphi'(W')$. We obtain a diffeomorphism $\theta \colon W \cong W'$ by setting $\theta = (\varphi')^{-1} \circ \Phi \circ \varphi$. Choose a path $\mathrm{id} \rightsquigarrow \Phi$ in $\mathrm{Diff}_c(U)_1$ and consider its composition with \eqref{eq:restriction}, where we set $N=U$, $M=W$ and $i = \varphi$. This is a path in $\mathrm{Emb}(W,U)/\mathrm{Diff}(W)$ from $[\varphi]$ to $[\Phi \circ \varphi] = [\varphi' \circ \theta]$. Hence $(W,\varphi) \sim (W',\varphi')$.

Conversely, suppose that $(W,\varphi) \sim (W',\varphi')$, so we have a diffeomorphism $\theta \colon W \cong W'$ and a path $[\varphi] \rightsquigarrow [\varphi' \circ \theta]$ in $\mathrm{Emb}(W,U)/\mathrm{Diff}(W)$. By Proposition \ref{prop:isotopy-extension}, we may lift this to a path $\mathrm{id} \rightsquigarrow \Phi$ such that $[\Phi \circ \varphi] = [\varphi' \circ \theta]$, and hence $\Phi(\varphi(W)) = \varphi'(\theta(W)) = \varphi'(W')$.
\end{proof}

\begin{lemma}
\label{lem:operation}
Choose any smooth, orientation-preserving embedding $e \colon U \sqcup U \hookrightarrow U$. Then the operation in $\mathcal{M}_d^V$ may be defined using this embedding:
\begin{equation}
\label{eq:operation2}
[(W_1,\varphi_1)] \cdot [(W_2,\varphi_2)] = [(W_1 \sqcup W_2 , e \circ (\varphi_1 \sqcup \varphi_2))].
\end{equation}
Moreover, if $[(W_1,\varphi_1)],\ldots,[(W_r,\varphi_r)]$ is any finite collection of elements of $\mathcal{M}_d^V$ and $e$ is a smooth, orientation-preserving embedding of $r$ disjoint copies of $U$ into $U$, then we have
\begin{equation}
\label{eq:operation3}
[(W_1,\varphi_1)] \cdots [(W_r,\varphi_r)] = [(W_1 \sqcup \cdots \sqcup W_r , e \circ (\varphi_1 \sqcup \cdots \sqcup \varphi_r))].
\end{equation}
\end{lemma}
\begin{proof}
The operation in $\mathcal{M}_d^V$ was defined exactly as in equation \eqref{eq:operation2} (see equation \eqref{eq:operation}), \emph{assuming} that the embedding $e$ is of a particular form (namely of the form $(\bR \sqcup \bR \hookrightarrow \bR) \times \mathrm{id}_V$). It follows that equation \eqref{eq:operation3} also holds, again \emph{assuming} that the embedding $e$ is of a particular form.\footnote{For example, one choice of conventions would lead to $e$ being of the following form: let us identify $U$ with $(0,1) \times V$ and, on the $i$th copy of $U$ in the disjoint union, we define $e$ to be the embedding $U \hookrightarrow U$ that acts by $t \mapsto (t+2^i-2)/2^i$ on the first coordinate and by the identity on $V$.} It will therefore suffice to prove that right-hand side of \eqref{eq:operation3} is independent of the choice of $e$, in other words, that
\[
(W_1 \sqcup \cdots \sqcup W_r , e \circ (\varphi_1 \sqcup \cdots \sqcup \varphi_r)) \sim (W_1 \sqcup \cdots \sqcup W_r , e' \circ (\varphi_1 \sqcup \cdots \sqcup \varphi_r))
\]
for any two smooth, orientation-preserving embeddings $e,e'$ of $r$ disjoint copies of $U$ into $U$. By Lemma \ref{lem:equivalence} this means that we need to find $\Phi \in \mathrm{Diff}_c(U)_1$ such that:
\begin{equation}
\label{eq:condition}
\Phi(e(\varphi_1(W_1) \sqcup \cdots \sqcup \varphi_r(W_r))) = e'(\varphi_1(W_1) \sqcup \cdots \sqcup \varphi_r(W_r)).
\end{equation}
Since $\varphi_s(W_s)$ is a compact subspace of $U = \bR^m$, we may assume that it is contained in the closed unit disc $\bD^m$.\footnote{If it is not, we go back and choose a different representative for $[(W_s,\varphi_s)]$ by ``shrinking'' $\varphi_s$ by an isotopy.} Let us write $i_s$ for the restriction of $e$ to the closed unit disc in the $s$th copy of $U$ for $s \in \{1,\ldots,r\}$. Similarly, write $j_s$ for the restriction of $e'$ to the closed unit disc in the $s$th copy of $U$. We may therefore rewrite \eqref{eq:condition} as:
\begin{equation}
\label{eq:condition2}
\Phi(i_1(\varphi_1(W_1))) \cup \cdots \cup \Phi(i_r(\varphi_r(W_r))) = j_1(\varphi_1(W_1)) \cup \cdots \cup j_r(\varphi_r(W_r)).
\end{equation}
Now Proposition \ref{prop:multi-disc-theorem} tells us that there exists $\Phi \in \mathrm{Diff}_c(U)_1$ such that $\Phi \circ i_s = j_s$ for all $s \in \{1,\ldots,r\}$, which implies \eqref{eq:condition2}.
\end{proof}

\begin{corollary}
\label{coro:model}
The monoid $h\mathcal{C}_d^V(\emptyset,\emptyset) = \mathcal{M}_d^V$ is isomorphic to the monoid whose underlying set is the set of smooth, closed, $d$-dimensional submanifolds of $U$, up to the natural left-action of $\mathrm{Diff}_c(U)_1$, and whose operation is given by
\[
([W_1],[W_2]) \mapsto [e_1(W_1) \cup e_2(W_2)],
\]
where $(e_1,e_2)$ is any choice of a pair of orientation-preserving self-embeddings $U \hookrightarrow U$ with disjoint images. In particular, this operation is commutative. Moreover, if $e_1,\ldots,e_r$ is any collection of orientation-preserving self-embeddings $U \hookrightarrow U$ with pairwise disjoint images, we have
\begin{equation}
\label{eq:operation4}
[W_1] \cdots [W_r] = [e_1(W_1) \cup \cdots \cup e_r(W_r)]
\end{equation}
in this monoid.
\end{corollary}
\begin{proof}
By Lemma \ref{lem:equivalence}, there is a bijection from $\mathcal{M}_d^V$ to the set of smooth, closed, $d$-dimensional submanifolds of $U$, up to the natural left-action of $\mathrm{Diff}_c(U)_1$, given by sending $[(W,\varphi)]$ to $[\varphi(W)]$. Using the description of the operation of $\mathcal{M}_d^V$ from Lemma \ref{lem:operation}, we see that it corresponds under this bijection to the operation $([W_1],[W_2]) \mapsto [e_1(W_1) \cup e_2(W_2)]$ described above, which is therefore in particular well-defined. Equation \eqref{eq:operation4} also follows directly from Lemma \ref{lem:operation}.
\end{proof}

\begin{remark}
Note that this in particular solves Exercise 5.3.1(d), since Corollary \ref{coro:model} implies that $\mathcal{C}_d^V(\emptyset,\emptyset)$ is commutative. For the rest of this solution, we identify $\mathcal{M}_d^V$ with the monoid described in Corollary \ref{coro:model}.
\end{remark}

\subsubsection*{A generating set.}

We now define a subset of $\mathcal{M}_d^V$ that will turn out to freely generate it.

\begin{definition}
\label{defn:inseparable}
Let us call a smooth, closed, $d$-dimensional submanifold $W$ of $U$ \emph{inseparable} if $W \neq \emptyset$ and, if for any finite collection of smooth embeddings $i_1,\ldots,i_r \colon U \hookrightarrow U$ with
\[
W \subseteq \bigcup_{s=1}^r i_s(U) \qquad\text{and}\qquad i_1(U),\ldots,i_r(U) \text{ are pairwise disjoint},
\]
we have that $W \subseteq i_s(U)$ for one $s \in \{1,\ldots,r\}$.
\end{definition}

\begin{remark}
\label{rmk:separability}
Some immediate consequences of Definition \ref{defn:inseparable} are as follows:
\begin{itemize}
\item If $W$ is path-connected, then it is inseparable.
\item Examples of inseparable but non-path-connected $W$ are (for $d=1,m=2$) two nested circles in $\bR^2$ and (for $d=1,m=3$) two non-trivially linked circles in $\bR^3$.
\item If $m=\mathrm{dim}(U) > 2d+1$ then inseparable $d$-submanifolds of $U=\bR^m$ are path-connected. To see this, let us denote the path-components of $W$ by $W_1,\ldots,W_r$ and choose any collection $i_1,\ldots,i_r$ of orientation-preserving smooth embeddings $U \hookrightarrow U$ with pairwise disjoint images. Let $i \colon W \hookrightarrow U$ be the embedding that acts by $i_s$ on $W_s$. Since the space of embeddings $\mathrm{Emb}(W,U)$ is path-connected for $\mathrm{dim}(U) > 2\mathrm{dim}(W)+1$, we may choose a path from the inclusion to $i$ and use Proposition \ref{prop:isotopy-extension} to lift this and find $\Phi \in \mathrm{Diff}_c(U)_1$ such that $\Phi|_W = i$. Hence $[i(W)] = [W]$, so by Lemma \ref{lem:inseparable-well-defined} below, $W$ is inseparable if and only if $i(W)$ is inseparable. But $i(W)$ is clearly inseparable if and only if $r=1$.
\end{itemize}
\end{remark}

\begin{lemma}
\label{lem:inseparable-well-defined}
Inseparability is a well-defined property of the element $[W] \in \mathcal{M}_d^V$. In other words, if $W \sim W'$ and $W$ is inseparable, then so is $W'$.
\end{lemma}
\begin{proof}
Suppose that $i_1,\ldots,i_r \colon U \hookrightarrow U$ are smooth embeddings with pairwise disjoint images and with $W' \subseteq i_1(U) \cup \cdots \cup i_r(U)$. Since $W \sim W'$, we have (by Corollary \ref{coro:model}) $\Phi \in \mathrm{Diff}_c(U)_1$ such that $\Phi(W) = W'$. If we define $\hat{\imath}_s = \Phi^{-1} \circ i_s$ for each $s$, then $\hat{\imath}_1,\ldots,\hat{\imath}_r \colon U \hookrightarrow U$ are smooth embeddings with pairwise disjoint images and with $W \subseteq \hat{\imath}_1(U) \cup \cdots \cup \hat{\imath}_r(U)$. Since $W$ is inseparable we have $W \subseteq \hat{\imath}_s(U)$ for some $s \in \{1,\ldots,r\}$, and hence $W' \subseteq i_s(U)$.
\end{proof}

\begin{definition}
Denote the set of inseparable elements of $\mathcal{M}_d^V$ by $\mathcal{I}_d^V$.\footnote{See Lemma \ref{lem:inseparable2} for an equivalent characterization of $\mathcal{I}_d^V$, where the $r$ of Definition \ref{defn:inseparable} is replaced by $2$.}
\end{definition}


The following is a consequence of Lemma \ref{lem:inseparable-well-defined} and the case $r=1$ of Corollary~\ref{coro:model}.

\begin{corollary}
\label{coro:self-embedding}
Let $e \colon U \hookrightarrow U$ be an orientation-preserving self-embedding and let $W$ be a smooth, closed, $d$-dimensional submanifold of $U$. Then $[e(W)] = [W]$ in $\mathcal{M}_d^V$. In particular, by Lemma \ref{lem:inseparable-well-defined}, $W$ is inseparable if and only if $e(W)$ is inseparable.\qed
\end{corollary}


Write $\bN[\mathcal{I}_d^V]$ for the free commutative monoid on the set $\mathcal{I}_d^V$. The inclusion $\mathcal{I}_d^V \hookrightarrow \mathcal{M}_d^V$ therefore induces a monoid homomorphism
\begin{equation}
\label{eq:monoid-hom}
\alpha \colon  \bN[\mathcal{I}_d^V] \longrightarrow \mathcal{M}_d^V.
\end{equation}
Our aim is to show that $\alpha$ is an isomorphism. (Surjectivity of $\alpha$ is the statement that $\mathcal{I}_d^V$ \emph{generates} $\mathcal{M}_d^V$ as a commutative monoid; injectivity is the statement that it \emph{freely} generates $\mathcal{M}_d^V$.)

\subsubsection*{Surjectivity and injectivity.}

\begin{proposition}
\label{prop:surjectivity}
The monoid homomorphism \eqref{eq:monoid-hom} is surjective.
\end{proposition}
\begin{proof}
Let $[W] \in \mathcal{M}_d^V$. We show that $[W]$ is in the image of $\alpha$ by induction on the number $n(W)$ of path-components of $W$. If $W=\emptyset$ then $[W] = \alpha(\phantom{-})$. If $W$ is path-connected, then it must be inseparable, so $[W] = \alpha([W])$. So we may assume that $n(W)\geq 2$. If $[W]$ is inseparable, we are done, so let us assume that there are smooth embeddings $i_1,\ldots,i_r \colon U \hookrightarrow U$ with pairwise disjoint images ($r \geq 2$) such that
\[
W \subseteq \bigcup_{s=1}^r i_s(U) \qquad\text{and}\qquad W \cap i_s(U) \neq \emptyset \text{ for all } s \in \{1,\ldots,r\}.
\]
Without loss of generality, we may also assume that $i_1,\ldots,i_r$ are orientation-preserving. Since $\{ W \cap i_s(U) \mid s \in \{1,\ldots,r\} \}$ is a partition of $W$ into clopen subsets, each $W \cap i_s(U)$ is a union of path-components of $W$, and in particular compact. Hence 
\[
W'_s = i_s^{-1}(W)
\]
are smooth, closed, $d$-dimensional submanifolds of $U$, so they represent elements $[W'_s]$ of $\mathcal{M}_d^V$. Since $W'_s$ are all non-empty and $r \geq 2$ we deduce that $n(W'_s) < n(W)$ for all $s$. Thus, by the inductive hypothesis, each $[W'_s]$ is in the image of $\alpha$. By Corollary \ref{coro:model},
\begin{equation}
\label{eq:composition}
[W'_1] \cdots [W'_r] = [i_1(W'_1) \cup \cdots \cup i_r(W'_r)] = [W],
\end{equation}
so $[W]$ is also in the image of $\alpha$.
\end{proof}

\begin{lemma}
\label{lem:key}
Let $W_1,\ldots,W_r$ be a finite collection of inseparable manifolds embedded in $U$ and let $e_1,\ldots,e_r \colon U \hookrightarrow U$ be a collection of smooth, orientation-preserving embeddings $U \hookrightarrow U$ with pairwise disjoint images, and let $i_1,\ldots,i_s \colon U \hookrightarrow U$ be another such collection. Assume that
\[
W \coloneqq e_1(W_1) \cup \cdots \cup e_r(W_r) \subseteq i_1(U) \cup \cdots \cup i_s(U)
\]
and that $i_t^{-1}(W)$ is compact and inseparable for each $t \in \{1,\ldots,s\}$. Then $r=s$ and $i_t(U) \cap W = e_{\sigma(t)}(W_{\sigma(t)})$ for each $t \in \{1,\ldots,r\}$ and some permutation $\sigma$ of $\{1,\ldots,r\}$.
\end{lemma}
\begin{proof}
We will first prove that, for each $t \in \{1,\ldots,s\}$,
\begin{equation}
\label{eq:step1}
i_t(U) \cap W = \bigcup_{a \in A_t} e_a(W_a), \text{ for some } A_t \subseteq \{1,\ldots,r\}.
\end{equation}
Given this, one may easily see that $\{A_1,\ldots,A_s\}$ is a partition of $\{1,\ldots,r\}$ into non-empty subsets. We will then prove that, for each $t \in \{1,\ldots,s\}$,
\begin{equation}
\label{eq:step2}
\lvert A_t \rvert = 1,
\end{equation}
which will complete the proof.

Suppose for a contradiction that \eqref{eq:step1} is false. This implies that, for some $u \in \{1,\ldots,r\}$, the submanifold $e_u(W_u)$ is contained in the union $i_1(U) \cup \cdots \cup i_s(U)$, but it is not contained in any one of the $i_1(U),\ldots,i_s(U)$. Hence $e_u(W_u)$ is separable, and so, by Corollary \ref{coro:self-embedding}, $W_u$ is separable, which contradicts the hypotheses.

Suppose for a contradiction that \eqref{eq:step2} is false. This implies that, for some $t \in \{1,\ldots,s\}$, the submanifold $i_t(U) \cap W$ is contained in the union $e_1(U) \cup \cdots \cup e_r(U)$, but it is not contained in any one of the $e_1(U),\ldots,e_r(U)$. Hence $i_t(U) \cap W = i_t(i_t^{-1}(W))$ is separable, and so, by Corollary \ref{coro:self-embedding}, $i_t^{-1}(W)$ is separable, which contradicts the hypotheses.
\end{proof}

\begin{proposition}
\label{prop:injectivity}
The monoid homomorphism \eqref{eq:monoid-hom} is injective.
\end{proposition}
\begin{proof}
Let $W_1,\ldots,W_r$ and $X_1,\ldots,X_s$ be inseparable manifolds embedded in $U$. Suppose that $[W_1] \cdots [W_r] = [X_1] \cdots [X_s]$ in $\mathcal{M}_d^V$. We then need to show that $r=s$ and $[X_t] = [W_{\sigma(t)}]$ for all $t \in \{1,\ldots,r\}$ and some permutation $\sigma$ of $\{1,\ldots,r\}$.

Let $e_1,\ldots,e_{\mathrm{max}(r,s)} \colon U \hookrightarrow U$ be any collection of orientation-preserving embeddings with pairwise disjoint images. Then, by Corollary \ref{coro:model}, we have some $\Phi \in \mathrm{Diff}_c(U)_1$ such that
\[
\Phi(e_1(X_1) \cup \cdots \cup e_s(X_s)) = e_1(W_1) \cup \cdots \cup e_r(W_r) \eqqcolon W.
\]
For each $t \in \{1,\ldots,s\}$, we have $(\Phi \circ e_t)^{-1}(W) = X_t$, which is a compact and inseparable submanifold of $U$, so Lemma \ref{lem:key} implies that $r=s$ and
\[
\Phi(e_t(U)) \cap W = e_{\sigma(t)}(W_{\sigma(t)})
\]
for each $t \in \{1,\ldots,r\}$ and some permutation $\sigma$ of $\{1,\ldots,r\}$. But $\Phi(e_t(U)) \cap W = \Phi(e_t(X_t))$, so by Corollaries \ref{coro:model} and \ref{coro:self-embedding} we have
\[
[X_t] = [e_t(X_t)] = [e_{\sigma(t)}(W_{\sigma(t)})] = [W_{\sigma(t)}].\qedhere
\]
\end{proof}

Propositions \ref{prop:surjectivity} and \ref{prop:injectivity} immediately imply:

\begin{corollary}
\label{coro:freely-generating}
The commutative monoid $\mathcal{M}_d^V$ is freely generated by the subset $\mathcal{I}_d^V$.
\end{corollary}

We note that the subset $\mathcal{I}_d^V \subset \mathcal{M}_d^V$ of equivalence classes (with respect to the action of $\mathrm{Diff}_c(U)_1$) of \emph{inseparable} submanifolds of $U$ may be characterised in a slightly simpler way than in Definition \ref{defn:inseparable}, namely: a submanifold $W$ may be separated by $r \geq 2$ pairwise disjoint embeddings $U \hookrightarrow U$ if and only if it may be separated by two disjoint embeddings $U \hookrightarrow U$.

\begin{lemma}
\label{lem:inseparable2}
An element $[W] \in \mathcal{M}_d^V$ is inseparable if and only if $W \neq \emptyset$ and, given any pair of smooth embeddings $i,j \colon U \hookrightarrow U$ with
\[
W \subseteq i(U) \cup j(U) \qquad\text{and}\qquad i(U) \cap j(U) = \emptyset,
\]
then $W \subseteq i(U)$ or $W \subseteq j(U)$.
\end{lemma}
\begin{proof}
The ``only if'' direction is immediate from Definition \ref{defn:inseparable}, setting $r=2$. For the ``if'' direction, let us assume that $W$ is separable and non-empty: we need to find a pair of smooth embeddings $i,j \colon U \hookrightarrow U$ with disjoint images, such that $W \subseteq i(U) \cup j(U)$ and $i(U) \cap W \neq \emptyset \neq j(U) \cap W$. Since $W$ is separable, we have a finite collection $i_1,\ldots,i_r \colon U \hookrightarrow U$ of smooth embeddings ($r \geq 2$) with pairwise disjoint images, such that
\[
W \subseteq \bigcup_{s=1}^r i_s(U) \qquad\text{and}\qquad i_s(U) \cap W \neq \emptyset \text{ for all } s \in \{1,\ldots,r\}.
\]
Without loss of generality we may assume that each $i_s$ is orientation-preserving. Moreover, each $i_s(U) \cap W = i_s(i_s^{-1}(W))$ is a closed subset of $W$, and hence compact, so $i_s^{-1}(W)$ is also compact. Thus we may also assume, without loss of generality, that $i_s^{-1}(W)$ is contained in the closed unit disc $\bD^m \subset \bR^m = U$. Write $j_s = i_s|_{\bD^m} \colon \bD^m \hookrightarrow U$ and choose another collection $k_1,\ldots,k_r \colon \bD^m \hookrightarrow U$ of orientation-preserving embeddings with pairwise disjoint images, such that
\[
k_1(\bD^m) \subseteq \mathring{\bD}^m(0,\ldots,0) \qquad\text{and}\qquad k_2(\bD^m),\ldots,k_r(\bD^m) \subseteq \mathring{\bD}^m(2,0,\ldots,0),
\]
where, for $x \in \bR^m$, $\mathring{\bD}^m(x)$ is the open unit disc in $\bR^m$ centred at $x$. Let $\hat{\imath},\hat{\jmath} \colon U \hookrightarrow U$ be embeddings with $\hat{\imath}(U) = \mathring{\bD}^m(0,\ldots,0)$ and $\hat{\jmath}(U) = \mathring{\bD}^m(2,0,\ldots,0)$. By Proposition \ref{prop:multi-disc-theorem}, we have $\Phi \in \mathrm{Diff}_c(U)_1$ such that $\Phi \circ j_s = k_s$ for all $s \in \{1,\ldots,r\}$. It follows that $\Phi(W)$ is contained in $\hat{\imath}(U) \cup \hat{\jmath}(U)$, but $\Phi(W) \not\subseteq \hat{\imath}(U)$ and $\Phi(W) \not\subseteq \hat{\jmath}(U)$. Now defining $i = \Phi^{-1} \circ \hat{\imath}$ and $j = \Phi^{-1} \circ \hat{\jmath}$, we obtain the required smooth embeddings $i,j \colon U \hookrightarrow U$.
\end{proof}

\subsubsection*{Tangential structures and $E_n$ algebras.}

\begin{remark}
[Tangential structures]
\label{rmk:tangential}
If we fix a tangential structure $\Theta$ and consider the topological cobordism category $\mathcal{C}_{\Theta}^V$, we may ask whether the monoid $\mathcal{M}_{\Theta}^V = h\mathcal{C}_{\Theta}^V(\emptyset,\emptyset)$ is also free commutative when $\mathrm{dim}(V)>0$. It turns out that it is, with a very similar proof to above: the result is that $\mathcal{M}_{\Theta}^V$ is \emph{free commutative} on the subset $\mathcal{I}_{\Theta}^V$ of those $\Theta$-manifolds whose underlying manifold (ignoring the $\Theta$-structure) is inseparable.
\end{remark}

\begin{remark}
[Dimension zero]
\label{rmk:dimension-zero}
When $\mathrm{dim}(V)=0$ (and $d=0$), the monoid $\mathcal{M}_{\Theta}^V$ is not commutative in general. Note that, in dimension zero, a tangential structure is just a topological space $\Theta$, and a $\Theta$-structure on a zero-dimensional manifold (i.e., discrete space) $W$ is just a function $W \to \Theta$. One may then easily see that, when $d=0$ and $\mathrm{dim}(V)=0$ (so $U \cong \bR$), the monoid $\mathcal{M}_{\Theta}^V$ is isomorphic to the \emph{free non-commutative monoid} on $\pi_0(\Theta)$.\footnote{If $\Theta$ is path-connected, this is also free commutative, but only ``by accident''.} Since the only inseparable zero-manifolds are single points (equipped with any $\Theta$-structure), we may identify the set $\mathcal{I}_{\Theta}^V$ of equivalence classes of inseparable $\Theta$-manifolds with $\pi_0(\Theta)$, and say that $\mathcal{M}_{\Theta}^V$ is the free non-commutative monoid on $\mathcal{I}_{\Theta}^V$.
\end{remark}

\begin{remark}
[$E_m$-algebras]
\label{rmk:Em-algebras}
The topological monoid $\mathcal{C}_{\Theta}^V(\emptyset,\emptyset)$ is homotopy equivalent to an $E_m$-algebra, where $m = \mathrm{dim}(V) + 1$. Defining $\mathcal{J}_{\Theta}^V \subset \mathcal{C}_{\Theta}^V(\emptyset,\emptyset)$ to be the subspace of \emph{inseparable} submanifolds, the inclusion induces a map of $E_m$-algebras
\begin{equation}
\label{eq:map-of-Em-algebras}
F_{E_m}(\mathcal{J}_{\Theta}^V) \longrightarrow \mathcal{C}_{\Theta}^V(\emptyset,\emptyset),
\end{equation}
where $F_{E_m}(X)$ denotes the free $E_m$-algebra on a space $X$. The result of this exercise (plus the previous two remarks) is that $\pi_0 (\ref{eq:map-of-Em-algebras}) = (\ref{eq:monoid-hom})$ is an isomorphism (of commutative monoids if $m\geq 2$ and of non-commutative monoids if $m=1$). One may naturally ask whether \eqref{eq:map-of-Em-algebras} induces isomorphisms also on higher homotopy groups. The answer is no in general.

As one counterexample, take $d=1$, $m=3$ and $\Theta = \{*\}$. Then the left-hand side of \eqref{eq:map-of-Em-algebras} is homotopy equivalent to the configuration space of finitely many unordered points in $\bR^3$ each labeled by a point in the space of inseparable links in $\bR^3$. One path-component $X$ of this space consists of configuration of two points in $\bR^3$ each labeled by a point in the space of unknots in $\bR^3$. Using \cite{Hatcher-Smale-conjecture} to see that the space of unknots in $\bR^3$ is homotopy equivalent to $\bR\bP^2$, we see that $H_1(X) \cong (\bZ/2\bZ)^2$. The corresponding\footnote{We have shown that \eqref{eq:map-of-Em-algebras} is a bijection on $\pi_0$, so ``corresponding'' makes sense.} path-component $Y$ of the right-hand side of \eqref{eq:map-of-Em-algebras} is the space of two-component unlinks in $\bR^3$. Using \cite[Theorem 1 and Proposition 3.7]{BrendleHatcher}, we see that $H_1(Y) \cong (\bZ/2\bZ)^3$. Hence \eqref{eq:map-of-Em-algebras} is not surjective on $H_1$, and therefore also not on $\pi_1$.

Another counterexample is given by $d=1$, $m=3$ and $\Theta = \{\pm 1\}$ the tangential structure of orientations. Then one path-component $X$ of the left-hand side of \eqref{eq:map-of-Em-algebras} is homotopy equivalent to the space of configurations of two points in $\bR^3$ each labeled by a point in the space of \emph{oriented} unknots in $\bR^3$. Since the space of oriented unknots in $\bR^3$ is homotopy equivalent to $S^2$, we have $H_1(X) \cong \bZ/2\bZ$. The corresponding path-component $Y$ of the right-hand side of \eqref{eq:map-of-Em-algebras} is the space of two-component \emph{oriented} unlinks in $\bR^3$. Using \cite[Theorem 1 and Proposition 3.3]{BrendleHatcher}, we see that $H_1(Y) \cong \bZ \oplus \bZ/2\bZ$. So, again, \eqref{eq:map-of-Em-algebras} is not surjective on $H_1$, and therefore also not on $\pi_1$.
\end{remark}



%
%
%
%
%
%
%

\bibspread

\bibliography{biblio}

\medskip

\end{document}